\documentclass[11pt]{amsart}

\usepackage[margin=3cm]{geometry}
\usepackage[foot]{amsaddr}
\usepackage{amsmath}
\usepackage{amssymb}
\usepackage{amsthm}
\usepackage{graphicx}
\usepackage{float}
\usepackage{color}
\usepackage{ulem}
\usepackage{dsfont}
\usepackage{epstopdf}
\usepackage{hyperref}
\usepackage{appendix}
\usepackage{comment}
\usepackage{threeparttable}
\allowdisplaybreaks[4]
\numberwithin{equation}{section}
\newtheorem{theorem}{Theorem}[section]
\newtheorem{corollary}[theorem]{Corollary}

\newtheorem{lemma}[theorem]{Lemma}
\newtheorem{proposition}[theorem]{Proposition}

\theoremstyle{definition}
\newtheorem{definition}[theorem]{Definition}
\newtheorem{remark}[theorem]{Remark}

\newtheorem*{thma}{Theorem A}
\newtheorem*{thmb}{Theorem B}
\newtheorem*{thmC}{Theorem C}

\newcommand{\ep}{\varepsilon}

\newcommand{\lam}[1]{\textcolor{blue}{#1}}

\title[Stacked invasion waves 
in a competition-diffusion model]
      {Stacked invasion waves in a competition-diffusion model with three species}

\thanks{{K.-Y. Lam}: Department of Mathematics, Ohio State University, Columbus, OH 43210, USA}

\thanks{{Q. Liu, S. Liu}: Institute for Mathematical Sciences, Renmin University of China, Beijing 100872, China.}

\thanks{Email: lam.184@math.osu.edu (K.-Y. Lam), liuqian0519@ruc.edu.cn (Q. Liu),\\
 liushuangnqkg@ruc.edu.cn (S. Liu)}
\author[King-Yeung Lam et al.]
{King-Yeung Lam, \ \ Qian Liu, \ \ Shuang Liu}

\begin{document}
\maketitle
\begin{abstract}
We investigate the spreading properties of a  three-species competition-diffusion system, which is non-cooperative.
We apply the Hamilton-Jacobi approach, due to Freidlin, Evans and Souganidis,
to establish upper and lower estimates of spreading speed for the slowest species, in terms of the spreading speeds of two faster species. 
The estimates we obtained are sharp in some situations.
The spreading speed is being characterized as the free boundary point of the viscosity solution for certain variational inequality cast in the space of speeds. 
To the best of our knowledge, this is the first theoretical result on three-species competition system in unbounded domains.
\end{abstract}




\section{Introduction}\label{S1}

Biological invasion (or spreading) is a fundamental and long-standing subject in ecology \cite{Shigesada_1997}. Mathematical studies have so far been focused on the single-species and two-species models, where the order-preserving property of the underlying dynamics was exploited to identify the speeds of the invasive species.
In this paper, we consider the diffusive Lotka-Volterra system consisting of three competing species,
which is not order-preserving. After suitable non-dimensionalization, the system reads
\begin{equation}\label{eq:1-1}
\left\{
\begin{array}{cc}
\partial_t u_1-d_1\partial_{xx}u_1=r_1u_1(1-u_1-a_{12}u_2 - a_{13}u_3)& \text{ in }(0,\infty)\times \mathbb{R},\\
\partial_t u_2-\partial_{xx}u_2=u_2(1-a_{21} u_1- u_2 - a_{23}u_3)& \text{ in }(0,\infty)\times \mathbb{R},\\
\partial_t u_3-d_3\partial_{xx}u_3=r_3u_3(1-a_{31} u_1- a_{32} u_2 -u_3) & \text{ in }(0,\infty)\times \mathbb{R},\\
u_i(0,x)=u_{i,0}(x) & \text{ on } \mathbb{R},\,i=1,2,3,
\end{array}
\right .
\end{equation}
where 
$u_i(t,x)$ represents the population density of the $i$-th competing species at time $t$ and location $x$. The positive constants $d_i$ and $r_i$ denote the diffusion coefficient and intrinsic growth rate of $u_i$ (we may assume $d_2=r_2=1$ by scaling the variables $x$ and $t$), and positive constant $a_{ij}$ is the competition coefficient of species $u_j$ to $u_i$.  
We will determine a class of solutions where
each competing species invades from left to right with a different speed; 
see Figure \ref{figure3}. A necessary condition is the competitive hierarchy:
\begin{equation}\label{eq:compactorder1}
d_3r_3 < 1 < d_1r_1, \quad a_{21} < 1 < a_{12},\quad \text{ and }\quad a_{31}+a_{32}<1,
\end{equation}
which says that the species $u_1,u_2,u_3$ are ordered from fastest to slowest, and that  
$u_2$ can competitively exclude $u_1$ in the absence of $u_3$, but both will eventually be invaded by $u_3$. We will assume \eqref{eq:compactorder1} holds throughout this paper.

%
%
%
\subsection{Spreading Speeds of the First Two Species}

When $u_3 \equiv 0$, system \eqref{eq:1-1} reduces to the two-species competition system
\begin{equation}\label{eq:2sp}
\left\{
\begin{array}{cc}
\partial_t u_1-d_1\partial_{xx}u_1=r_1u_1(1-u_1-a_{12}u_2)& \text{ in }(0,\infty)\times \mathbb{R},\\
\partial_t u_2-\partial_{xx}u_2=u_2(1-a_{21} u_1- u_2)& \text{ in }(0,\infty)\times \mathbb{R},\\
u_i(0,x) = u_{i,0}(x) & \text{ on }\mathbb{R},\, i=1,2.
\end{array}
\right .
\end{equation}
When $a_{21} < 1 < a_{12}$ (i.e. the second species is competitively superior to the first one), and that $u_{1,0}$ and $1-u_{2,0}$ are both nonnegative, compactly supported and bounded from above by $1$, a classical spreading result due to Li et al. \cite{LWL_2005} says that there exists $\hat{c}_{\rm LLW} \in [2\sqrt{1-a_{21}},2]$ such that $u_2$ invades into the territory of $u_1$ with speed $\hat{c}_{\rm LLW}$ in the following sense:
\begin{equation}\label{qq:1.3}
    \begin{cases}
    \lim\limits_{t\to\infty} \sup\limits_{x > (\hat{c}_{\rm LLW} + \eta)t} (|u_1(t,x)-1| + |u_2(t,x)|)=0,\\
    \lim\limits_{t\to\infty} \sup\limits_{0\leq x < (\hat{c}_{\rm LLW} - \eta)t} (|u_1(t,x)| + |u_2(t,x)-1|)=0.
    \end{cases}
\end{equation}
Furthermore, $\hat{c}_{\rm LLW}$ coincides with the minimal wave speed for the existence of a traveling wave solution of \eqref{eq:2sp} connecting $(1,0)$ and $(0,1)$. A linearization of \eqref{eq:2sp} at the equilibrium $(1,0)$ that is being invaded, shows that $\hat{c}_{\rm LLW} \geq 2\sqrt{1-a_{21}}$.

When $\hat{c}_{\rm LLW} = 2\sqrt{1-a_{21}}$, we say that the spreading speed $\hat{c}_{\rm LLW}$ is linearly determined. In this case, the resulting invasion wave is a pulled wave, in the sense that the invading population is fueled by the growth of population at the leading edge of the front.
When 
$\hat{c}_{\rm LLW} > 2\sqrt{1-a_{21}}$, we say that the spreading speed $\hat{c}_{\rm LLW}$ is nonlinearly determined. In this case, the resulting invasion wave is a pushed wave, in the sense that the expansion is pushed by all components of the population. Thus a pushed wave is a mechanism to speed up the invasion of an compactly supported population. A signature of a pushed wave is its fast exponential decay at $x=\infty$ \cite{Alhasanat2019,Roques2015}.


Yet another mechanism of speed enhancement takes effect when the two species are invading an open habitat. Namely, when both $u_{1,0}$ and $u_{2,0}$ are compactly supported. 
This question was raised by Shigesada and Kawasaki \cite{Shigesada_1997} as they considered the invasion of two or more tree species into the North American continent at the end of the last ice age (approximately 16,000 years ago) \cite{Davis1981}. The case of two competing species
was first considered by Lin and Li \cite{Lin2012}, and is completely solved in \cite{Girardin_2019} for compactly supported initial data via a delicate construction of super- and sub-solutions. See also \cite{Salako2020} for the existence of entire solutions which are stacked waves as $t\to\infty$.
Specializing to the two-species system, \eqref{eq:compactorder1} becomes
\begin{equation}\label{qq:1.5}
    d_1r_1 >1 \quad \text{  and } \quad a_{21} < 1 < a_{12}.
\end{equation}
The first condition says that, in the absence of competition, $u_1$ spreads faster than $u_2$. The second condition says that $u_2$ is competitively superior to $u_1$.
\begin{theorem}[\cite{Girardin_2019}]
Assume \eqref{qq:1.5}.
Let $(u_i)_{i=1}^2$ be any solution of \eqref{eq:2sp} 
with compactly supported initial data which are non-negative and non-trivial. Then for each small $\eta>0$, the following spreading results hold:
\begin{equation*}
    \begin{cases}
    \lim\limits_{t\to\infty} \sup\limits_{x > (c_1 + \eta)t} (|u_1(t,x)| + |u_2(t,x)|)=0,\\
    \lim\limits_{t\to\infty} \sup\limits_{(c_2 + \eta)t<x < (c_1 - \eta)t} (|u_1(t,x)-1| + |u_2(t,x)|)=0,\\
    \lim\limits_{t\to\infty} \sup\limits_{0\leq x < (c_2 - \eta)t} (|u_1(t,x)| + |u_2(t,x)-1|)=0.
    \end{cases}
\end{equation*}
Here the spreading speeds are given by 
$c_1 = 2\sqrt{d_1r_1}$ and 
$$
c_2 = \begin{cases}
\max\left\{\hat{c}_{\rm LLW}, \frac{c_1}{2} - \sqrt{a_{21}} + \frac{1-a_{21}}{c_1/2 - \sqrt{a_{21}}} \right\} &\text{ if }c_1 < 2(\sqrt{a_{21}} + \sqrt{1-a_{21}}),\\
\hat{c}_{\rm LLW} &\text{otherwise},
\end{cases}
$$
where $\hat{c}_{\rm LLW} \in [2\sqrt{1-a_{21}},2]$ is the spreading speed given by \eqref{qq:1.3}.
\end{theorem}
Observe that $c_1>2$, by \eqref{qq:1.5}. When $c_1=2\sqrt{d_1r_1}  \searrow 2$ (e.g. by varying $r_1$), the speed of the second species approaches $2$, which is larger than $\hat{c}_{\rm LLW}$. This novel mechanism of speed enhancement was first discovered by Holzer and Scheel \cite{Holzer2014} when $a_{12} =0$ (in such case \eqref{eq:2sp} is decoupled). In \cite{Girardin_2019}, it is called a ``nonlocally pulled wave": It is ``nonlocal" since $c_2$ depends on $c_1$, and it is considered a kind of pulled wave since it is of slow decay (see \cite{LLL2019} for further discussion). 
The weak competition case (i.e. $0 < a_{12}, a_{21} <1$) was subsequently  considered in \cite{LLL2019,LLL20192} via obtaining large deviations type estimates. 
An important observation is that the faster moving front can influence the slower moving front, but not vice versa. This enables us to estimate each invasion front separately, from the fastest to the slowest.  
In contrast to \cite{Girardin_2019} where all the speeds are determined at once by a single pair of global super- and sub-solutions, this new point of view 
opens the door to analyzing more general non-cooperative systems.



In this paper, we are interested in the spreading dynamics of the three-species competition system \eqref{eq:1-1}, with initial data satisfies one of the following conditions. 
\begin{itemize}
\item [{$\rm{(H_{\infty})}$}] For $i=1,2,3$, $u_{i,0}\in C(\mathbb{R}; [0,1])$  is non-trivial and  has compact support.
  \item [{$\rm(H_{\lambda})$}] For $i=1,2$, $u_{i,0}\in C(\mathbb{R}; [0,1])$  is non-trivial and  has compact support, and the initial data $u_{3,0}\in C(\mathbb{R}; [0,1])$   satisfies $u_{3,0}(x)>0$ for all $x \in \mathbb{R}$, and 
  $$0 < \liminf\limits_{x \to \infty} e^{\lambda x} u_{3,0}(x) \leq \limsup\limits_{x \to \infty} e^{\lambda x} u_{3,0}(x) <\infty  \text{ for some }\lambda \in (0,\infty).$$
  \end{itemize}
  
  To facilitate our discussion, we introduce the maximal and minimal spreading speeds for each of $u_i$ as follows  (see, e.g. \cite[Definition 1.2]{Hamel2012}, for  related concepts for a single species):
\begin{equation*}
\begin{cases}
\smallskip
\overline{c}_i=\inf{\{c>0~|~\limsup \limits_{t\rightarrow \infty}\sup\limits_{x>ct} u_i(t,x)=0\}},\\
\smallskip
\underline{c}_i=\sup{\{c>0~~|\liminf \limits_{t\rightarrow \infty}\inf\limits_{ct-1<x<ct} u_i(t,x)>0\}}, \\
 \end{cases}
\text{ for }i=1,2,3.
\end{equation*}
Note that $\overline{c}_i \geq \underline{c}_i$.  Furthermore, the species $u_i$ has a spreading speed $c$ in the sense of \cite{Aronson_1975, Aronson_1978} if and only if 
$c=\overline{c}_i = \underline{c}_i$. Different from the spreading speed, these maximal and minimal speeds are well defined {\it a priori}, and are more amenable for estimation.
%

Let us assume, without loss of generality, that $u_3$ is the slowest species. By the observation that the slower front does not affect the faster fronts, the spreading speeds of the two faster species can be determined based on
%
%
\cite{Girardin_2019,LLL2019,LLL20192}. 
To state the theorem, we  
define the nonlocally pulled wave speed:
\begin{equation}\label{qq:1.7a}
\hat{s}_{\rm nlp}(c_1) := 
\begin{cases}
\frac{c_1}{2} - \sqrt{a_{21}} + \frac{1-a_{21}}{\frac{c_1}{2} - \sqrt{a_{21}}} & \text{ if }c_1\leq 2 (\sqrt {a_{21}} + \sqrt{1-a_{21}}),\\
2\sqrt{1-a_{21}} &\text{ otherwise.}
\end{cases}
\end{equation}
(Note that $\hat{s}_{\rm nlp} \in [2\sqrt{1-a_{21}},2]$.)
\begin{theorem}\label{thm:suff1.2}
Assume that \eqref{eq:compactorder1} holds and that $\hat{c}_{\rm LLW}=2\sqrt{1-a_{21}}$ (i.e. $\hat{c}_{\rm LLW}$ is linearly determined). Let $(u_i)_{i=1}^3$ be a solution to \eqref{eq:1-1}, such that one of the following conditions hold:
\begin{itemize}
    \item[{\rm(i)}] $({\rm H}_\infty)$ holds and $2\sqrt{d_3r_3} < \hat{s}_{\rm nlp}(2\sqrt{d_1r_1})$; or 
    \item[{\rm(ii)}] $({\rm H}_\lambda)$ holds for some $\lambda\in(0,\infty)$, and $\sigma_3(\lambda) < \hat{s}_{\rm nlp}(2\sqrt{d_1r_1})$, where 
    \begin{equation}\label{eq:sigma3}
    \sigma_3(\lambda):= \begin{cases}
    \smallskip
    d_3 \lambda + \frac{r_3}{\lambda} &\text{if }\,0 < \lambda < \sqrt{r_3/d_3},\\
    2\sqrt{d_3 r_3} &\text{if } \,\lambda \geq \sqrt{r_3/d_3}.
    \end{cases}
    \end{equation}
\end{itemize}
Then, letting $c_1=2\sqrt{d_1r_1}$ and $c_2 = \hat{s}_{\rm nlp}(2\sqrt{d_1r_1})$, we have $c_1 > c_2 > \sigma_3(\lambda) \ge\overline{c}_3$. Furthermore, 
the spreading dynamics of the first two species satisfy, for each $\eta>0$ small,
 \begin{equation}\label{eq:H.c1c2}
\begin{cases}
\lim\limits_{t\rightarrow \infty} \sup\limits_{ x>(c_1+\eta) t} (|u_1(t,x)|+|u_2(t,x)|)=0, \\
\lim\limits_{t\rightarrow \infty} \sup\limits_{(c_2+\eta) t< x<(c_1-\eta) t} (|u_1(t,x)-1|+|u_2(t,x)|)=0, \\
\lim\limits_{t\rightarrow \infty} \sup\limits_{( \overline{c}_{3} + \eta)t< x<(c_2-\eta) t} (|u_1(t,x)|+|u_2(t,x)-1|)=0,\\ 
\lim\limits_{t\rightarrow \infty} \sup\limits_{x > (\overline{c}_3 + \eta)t} |u_3(x,t)| = 0.
\end{cases}
\end{equation}
\end{theorem}
The proof of Theorem \ref{thm:suff1.2} is a direct application of \cite[Theorem 7.1]{LLL20192} and is thus omitted.
Therefore, we can reduce the problem into determining the speed of the slowest species. 
\begin{remark}\label{rmk:1--2a}
By \cite[Theorem 2.1]{Lewis_2002}, 
a sufficient condition for $\hat{c}_{\rm LLW} = 2\sqrt{1-a_{21}}$ is $d_1 =1$,  $a_{21}< 1 < a_{12}$, and $a_{21}a_{12} <\max\{1,2(1-a_{21})\}$. 
(See also \cite{Alhasanat2019,Huang_2010}.) However, the linear determinacy assumption was added only for simplicity purpose. In fact, it is possible to remove the assumption, by replacing $c_2=\max\{\hat{s}_{\rm nlp}, \hat{c}_{\rm LLW}\}$ in the conclusions. 
\end{remark}

\begin{definition}\label{def:h3}
Given $c_1,c_2\in (0,\infty)$ and $\lambda \in (0,\infty]$. We say that $({\rm H}_{c_1,c_2,\lambda})$ holds if \begin{itemize} \smallskip
    \item [{\rm (i)}] $c_1>c_2 > \sigma_3(\lambda)$,
    \smallskip
    \item [{\rm (ii)}] the solution $(u_i)_{i=1}^3$ of \eqref{eq:1-1} has initial condition satisfying $({\rm H}_\lambda)$, and
    \smallskip
    \item [{\rm (iii)}] the spreading conditions \eqref{eq:H.c1c2} hold.
\end{itemize}   
\end{definition}
The conclusion of Theorem \ref{thm:suff1.2} can be rephrased as $({\rm H}_{c_1,c_2,\lambda})$ being satisfied for 
$$c_1=2\sqrt{d_1r_1}, \quad c_2 = \hat{s}_{\rm nlp}(c_1), \quad \text{ and  for some }\quad  \lambda \in (0,\infty].
$$

%

Note that $\sigma_3(\lambda)$, defined in \eqref{eq:sigma3}, is an upper bound of the spreading speed of $u_3$ when it has exponential decay $\lambda$. Thus 
{\rm (i)}
means that the three species are ordered from the fastest to the slowest. 

\subsection{The Spreading Speed of the Third Species}

Hereafter we will work under the assumption that $({\rm H}_{c_1,c_2,\lambda})$ holds for some $c_1,c_2,\lambda$, and proceed to prove
%
upper and lower bounds of the spreading speed $c_3$ of the third species in terms of  spreading speeds of the first two species. Furthermore, we will show that these estimates are sharp in case the invasion wave of $u_3$ is nonlocally pulled.

To this end, we introduce the speed $s_{\rm nlp} = s_{\rm nlp}(c_1,c_2,\lambda)$ as a free boundary point of the 
viscosity solution of a variational inequality. 
\begin{definition}\label{def:snlp}
 For given  $c_1>c_2>0$ and $\lambda\in(0,\infty]$, let $\rho_{\rm nlp}:[0,\infty)\rightarrow[0,\infty)$ be the unique viscosity solution of the following variational inequality:
\begin{align}\label{eq:w3nlp}
\left\{\begin{array}{l}
\smallskip
\min\{\rho-s\rho'+d_3|\rho'|^2+\mathcal{R}(s),\rho\}=0 \,\,\text{ in } (0,\infty),\\
\rho(0)=0,\quad\lim\limits_{s\to\infty}\frac{\rho(s)}{s}=\lambda,
\end{array}
\right.
\end{align}
where $\mathcal{R}(s)=r_3(1-a_{31}\chi_{\{c_2\leq s\leq c_1\}}-a_{32}\chi_{\{s\leq c_2\}})$ and $\chi_S$ is the indicator function of the set $S$. 
 See Definition \ref{defviscosity} for the definitions of viscosity solutions. 
We define the speed $s_{{\rm nlp}}=s_{{\rm nlp}}(c_1,c_2,\lambda)
$ 
 as the free boundary point given by
\begin{equation}\label{charact_snlp}
 s_{{\rm nlp}}
 = \sup \{s:\rho_{{\rm nlp}}(s) = 0\}.
\end{equation}
\end{definition}

\begin{remark}
The quantity $s_{\rm nlp}$ is well-defined since $\rho_{\rm nlp}(s)$ is non-negative and non-decreasing in $s$ (see Lemma \ref{lemma_w3nlp}). 
In the special case when $a_{31}=a_{32}=0$ so that species $u_3$ can spread as a single species, it is not difficult to see that 
$$
\rho_{\rm nlp}(s) = \max\left\{ \frac{s^2}{4d_3} - r_3,0\right\} \quad \text{ and }\quad s_{\rm nlp} = 2\sqrt{d_3 r_3}
$$
when $u_{3,0}$ is compactly supported, i.e. $\lambda = \infty$; and that
$$
\rho_{\rm nlp}(s) = \max\left\{ \lambda\left(s - d_3 \lambda - \frac{r_3}{\lambda}  \right),0\right\} \quad \text{ and }\quad s_{\rm nlp} = d_3 \lambda + \frac{r_3}{\lambda}
$$
when $\lambda \in [0,\sqrt{r_3/d_3})$. This recovers the classical Fisher-KPP (locally pulled) wave speed for the single species with compactly supported initial data, and the result of \cite{Uchiyama1978} when the exponential decay rate $\lambda$ is subcritical.
\end{remark}

\begin{remark}\label{rmk:1.3}
The following results will be proved in Lemma \ref{lem:betamu3} and Proposition \ref{charactsnlpc}.
\begin{itemize}
    \item [{\rm(i)}] $2\sqrt{d_3 r_3(1-a_{32})} \leq s_{\rm nlp} \leq \sigma_3(\lambda)$, where $\sigma_3(\lambda)$ is defined in \eqref{eq:sigma3}. 
    
    \smallskip
    \item [{\rm(ii)}] 
    If $a_{31} < a_{32}$ and $2\sqrt{d_3 r_3} < c_2 < c_1 < 2\sqrt{d_3 r_3}(\sqrt{a_{32}} + \sqrt{1-a_{32}})$, then 
    $$s_{\rm nlp} 
    > 2\sqrt{d_3 r_3(1-a_{32})}.$$
\end{itemize}
\end{remark}

We also introduce the speed $c_{\rm LLW}$, which is due to Kan-on \cite{Kan-on1997}.
\begin{definition}\label{c3LLW}
Let $c_{\rm LLW}$ be the minimal speed of traveling wave solutions (i.e. $(u,v)=(\varphi(x-ct),\psi(x-ct)$) of 
\begin{equation}\label{eq:1-2'}
\left\{
\begin{array}{ll}
\partial_t   u-\partial_{xx} u=  u(1-a_{21}-  u-a_{23}v)& \text{ in }(0,\infty)\times \mathbb{R},\\
\partial _t   v-d_3\partial_{xx}  v=r_3  v(1-a_{32}  u-  v)& \text{ in }(0,\infty)\times \mathbb{R},\\
\end{array}
\right .
\end{equation}
such that $\lim\limits_{\xi\rightarrow \infty} (\varphi, \psi)(\xi) = (1-a_{21},0)$ and $\lim\limits_{\xi\rightarrow -\infty} (\varphi, \psi)(\xi) = (u^*,v^*)$,
where $(u^*,v^*)$ is the unique stable constant equilibrium such that $v^*>0$.
 \end{definition}
\begin{remark}\label{rem1.3}
It is well-known that $c_{\rm LLW} \in [2\sqrt{d_3 r_3(1-a_{32}(1-a_{21}))}, 2\sqrt{d_3r_3}]$.
Furthermore, $c_{\rm LLW} = 2\sqrt{d_3r_3(1-a_{32}(1-a_{21}))}$ if (see Lemma \ref{lem:1estamoverc30} and \cite[Theorem 2.1]{Lewis_2002})
\begin{equation*}\label{qq:1.12}
d_3 \geq \frac{1}{2}, \quad   a_{32}(1-a_{21}) < 1 < \frac{a_{23}}{1-a_{21}}, \quad \text{ and }\quad a_{32}a_{23} <1.
\end{equation*}
\end{remark}


We can now state the main theorem of this paper.

\begin{thma}\label{thm:1-1}
 \textit{Assume the hierarchy condition \eqref{eq:compactorder1}, and, in addition,
\begin{equation}\label{eq:1--1b}
d_1=1\quad \text{ and }\quad a_{31}a_{12} \leq a_{32}.
\end{equation} 
Let $(u_i)_{i=1}^3$  be any solution of \eqref{eq:1-1} such that $({\rm H}_{c_1,c_2,\lambda})$ holds for some $c_1>c_2>0$ and $\lambda \in (0,\infty]$. 
Then the maximal and minimal speeds $\overline{c}_3, \underline{c}_3$ can be estimated as follows. 
\begin{equation}\label{estamitec3}
  2\sqrt{d_3r_3(1-a_{31} - a_{32})}
  \leq\underline{c}_3\leq\overline{c}_3\leq\max\{s_{{\rm nlp}}(c_1,c_2,\lambda),c_{\rm LLW}\}<c_2.
\end{equation}
Furthermore,  for each $\eta>0$  small, the following spreading results hold:
\begin{equation}
\begin{cases}
\lim\limits_{t\rightarrow \infty} \sup\limits_{ x>(c_1+\eta) t} (|u_1(t,x)|+|u_2(t,x)|+|u_3(t,x)|)=0, \\
\lim\limits_{t\rightarrow \infty} \sup\limits_{(c_2+\eta) t< x<(c_1-\eta) t} (|u_1(t,x)-1|+|u_2(t,x)|+|u_3(t,x)|)=0, \\
\lim\limits_{t\rightarrow \infty} \sup\limits_{( \overline{c}_{3}+\eta)t< x<(c_2-\eta) t} (|u_1(t,x)|+|u_2(t,x)-1|+|u_3(t,x)|)=0,\\
\liminf\limits_{t \rightarrow \infty} \inf\limits_{0 \leq  x < (\underline{c}_3 -\eta)t} u_3(t,x)>0. 
 \end{cases} \label{eq:spreadingly}
\end{equation}
If, in addition, $ s_{{\rm nlp}}(c_1,c_2,\lambda)
\geq c_{\rm LLW}$, then the spreading speed of $u_3$  is 
fully determined by
\begin{equation}\label{eq:spreadings0}
\underline{c}_3=\overline{c}_3=s_{{\rm nlp}}(c_1,c_2,\lambda).
\end{equation} }
\end{thma}

\begin{remark}
\mbox{}

\begin{itemize}
\item[{\rm(i)}] The condition \eqref{eq:1--1b} is needed to ensure (see Proposition \ref{thm:u1v1111})
$$
\lim\limits_{t\rightarrow \infty} \inf\limits_{(c_2-\eta) t< x<(c_2+\eta) t} (a_{31}u_1(t,x)+a_{32}u_2(t,x))\geq \min\{a_{31},a_{32}\},
$$
which says that there is no ``gap" for $u_3$ to exploit when $u_2$ is taking over $u_1$. See Remark \ref{rmk:3.21} for further discussions.

\item[{\rm(ii)}] The condition $s_{\rm nlp}(c_1,c_2,\lambda) \geq c_{\rm LLW}$ is always satisfied for some $\lambda \in (0,\infty]$.

\end{itemize}
See also Corollary \ref{corollary1.5} and Proposition  \ref{coro:1estamoverc001} for two instances when all the hypotheses of {Theorem A}, including $({\rm H}_{c_1,c_2,\lambda})$, can be verified. 
\end{remark}

We also determine the asymptotic profile of $u_3$ in the final zone $\{(t,x):x < \underline{c}_3 t\}$. 
and give explicit formulas of $s_{\rm nlp}(c_1,c_2,\lambda)$ by solving \eqref{eq:w3nlp}. These are contained in Section \ref{S5} and Appendix \ref{appendix_C}, respectively.

\begin{thmb}
\label{afterunderc3}
 \textit{Let $(u_i)_{i=1}^3$  be any solution of \eqref{eq:1-1} with the initial data 
$u_{3,0} \not\equiv 0$. Suppose that
 $a_{13},a_{23}>1$ and \eqref{eq:compactorder1} hold. Then $\underline{c}_3 \geq 2\sqrt{d_3 r_3} \sqrt{1-a_{31}-a_{32}}$, and  for each small $\eta>0$,
\begin{equation}\label{spreadings}
\lim\limits_{t\to\infty}\sup\limits_{0<x<(\underline{c}_3-\eta)t}(|u_1(t,x)|+|u_2(t,x)|+|u_3(t,x)-1|)=0.
\end{equation}}
\end{thmb}
 The assumptions \eqref{eq:compactorder1} and $a_{13},a_{23}>1$ mean that the species $u_3$ is a strong competitor to species $u_1,u_2$, and hence  eventually invades  and drives $u_1,u_2$ to extinction. It is illustrated by  numerics in Subsection \ref{S1.3} that the condition  $a_{13}, a_{23}>1$ is likely optimal to ensure \eqref{spreadings}.

The following result gives an explicit formula for the speed $s_{\rm nlp}(c_1,c_2,\lambda)$. The proof is presented in Appendix \ref{appendix_C}.
\begin{thmC}
\label{charactsnlp}
\textit{Let $s_{{\rm nlp}}
$ be defined by \eqref{charact_snlp} for given $c_1>c_2>0$ and $\lambda\in(0,\infty]$. Then
 \begin{equation}\label{eq:c_3nlp}
s_{{\rm nlp}}=
\begin{cases}
\medskip
d_3\lambda_{{\rm nlp}1}+\frac{r_3(1-a_{32})}{\lambda_{{\rm nlp}1}} & \text{for } \zeta_1>\frac{c_2}{2{d_3}}, a_{31}<a_{32},  \text{ and } \lambda_{{\rm nlp}1}\leq\sqrt{\frac{r_3(1-a_{32})}{d_3}},
\\
\medskip
d_3\lambda_{{\rm nlp}2}+\frac{r_3(1-a_{32})}{\lambda_{{\rm nlp}2}} &
\begin{cases}
\medskip
\text { for } \zeta_1\leq \frac{c_2}{2{d_3}},  \lambda_{{\rm nlp}2}\leq\sqrt{\frac{r_3(1-a_{32})}{d_3}} \text{ and }\\
\begin{array}{l}
\medskip
{\rm(i)}\,   a_{31}<a_{32}\,  \text{ or }\\
{\rm(ii)}\,  a_{31}\geq a_{32} \text{ and } \zeta_1+\zeta_2<\frac{c_2}{d_3},
\end{array}
\end{cases}\\
2\sqrt{d_3r_3(1-a_{32})}  &\text{otherwise},
\end{cases}
\end{equation}
where
\begin{equation}\label{lambda3}
\begin{cases}
\smallskip
\zeta_1 =
\begin{cases}
\frac{c_1}{2d_3} - \sqrt{\frac{r_3a_{31}}{d_3}} &\text{ for }\lambda\geq \frac{c_1}{2d_3},\\
\frac{c_1}{2d_3}- \frac{\sqrt{(c_1 - 2d_3\lambda)^2 + 4d_3r_3a_{31}}}{2d_3} &\text{ for } \lambda< \frac{c_1}{2d_3}{,}
\end{cases}\\
\zeta_2=\frac{c_2}{2d_3}+\sqrt{\frac{r_3(a_{31}-a_{32})}{d_3}},\\
 \smallskip
\lambda_{{\rm nlp}1}=\frac{c_2}{2d_3}-\sqrt{\frac{r_3(a_{32}-a_{31})}{d_3}},\\ 
\lambda_{{\rm nlp}2}=\frac{c_2-\sqrt{\left(c_2-2d_3\zeta_1\right)^2+4d_3r_3(a_{32}-a_{31})}}{2d_3}. 
\end{cases}
\end{equation}}
\end{thmC}

We briefly discuss the difference of this work with our previous work \cite{Girardin_2019,LLL2019,LLL20192}. 
In \cite{Girardin_2019}, the two-species system \eqref{eq:2sp} generates a monotone dynamical system, so that the result can be obtained by constructing a single pair of weak super- and sub-solutions, and applying comparison principle in the entire domain $\mathbb{R} \times (0,\infty)$. In \cite{LLL2019,LLL20192}, by analyzing the Hamilton-Jacobi equations which are satisfied by the limit of the rate function $w_2(t,x) = \lim\limits_{\epsilon \to 0} -\epsilon \log u_2\left(\frac{t}{\epsilon}, \frac{x}{\epsilon}\right)$, we obtained large deviation type estimates for $u_2$ along the ray $\{(t,x): x =c_1t\}$, which in the case of compactly supported initial data says
\begin{equation*}
u_2(c_1 t, t) = \exp\left(  -(\mu_0  + o(1))t \right), \quad \text{ with }\,\,\,\mu_0 = \left( \frac{c_1}{2} - \sqrt{a_{21}}\right)(c_1 - \hat{s}_{\rm nlp}),
\end{equation*}
where $c_1 = 2\sqrt{d_1 r_1}$ and  $\hat{s}_{\rm nlp}$ is given in \eqref{qq:1.7a}. Hence, we can restrict the equation \eqref{eq:2sp} into the sectorial domain $\{(x,t):\, 0 \leq x \leq c_1 t\}$ with the boundary conditions
\begin{equation*}
(u,v)(0,t) \to (1,0),\quad (u,v)(c_1 t, t) \to (0,1), \quad \text{ and }\quad u(c_1 t,t) \sim e^{-\mu_0 t},
\end{equation*}
for $t \gg 1$. Hence, only one comparison with the traveling wave solution was enough to determine the speed $c_2$.

The main difficulty of treating the three-species system \eqref{eq:1-1}, in connection with the spreading speed of the slowest species $u_3$, is the lack of monotonicity of the full system. Our first idea is to use the subsystem \eqref{eq:1-2'} between the second and the third species to estimate $c_3$ from above. However, \eqref{eq:1-2'} is non-optimal, as we have set $u_1 \equiv 1$, whereas it ought to hold that $u_1 \approx \chi_{\{c_2 t < x < c_1 t\}}$ in the ``correct" system. In fact, whenever $a_{21}>0$, 
the traveling wave solutions of \eqref{eq:1-2'} always overestimate $c_3$. Similarly, $a_{31}>0$ causes trouble when we try to estimate $c_3$ from below. (When either one of them is sufficiently small, we can determined $c_3$ exactly, see  Corollary \ref{corollary1.5} and Proposition  \ref{coro:1estamoverc001}.)

Our second idea is to estimate the rate function  $w_3(t,x) = \lim\limits_{\epsilon \to 0} -\epsilon\log u_3\left(\frac{t}{\epsilon}, \frac{x}{\epsilon}\right)$ directly, to show that  $w_3(t,x) >0$ for $ x/t 
> \max\{c_{\rm LLW},s_{\rm nlp}\}$,
which is equivalent to $\overline{c}_3 \leq\max\{c_{\rm LLW},s_{\rm nlp}\}.$ While this cannot be achieved by a single comparison, we can leverage the second species $u_2$ to control the first species $u_1$, and use an iterative method to improve the estimate step by step. 
{Theorem A} further implies that the estimate \eqref{estamitec3} is sharp in case $s_{{\rm nlp}}
\geq c_{\rm LLW}$.

A particular instance when we can completely determine the speed of $u_3$ occurs when $a_{21}$ is small while the other parameters $d_i,r_i,a_{ij}$ are fixed and satisfy
 \begin{equation}\label{eq:suffc1}
d_1=1,\quad a_{12} >1, \quad a_{32} \leq \frac{1}{2}, \quad a_{31} \leq \frac{a_{32}}{a_{12}},
\end{equation}
 \begin{equation}\label{eq:suffc2}
d_3 \geq\frac{1}{2},\quad a_{13} >1, \quad 1 < a_{23} <\frac{1}{a_{32}}, 
\end{equation}
 \begin{equation}\label{eq:suffc3}
\frac{1}{\sqrt{d_3}(\sqrt{a_{32}} + \sqrt{1-a_{32}})}< \sqrt{r_3} < \frac{1}{\sqrt{d_3}},\quad  
1 <\sqrt{r_1} <  \sqrt{d_3r_3}(\sqrt{a_{32}} + \sqrt{1-a_{32}}).
\end{equation}
\begin{remark}\label{rem1.2}

We claim that the set of parameters satisfying \eqref{eq:suffc1}-\eqref{eq:suffc3} is nonempty. Indeed, one can choose $d_i,r_i,a_{ij}$ in the following order: fix  $d_1, a_{12}, a_{32}, a_{31}$ by \eqref{eq:suffc1}, then fix  $d_3, a_{13}, a_{23}$ by \eqref{eq:suffc2}, finally fix $r_1$ and $r_3$ by \eqref{eq:suffc3}. 
\end{remark}

\begin{corollary}\label{corollary1.5}
Fix all coefficients, except for $a_{21}$, to satisfy  \eqref{eq:suffc1}-\eqref{eq:suffc3}. Then there exists $\delta>0$ such that for all $a_{21} \in [0,\delta)$, any solution $(u_i)_{i=1}^3$  of \eqref{eq:1-1} with compactly supported initial data, 
 satisfies  \eqref{eq:spreadingly},  \eqref{eq:spreadings0},and  \eqref{spreadings} with
 \begin{equation*}
c_1 = 2\sqrt{d_1 r_1}, \quad c_2 = 2\sqrt{1-a_{21}},\quad \overline{c}_3=\underline{c}_3 = s_{\rm nlp}(c_1,c_2,\infty).
\end{equation*}
\end{corollary}
\begin{proof}
By $a_{31} < a_{32}$ (from \eqref{eq:suffc1}) and  the latter part of \eqref{eq:suffc3}, we can apply 
Remark \ref{rmk:1.3}(ii) (to be proved in Proposition \ref{charactsnlpc})
to show that $s_{\rm nlp} > 2\sqrt{d_3r_3(1-a_{32})}$, where $s_{\rm nlp} = s_{\rm nlp}(2\sqrt{d_1r_1},2\sqrt{1-a_{21}},\infty).$
By taking $a_{21}$ sufficiently small, we can further assume 
\begin{equation}\label{eq:suffc4}
a_{12}a_{21}< 1, 
\quad 2\sqrt{d_3r_3(1-a_{32}(1-a_{21}))} < s_{\rm nlp},
\end{equation}
and
\begin{equation}\label{eq:suffc5}
a_{32} (1-a_{21}) < 1 < \frac{a_{23}}{1-a_{21}},\quad \sqrt{d_3r_3} < \sqrt{1-a_{21}},\quad \sqrt{a_{21}}+\sqrt{1-a_{21}} < \sqrt{r_1}.
\end{equation}

Observe next that \eqref{eq:compactorder1} and \eqref{eq:1--1b} are a consequences of $a_{21} \in (0,1)$  
and \eqref{eq:suffc1}. 

 We claim that $\hat{c}_{\rm LLW} = 2\sqrt{1-a_{21}} = \hat{s}_{\rm nlp}(2\sqrt{d_1r_1})$. The first equality follows from $d_1=1$, $a_{21} < 1 < a_{12} $, and $a_{12}a_{21}<1$ as  in Remark \ref{rmk:1--2a}. 
 The second equality is due to $d_1=1$, the latter part of \eqref{eq:suffc5}, and \eqref{qq:1.7a}. Combining with the middle part of \eqref{eq:suffc5}, we have
 \begin{equation*}
 2\sqrt{d_3r_3} < 2\sqrt{1-a_{21}} = \hat{s}_{\rm nlp}(2\sqrt{d_1 r_1}) \quad \text{ and }\quad 
\hat{c}_{\rm LLW} = 2\sqrt{1-a_{21}}.
\end{equation*}
Hence, we may apply Theorem \ref{thm:suff1.2} to conclude that 
$({\rm H_{c_1,c_2,\lambda}})$ holds with 
$$
c_1 = 2\sqrt{d_1r_1},\quad c_2 = \max\{\hat{s}_{\rm nlp},\hat{c}_{\rm LLW}\} = 2\sqrt{1-a_{21}}, \quad \lambda = \infty.
$$

Having verified $({\rm H}_{c_1,c_2,\lambda})$, \eqref{eq:compactorder1}, and  \eqref{eq:1--1b}, we can apply {Theorem A}. Assuming 
\begin{equation}\label{qq:2.0}
s_{\rm nlp}(c_1,c_2,\infty) \geq c_{\rm LLW},
\end{equation}
then the spreading speed of the third species can be uniquely determined as $s_{\rm nlp}(c_1,c_2,\infty)$. Since also $a_{13}>1$ and $a_{23} >1$, we apply {Theorem B} to yield that $(u_1,u_2,u_3) \approx (0,0,1)$ in the final zone after the invasion of $u_3$.

It remains to show \eqref{qq:2.0}. To this end, we first claim that $c_{\rm LLW} = 2\sqrt{d_3 r_3} \sqrt{1-a_{32}(1-a_{21})}$. This follows since  $d_3>1/2$, $a_{32} (1-a_{21}) < 1 < \frac{a_{23}}{1-a_{21}}$, and $a_{32}a_{23}<1$. See Lemma \ref{lem:1estamoverc30}. 

Combining $c_{\rm LLW} = 2\sqrt{d_3 r_3} \sqrt{1-a_{32}(1-a_{21})}$ with the 
latter part of \eqref{eq:suffc4}, we deduce \eqref{qq:2.0}. This concludes the proof.
\end{proof}
The conditions on coefficients in Corollary \ref{corollary1.5} can be further relaxed, if we allow $u_3$ to have exponential decay.
\begin{proposition}\label{coro:1estamoverc001}
Assume the hierarchy condition \eqref{eq:compactorder1} and condition \eqref{eq:1--1b} hold, and 
\begin{equation}\label{eq:prop1.14a}
d_3 \geq \frac{1}{2}, \quad a_{12}a_{21} < 1, \quad 2\sqrt{d_3 r_3} < \hat{s}_{\rm nlp}(2\sqrt{d_1r_1}), \quad a_{13}>1, \quad a_{23}>1. 
\end{equation}
Then there exist $\lambda \in (0,\infty]$ and $\delta>0$ such that for any solution of \eqref{eq:1-1} with 
initial data satisfying $({\rm H}_{\lambda})$ and $a_{31}\in [0,\delta)$,  the conclusions \eqref{eq:spreadingly}, 
\eqref{eq:spreadings0}, and  \eqref{spreadings} hold with
$$
c_1 = 2\sqrt{d_1 r_1},\quad c_2 = \hat{s}_{\rm nlp}(c_1), \quad \overline{c}_3 = \underline{c}_3= s_{\rm nlp}(c_1,c_2,\lambda),
$$
where $\hat{s}_{\rm nlp}(c_1)$ is given by \eqref{qq:1.7a}.
\end{proposition}
Apart from the smallness of $a_{31}$, here we need only the technical conditions $d_3 \geq 1/2$, \eqref{eq:1--1b}, and $a_{21}a_{12} <1$. All other conditions on coefficients are natural. The proof of Proposition \ref{coro:1estamoverc001} is postponed to Section \ref{S6}.

\subsection{Numerical simulations 
}\label{S1.3}

In this subsection, 
we present some 
 numerical results of system \eqref{eq:1-1} with compactly supported initial data to illustrate our main findings. 

In the first numerical result, we simulate the speed of $u_3$ to  illustrate {Theorem A}.  
 The parameters in \eqref{eq:1-1}, except for $a_{21}$, are fixed by $r_1=1.08$,  $d_1=1$,  $r_3=1.1$, $d_3=0.6$, $a_{12}=1.2$,  $a_{31}=0.1$, $a_{13}=1.1$,  $a_{32}=0.4$, $a_{23}=1.1$.
 It is straightforward to verify that \eqref{eq:suffc1}-\eqref{eq:suffc3} are satisfied.

 First, we take $a_{21}=0.01$ and then $c_{\rm LLW}= 2\sqrt{d_3 r_3} \sqrt{1-a_{32}(1-a_{21})}=1.2628$, since $c_{\rm LLW}$ is linearly determined for the chosen parameters. 
 We use the second-order finite difference schemes to discretize $[x, t]$ domain  and use the Explicit Euler scheme to solve
 system  \eqref{eq:1-1} numerically. Noting that for $a_{21}=0.01$ small, the spreading speed of $u_3$ can be fully determined by
$c_3=s_{{\rm nlp}}>c_{\rm LLW}$ (Corollary \ref{corollary1.5}). This is in agreement with the numerical result in  Figure \ref{figure1''} and illustrates that the estimate \eqref{estamitec3} in {Theorem A} is sharp in this case. However, if we take $a_{21}=0.5$ so large that Corollary \ref{corollary1.5} is  not applicable, then it is shown in Figure \ref{figure1''} that the case $c_3<c_{\rm LLW}$ could happen, where 
 $c_{\rm LLW}= 2\sqrt{d_3 r_3} \sqrt{1-a_{32}(1-a_{21})}=1.4533$. This suggests that the estimate \eqref{estamitec3} is not necessarily sharp in all situations and it cannot be expected that $c_3=\max\{s_{\rm nlp}, c_{\rm LLW}\}$. 
 It remains open to determine the spreading speed of $u_3$ for the case $s_{\rm nlp}<c_{\rm LLW}$. An improved estimate for the lower bound of $\underline{c}_3$  is provided in  Proposition \ref{underc3}, which is given as the free boundary point of a  variational inequality associated with $c_{\rm LLW}$.


 \begin{figure}[http!!]
\centering
\includegraphics[height=2.35in]{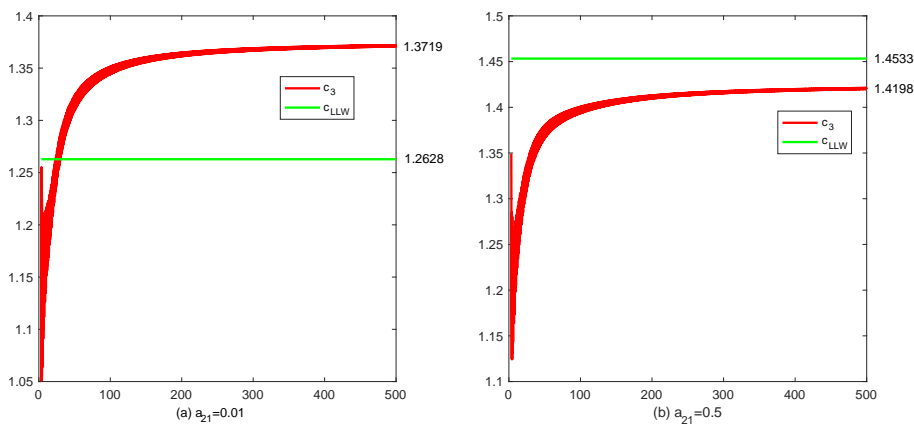}
\caption{\small Approximate speed of $u_3$ with  the initial value $u_1(0,x)=u_2(0,x)=u_3(0,x)=\chi_{[0,10]}$, and with (a) $a_{21}=0.01$, and (b) $a_{21}=0.5$, where other parameters are chosen by $r_1=1.08$, $d_1=1$, $r_3=1.1$, $d_3=0.6$, $a_{12}=1.2$, $a_{31}=0.1$, $a_{32}=0.4$, $a_{13}=1.1$, $a_{23}=1.1.$  
}\label{figure1''}
  \end{figure}

Our next  numerical result illustrates that the condition $a_{13},a_{23}>1$ in {Theorem B} is  optimal to guarantee \eqref{spreadings}. The asymptotic behaviors of the solution for system \eqref{eq:1-1} are illustrated in Figure \ref{figure3} for the four cases: \rm{(a)} $a_{31}>1$ and $a_{32}<1$, \rm{(b)} $a_{31}<1$ and $a_{32}>1$, {\rm{(c)}} $a_{31}<1$ and $a_{32}<1$, {\rm{(d)}}  $a_{31}>1$ and $a_{32}>1$. 
Note that {Theorem A} is independent of $a_{31}$ and $a_{32}$, and our choices of parameters in the simulation satisfy the hypotheses of {Theorem A}.
It is shown in Figure \ref{figure3} that for the case {\rm{(d)}} when $a_{13}=a_{23}=1.1>1$, the solutions of \eqref{eq:1-1} behave as predicted by {Theorem B}, i.e. species $u_1$ and $u_2$ are driven to extinction  behind the spreading of $u_3$.
However,  once $a_{13}>1$ and  $a_{23}<1$ even though they are closed  to $1$, it is shown in   Figure \ref{figure3}{\rm{(a)}} that species  $u_2$ and $u_3$ may coexist in the final zone $\{(t,x):x < \underline{c}_3 t\}$, and similarly species  $u_1$ and $u_3$ may coexist when $a_{13}<1$ and  $a_{23}>1$; see Figure \ref{figure3}{\rm{(b)}}. Interestingly, when $a_{13}<1$ and  $a_{23}<1$, all the three species may coexist and studying 
the general propagating terraces in the final zone is beyond the scope of this paper.

\begin{figure}[http!!]
\includegraphics[height=4.4in]{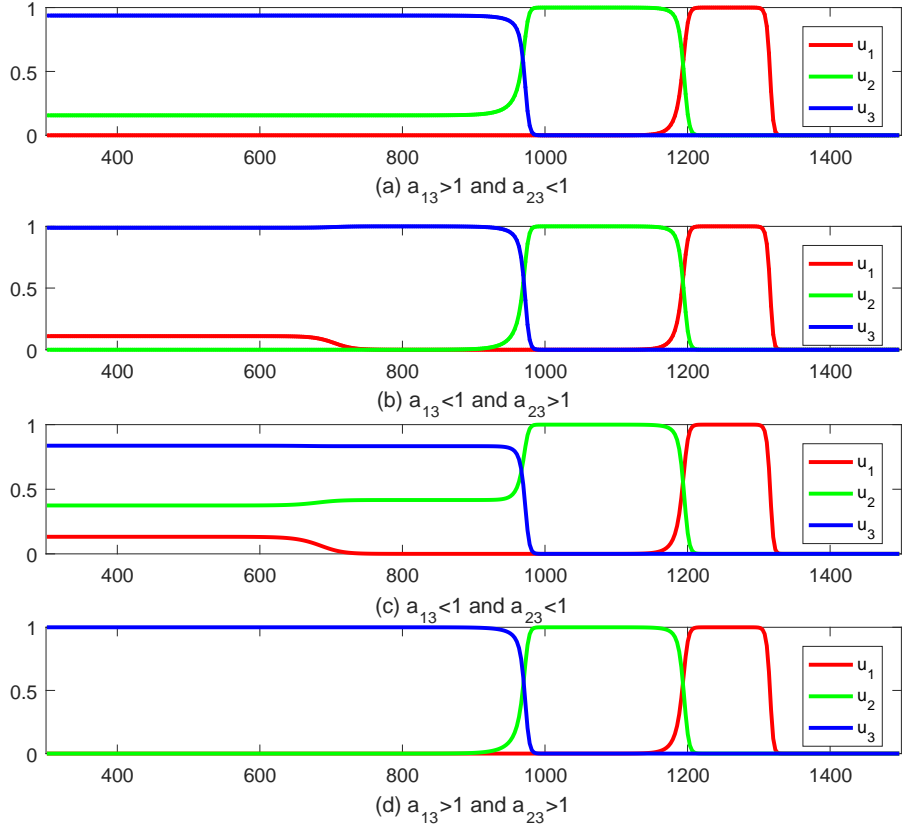}
  \caption{\small Asymptotic behaviors of solutions of \eqref{eq:1-1} with different $a_{13}$ and $a_{23}$, where the other parameters are chosen by $r_1=1.08$, $d_1=1$, $r_3=1.1$, $d_3=0.6$, $a_{12}=1.2, a_{21}=0.3, a_{31}=0.1, a_{32}=0.4$. In case \rm{(a)},  $a_{13}=1.1$ and $a_{23}=0.9$; In case \rm{(b)}, $ a_{13}=0.9$ and $a_{23}=1.1$; In case \rm{(c)}, $ a_{13}=0.5$ and $a_{23}=0.7$; In case \rm{(d)}, $ a_{13}=1.1$ and $ a_{23}=1.1$.  
  The initial value is set as $u_{1,0}=u_{2,0}=u_{3,0}=\chi_{[0,10]}$.
 }\label{figure3}
 \vspace{-0.2cm}
\end{figure}

\subsection{Related results} 
The pioneering works due to Aronson and Weinberger \cite{Aronson_1975, Aronson_1978}
established  the spreading speed for a single-species model with monotone nonlinearity, which coincides with the minimal speed of traveling waves.  Weinberger later introduced  a powerful method based on recursion to determine spreading speed for single-species models \cite{Weinberger_1982}, which is subsequently generalized to systems of equations \cite{Lui_1989} and to general monotone dynamical systems framework \cite{LWL_2005,Liang_2007}. 
A closely related work is \cite{Iida_2011} concerning cooperative systems
with equal diffusion coefficients, where the existence of stacked fronts was also studied.

For a single-species model with non-monotone nonlinearity, Thieme showed in  \cite{Thieme_1979} that the spreading speeds can still be obtained by constructing  monotone representation of the nonlinearity. This idea was used in \cite{WKS_2009} to establish spreading speeds for a partially cooperative system, which is a non-cooperative system that can be controlled from above and from below by cooperative systems. 
See also \cite{Wang_2011} where results on general partially cooperative systems,
%
in the spirit of those in \cite{LWL_2005} were obtained.
Besides, by Schauder's fixed-point theorem, Girardin \cite{Girardin2_2018, Girardin1_2018} established a number of general results on spreading speeds and traveling waves for non-cooperative systems with the property that the linearization at the trivial equilibrium is cooperative.

General non-cooperative systems, such as predator-prey systems, 
cannot always be controlled by cooperative systems.
Due to the lack of comparison principles, much less is known about the spreading properties for such type of systems.
Recent work due to Ducrot et al. \cite{Ducrot_preprint} investigated certain  predator-prey systems by means of  ideas in persistence theory \cite{{ST_2011}} in dynamical systems; see also \cite{Ducrot_2013,Ducrot_2016} for related results. For integro-difference models, the spreading properties were considered by Hsu and Zhao \cite{HZ_2008} and Li \cite{Li_2018}, where the linearly determined speeds were obtained under appropriate assumptions.

For  the three-species systems, 
In the special case when $a_{12}=a_{21}=0$, $a_{13},a_{23}>1$ and $a_{31}+a_{32}<1$, \eqref{eq:1-1}  can be transformed into a cooperative system, and Guo et al. \cite{GWWW2015} studied the minimal speed of traveling waves in this setting. Therein they applied the monotone iteration scheme to give some conditions on the parameters such that the minimal speed of traveling waves connecting $(1,1,0)$ to $(0,0,1)$ is linearly determined. See also 
\cite{Hou_2017} and \cite{Liu2015} for similar spreading results in two other such three-species models, and 
\cite{Chen_2018} for the existence of entire solutions for the monotone system, which behaved as two traveling fronts moving towards each other from both sides of $x$-axis.

To the best of our knowledge, however, much less is known about the spreading properties for systems of equations for which the comparison principle does not hold, except for the recent works \cite{CT2020, Ducrot_preprint,Xiao_2019}. 
In particular,
the rigorous analysis of spreading properties of fully coupled three-species competition system \eqref{eq:1-1}, has never been carried out before.  It is this knowledge gap that has motivated the research in this paper.

\subsection{Organization of this paper}

In Section \ref{sec:prelim}, we introduce and recall some comparison principles for a class of variational inequality. This inequality is cast in the space of speeds, and its solutions are to be understood in the viscosity sense. 
In Section \ref{Section2}, we derive upper and lower estimates of the spreading speed of the slowest species under hypothesis $({\rm H}_{c_1,c_2,\lambda})$, which is sharp in some situations. The main result of this paper, i.e. {Theorem A}, is proved in Sections \ref{sec:prelim} and \ref{Section2}, which are self-contained. The readers who are only interested in the main ideas of the paper can focus their attention here.

In Section \ref{S5}, we determine the convergence to homogeneous state in the wake of the invasion wave of the third species and prove {Theorem B}. In Section \ref{S6}, we derive Remark \ref{rmk:1.3} and prove Proposition \ref{coro:1estamoverc001}, 
which lead to sufficient conditions for the full determination of spreading speeds. Finally, we include some useful lemmas in 
Appendix  \ref{appendixlemmas} and  establish  Lemma \ref{lemA1rho} and Proposition \ref{lemA3rho} in Appendix \ref{sec:B}. 

\section{Preliminaries}\label{sec:prelim}\label{comparison}

In the introduction, there are various quantities including $s_{\rm nlp}$ defined via the viscosity solutions to some variational inequalities. We briefly explain the connection of such quantities to the spreading speeds of the population here. Suppose $u_\ep(t,x)$ is a solution to the rescaled Fisher-KPP equation
$$
\partial_t u_\ep = \ep\hat{d} \partial_{xx} u_\ep +\frac{1}{\ep} u_\ep (r(t,x) - u_\ep), 
$$
where $\hat{d}>0$ is a constant and $r(t,x)$ is a bounded function. 
It was observed in \cite{Evans_1989, Freidlin_1985} that 
the propagation phenomena is well described by the rate function $w_\ep(t,x) = -\ep \log u_\ep(t,x)$. Moreover, the local uniform limit $w(t,x) = \lim\limits_{\ep \to 0} w_\ep(t,x)$, if it exists, satisfies the following first order Hamilton-Jacobi equation in the viscosity sense:
\begin{equation}\label{qq:weq}
\min\left\{\partial_t w(t,x) + \hat{d}|\partial_x w(t,x)|^2 + r(t,x),\, w(t,x)\right\}=0.
\end{equation}
 Roughly speaking, the population density $u_\ep(t,x)$ is exponentially small when $w(t,x)>0$, while $u_\ep(t,x)$ is bounded below  by some positive number when $w(t,x) =0$. See Lemma \ref{lem:underlinec} for the precise statement of the latter claim.

In the special case $r(t,x) = \hat{\mathcal{R}}\left( \frac{x}{t}\right)$ (i.e. it depends only on $x/t$), then the limit $w(t,x)$ can be represented in the form $t\rho\left( \frac{x}{t}\right)$ for some continuous function $\rho$ (see Remark \ref{rmk:w1w3} for details). In such an event, it is convenient to work on $\rho$, which satisfies a reduced equation cast in the space of speed $s=x/t$, which is one-dimensional. (See Proposition \ref{lemA3rho} for the proof.) 
%
%
%
%
%
%
%
\begin{align}\label{hatrho00}
\min\{\rho(s)-s\rho'(s)+\hat d|\rho'(s)|^2+\hat{\mathcal{R}}(s),\rho(s)\}=0. 
\end{align}
Now, if there is $\hat{s}>0$ such that
$$
\rho(s) = 0 \quad \text{ for }\, 0 < s < \hat{s},\quad \text{ and }\quad \rho(s)>0 \quad \text{ for }\,  s > \hat{s},
$$
Then the population is exponentially zero in $\{(t,x):\, x/t > \hat{s}\}$, and is bounded from below in $\{(t,x):\, x/t < \hat{s}\}$, i.e. it spreads at speed $\hat{s}$ in the sense of  \cite{Aronson_1975, Aronson_1978}.   
This motivates the definition of spreading speed in terms of the free boundary point of \eqref{hatrho00} in this paper.  

Next, we give the definitions of viscosity super- and sub-solutions associated with \eqref{hatrho00} (see \cite[Sect. 6.1]{BarlesLect}). 
The corresponding definitions for \eqref{qq:weq} are similar and are given in \cite[Appendix A]{LLL20192}, where a comparison principle is established. For our purposes, we will henceforth assume that the function $\hat{\mathcal{R}}(s)$ is bounded and piecewise
Lipschitz continuous.

\begin{definition}\label{defviscosity}
We say that a  lower semicontinuous function $\hat\rho$ is  a  viscosity super-solution of \eqref{hatrho00} if $\hat\rho\geq 0$, and for all test functions $\phi\in C^1$, if $s_0$ is a strict local minimum point of $\hat\rho-\phi$, then
 $$\hat\rho(s_0)-s_0\phi'(s_0)+\hat d|\phi'(s_0)|^2+\hat{\mathcal{R}}^*(s_0)\geq 0.$$
We say that a upper semicontinuous  function $\hat\rho$ is a  viscosity sub-solution of \eqref{hatrho00} if for any test function $\phi\in C^1$,  if $s_0$  is a strict local maximum point of $\hat\rho-\phi$ such that $\hat\rho(s_0) >0$, then
  $$\hat\rho(s_0)-s_0\phi'(s_0)+\hat d|\phi'(s_0)|^2+\hat{\mathcal{R}}_*(s_0)\leq 0.$$
Finally, $\hat\rho$ is a  viscosity solution  of \eqref{hatrho00} 
if and only if $\hat\rho$ is a viscosity super- and  sub-solution. 
\end{definition}

The functions $\hat{\mathcal{R}}^*$ and $\hat{\mathcal{R}}_*$ appeared above  denote respectively the upper semicontinuous  and lower semicontinuous  envelope of $\hat{\mathcal{R}}$, i.e. 
$$\hat{\mathcal{R}}^*(s)=\limsup_{s'\to s}\hat{\mathcal{R}}(s')\quad \text{ and } \quad \hat{\mathcal{R}}_*(s)=\liminf_{s'\to s}\hat{\mathcal{R}}(s').$$

\begin{lemma}\label{lem:wrhoviscosity}
Let $c_b \in (0,\infty]$ be given. A function $\rho(s)$ is a viscosity sub-solution (resp. super-solution) of \eqref{hatrho00} in the interval $(0,c_b)$ if and only if $w(t,x) = t \rho\left( \frac{x}{t}\right)$ is a viscosity sub-solution (resp. super-solution) of 
\begin{equation}\label{eq:hatw00}
\min\left\{\partial_t w+\hat d|\partial_xw|^2+\hat{\mathcal{R}}\left(x/t 
\right),w\right\}=0 
\end{equation}
in the domain $\{(t,x): 0 < x < c_b t\}$.
\end{lemma}
\begin{proof}
Let $\rho(s)$ be a viscosity sub-solution of \eqref{hatrho00} in $(0,c_b)$. Let us verify that $w(t,x)= t\rho\left(\frac{x}{t}\right)$ is a viscosity sub-solution of \eqref{eq:hatw00}.
Suppose that ${w}-\varphi$ attains a strict local maximum at point $(t_*,x_*)$ such that ${w}(t_*,x_*) >0$ for any test function $\varphi\in C^1$.
Since ${w}(t,x) = t {\rho}(\frac{x}{t})$, we deduce that ${\rho}(\frac{x_*}{t_*}) >0$ and $\tau \mapsto \tau t_* {\rho}(\frac{x_*}{t_*})-\varphi(\tau t_*, \tau x_*)$ has a strict local maximum at $\tau=1$, so that letting $s_* =x_*/{t_*}$  we have
\begin{equation}\label{equality1}
  t_* {\rho}\left(s_*\right)-t_*\partial_t\varphi(t_*,x_*)-x_*\partial_x\varphi(t_*,x_*) =0.
\end{equation}
Set
$\phi(s):=\varphi(t_*,s t_*)/t_*.$
It can be verified that $\rho(s)-\phi(s)$ takes a strict local maximum point $s=s_*$ and $\rho(s_*)>0$. Moreover, by \eqref{equality1}  we arrive at
\begin{equation*}
\partial_x \varphi(t_*,x_*)=\phi'(s_*) \quad\text{ and }\quad \partial_t \varphi(t_*,x_*)=\rho(s_*)-s_*\phi'(s_*).
\end{equation*}
Hence at the point $(t_*, x_*)$, direct calculation yields 
\begin{align*}
\partial_t \varphi+\hat d\left|\partial_x \varphi\right|^2+\hat{\mathcal{R}}_*\left(x_*/t_* 
\right) 
 &= \rho(s_*)-s_*\phi'(s_*)+\hat{d}|\phi'(s_*)|^2+\hat{\mathcal{R}}_*\left(s_*\right)\leq 0,
\end{align*}
where the last inequality holds since $\rho$ is a viscosity sub-solution of \eqref{hatrho00} with $\phi(s)$ being the test function.
Hence ${w}$ is a viscosity sub-solution of \eqref{eq:hatw00}. 

Conversely, let $w(t,x) = t\rho\left(\frac{x}{t}\right)$ be a viscosity sub-solution of \eqref{qq:weq}. Choose  any test function $\phi\in C^1$ such that $\rho(s) - \phi(s)$ attains a strict local maximum at $s_*$ such that $\rho(s_*)>0$. Without loss of generality, we may assume $\rho(s_*)-\phi(s_*)=0$. 
Then $w(t,x) -t\phi\left(\frac{x}{t}\right)-(t-1)^2= t\rho\left(\frac{x}{t}\right)-t\phi\left(\frac{x}{t}\right)-(t-1)^2$ attains a  strict local maximum at $(s_*,1)$. Hence, by the definition of $w(t,x)$ being a sub-solution, we deduce that
$$
\phi(s_*) - s_* \phi'(s_*) + |\phi'(s_*)|^2 + \hat{\mathcal{R}}_*(s_*) \leq 0,
$$
which implies that  $\rho$ is a viscosity sub-solution of \eqref{hatrho00}.

The proof of the equivalence for viscosity super-solutions is similar and is omitted. 
\end{proof}

We now present a comparison result associated with \eqref{hatrho00}. 
\begin{lemma}\label{lemA2rho}
Fix any $c_{b}\in(0,\infty]$. Let  $\overline{\rho}$ and  $\underline{\rho}$ be a pair of viscosity super- and sub-solutions of \eqref{hatrho00} in the interval $(0,c_b)$ with the boundary conditions
\begin{equation}\label{condition1'}
  \underline{\rho}(0)\leq \overline{\rho}(0), \quad \limsup\limits_{s\to c_{b}}\frac{\underline{\rho}(s)}{s}\leq \liminf\limits_{s\to c_{b}}\frac{\overline{\rho}(s)}{s},\quad\text{ and }\quad \lim_{s \to c_b} \frac{\underline{\rho}(s)}{s} < \infty.
\end{equation}
Then we have
$
\overline{\rho} \geq \underline{\rho}$ in $[0, c_{b}].$
\end{lemma}
\begin{proof}
If $(0,c_{b})$ is a bounded interval, then Lemma \ref{lemA2rho} is a direct consequence of \cite[Theorem 2]{Tourin_1992}. It remains to consider the case $c_{b}=\infty$. 
Define
$$
\overline{w}(t,x):=t \overline\rho\left(\frac{x}{t}\right) \quad \text{  and 
} \quad \underline{w}(t,x):=t \underline\rho\left(\frac{x}{t}\right).$$
We will apply the comparison principle \cite[Theorem A.1]{LLL20192} to derive the corresponding result for our reduced equation here. 
By Lemma \ref{lem:wrhoviscosity}, $\overline{w}$ and $\underline{w}$ is a pair of super- and sub-solutions of \eqref{eq:hatw00}. It remains to verify the boundary conditions $ \underline{w}(t,0)\leq \overline{w}(t,0)$ and $\underline{w}(0,x)\leq \overline{w}(0,x)$ for $t> 0,\,x>0$, 
which follows by the following calculations.
\begin{equation}\label{eq:wt0}
\underline{w}(t,0)=t\underline{\rho}(0)\leq t\overline{\rho}(0)=\overline{w}(t,0) \quad \text{for } t\geq 0,
\end{equation}
and
\begin{equation}
\begin{split}
\underline{w}(0,x)\leq\limsup_{t\to 0}\left[t\underline{\rho}\left(\frac{x}{t}\right)\right]&=x\limsup_{s\to \infty}\frac{\underline{\rho}(s)}{s}\leq x\liminf_{s\to \infty}\frac{\overline{\rho}(s)}{s}\\
&=\liminf_{t\to 0}\left[t\overline{\rho}\left(\frac{x}{t}\right)\right]\leq\overline{w}(0,x),
\end{split}
\end{equation}
where we used \eqref{condition1'}. Therefore, we apply \cite[Theorem A.1]{LLL20192} to deduce
 $ \overline{w}\geq \underline{w}$ in $[0,\infty)\times[0,\infty)$, which implies that 
 $\overline{\rho}(s)\geq \underline{\rho}(s) $ for $ s\in[0,\infty)$.
The proof is now complete.
\end{proof}

Hereafter, we let $\hat{\mathcal{R}}(s)=\hat r-g(s)$ for some constant $\hat r>0$ and  $g:[0,\infty)\to\mathbb{R}$ such that
\begin{itemize}
  \item [{$\rm(H_{g})$}]  The function $g$ is  nonnegative, bounded, and piecewise Lipschitz continuous, and  $\mathrm{spt}\, g\subset\left[0, c_{g}\right]$ for some
  $ c_{g}\geq\hat{d}(\hat\lambda\wedge \sqrt{\hat{r}/\hat{d}})+\frac{\hat{r}}{\hat\lambda\wedge \sqrt{\hat{r}/\hat{d}}}$.
\end{itemize}
 Namely,  we consider 
\begin{align}\label{hatrho0}
\left\{\begin{array}{l}
\smallskip
\min\{\rho-s\rho'+\hat d|\rho'|^2+\hat r-g(s),\rho\}=0 \,\,\text{ in } (0,\infty),\\
\rho(0)=0,\quad\lim\limits_{s\to\infty}\frac{\rho(s)}{s}=\hat{\lambda}.
\end{array}
\right.
\end{align}
We mention that {$\rm(H_{g})$} is always satisfied for the variational inequalities derived in this paper. The next result is related to the finite speed of propagation 
for Hamilton-Jacobi equations. 


\begin{lemma}\label{lemA1rho}
Assume that {$\rm(H_{g})$} holds. Then for any $\hat{\lambda}\in(0,\infty]$,
there exists  a unique viscosity solution $\hat\rho$ of \eqref{hatrho0}, and it follows that 
\begin{itemize}
\item[{\rm{(a)}}] if $\hat\lambda\leq \frac{c_{g}}{2\hat{d}}$, then
$\hat\rho(s)=\hat\lambda s-(\hat d\hat\lambda^2+\hat r)$ for $s\geq c_{g};$
\item[{\rm{(b)}}] if $\hat\lambda> \frac{c_{g}}{2\hat{d}}$, then
$$\hat\rho(s)=
\begin{cases}
\medskip
\hat\lambda s-(\hat d\hat\lambda^2+\hat r)&\text{for }\,s\geq 2\hat{d}\hat\lambda ,\\
\frac{s^2}{4\hat d}-\hat r&\text{for }\, c_{g}\leq s<2\hat{d}\hat\lambda.
\end{cases}
$$
\end{itemize}
\end{lemma}
It is well-known that the viscosity solution $\hat{\rho}$ to \eqref{hatrho0} exists and is unique \cite[Theorem 2]{Crandall_1986}. 
Set $\hat{w}(t,x):= t \hat{\rho}\left(\frac{x}{t}\right)$. By Lemma \ref{lem:wrhoviscosity},  $\hat{w}(t,x)$ is a viscosity solution of \eqref{eq:hatw00} in $(0,\infty)\times (0,\infty)$ with the boundary condition $w(t,0) = 0$ 
and the initial condition $w(0,x) = h_{\hat\lambda}(x)$, where 
$$
h_{\hat\lambda}(x) = \hat\lambda x \,\,\text{ when }0<\hat\lambda< \infty,\,\,\text{ and }\,\,
h_\infty(x) = \begin{cases}
0 &\text{ for }x=0, \\
\infty &\text{ for }x>0. 
\end{cases}
$$
In either case, one can use dynamic programming principle (see, e.g. \cite[Theorem 1]{Freidlin_1985} or \cite[Theorem 5.1]{Evans_1989}) to deduce $\hat{w}(t,x)=\max\{J(x,t),0\}$ with
\begin{equation*}
J(t,x) = \inf_{\gamma} \left\{ \int_0^t \left[\frac{|\dot{\gamma}(s)|^2}{4d_3} - \hat{r} + g(\gamma(s)) \right]\mathrm{d}s + h_{\hat\lambda}(\gamma(0))\right\},
\end{equation*}
where the infimum is taken over all absolutely continuous paths $\gamma:[0,t] \to [0,\infty)$ such that $\gamma(t)=x$. Therefore, one can obtain the above assertions (a) and (b) by direct calculations as in \cite[Appendix B]{LLL2019}, which says that the viscosity solution  $\hat{\rho}(s)$ restricted to the interval $[c_g,\infty)$ does not depend on $g$. Alternatively, one may proceed by constructing simple super- and sub-solutions and applying Lemma \ref{lemA2rho}. Since the proof is straightforward but tedious, we will present it in the Appendix \ref{sec:B}. 

\begin{proposition}\label{lemA3rho}
Assume that {$\rm(H_{g})$} holds.  For each $\hat{\lambda}\in(0,\infty]$, let $\tilde{w}(t,x)$ be a viscosity super-solution (resp. sub-solution) of \eqref{qq:weq} with $r = \hat{r}-g(x/t)$.
 Then 
$\tilde{w}(t,x)\geq t\hat\rho\left(\frac{x}{t}\right)$ (resp. $\tilde{w}(t,x)\leq t\hat\rho\left(\frac{x}{t}\right)$) in $(0,\infty)\times(0,\infty)$,
where $\hat\rho$ defines the unique viscosity solution of \eqref{hatrho0}.
\end{proposition}
\begin{proof}
By Lemma \ref{lem:wrhoviscosity}, $t\hat\rho\left(\frac{x}{t}\right)$ is a viscosity solution of \eqref{qq:weq} with $r = \hat{r}-g(x/t)$. Hence, the conclusion follows directly from \cite[Theorem A.1]{LLL20192} for the case $\hat{\lambda}\in(0,\infty)$. The case $\hat{\lambda}=\infty$ will be established by an approximation argument, and is deferred to the Appendix \ref{sec:B}.
\end{proof}

\begin{corollary}\label{corA3rho}
Assume  {$\rm(H_{g})$}.  Fix any $\hat{\lambda}\in(0,\infty]$. Let $\tilde{\rho}$ be a viscosity super-solution (resp. sub-solution) of \eqref{hatrho0}.
 Then 
$\tilde{\rho}(s)\geq \hat\rho\left(s\right)$ (resp. $\tilde{\rho}(s)\leq\hat\rho\left(s\right)$) for $s\in(0,\infty)$,
where $\hat\rho$ is the unique viscosity solution of \eqref{hatrho0}.
\end{corollary}
\begin{proof}
It is a direct consequence of Proposition \ref{lemA3rho} by noting that $\tilde{w}(t,x):=t\tilde{\rho}(x/t)$ defines a viscosity super-solution (resp. sub-solution) of  \eqref{qq:weq} with $r = \hat{r}-g(x/t)$; see Lemma \ref{lem:wrhoviscosity}.
\end{proof}

\section{Proof of {Theorem A}}\label{Section2}

This section is devoted to the proof of the main result, namely, {Theorem A}. We will define some notations in Subsection \ref{subsect:3.1}. Then the estimates for $\overline{c}_3$ and $\underline{c}_3$ are proved in Subsections \ref{S3} and \ref{S4}, respectively. 
We assume, throughout this entire section, that $d_i, r_i, a_{ij}$ and the initial conditions $u_{i,0}$ are fixed in such a way that  $({\rm H}_{c_1,c_2,\lambda})$ holds for some $c_1>c_2$ and $\lambda \in (0,\infty]$ (see Definition \ref{def:h3}). 
We apply the idea of large deviation and introduce a small parameter $\epsilon$ via the following scaling:
\begin{equation}\label{scaling}
u^\epsilon_i(t,x)=u_i\left(\frac{t}{\epsilon},\frac{x}{\epsilon}\right), \quad i=1,2,3.
\end{equation}
Under the scaling above, we may rewrite the equation of $u_3$ in \eqref{eq:1-1} as
 \begin{align*}
\left \{
\begin{array}{ll}
\smallskip
\partial_t u_3^\epsilon=\epsilon \partial_{xx} u_3^\epsilon+\frac{r_3}{\epsilon}u_3^{\epsilon}(1-a_{31}u_1^\epsilon-a_{32}u_2^\epsilon-u_3^\epsilon) & \mathrm{in}~ (0,\infty)\times\mathbb{R},\\
u_3^\epsilon(0,x)=u_{3,0}(\frac{x}{\epsilon})  & \mathrm{on}~ \mathbb{R}.
\end{array}
\right.
\end{align*}

As discussed in the beginning of Section \ref{sec:prelim}, one can obtain the asymptotic behavior of $u_3^\epsilon$ as $\epsilon\rightarrow 0$ by considering the WKB-transformation, which is given by
\begin{equation}\label{eq:w3epsilon}
w_3^\epsilon(t,x)=-\epsilon\log{u_3^\epsilon(t,x)}{\color{red},}
\end{equation}
and satisfies the following equation.
\begin{equation}\label{eq:w3}
\left \{
\begin{array}{ll}
\smallskip
\partial_tw_3^\epsilon-\epsilon \partial_{xx} w_3^\epsilon+|\partial_x w_3^\epsilon|^2+r_3(1-a_{31}u_1^\epsilon-a_{32}u_2^\epsilon-u_3^\epsilon)=0 & \mathrm{in}~ (0,\infty)\times(0,\infty),\\
\smallskip
w_3^\epsilon(0,x)=-\epsilon\log{u_{3,0}(\frac{x}{\epsilon})}  & \text{on} ~[0,\infty),\\
w_3^\epsilon(t,0)=-\epsilon\log{u_{3}^\epsilon(t,0)}  & \text{on} ~[0,\infty).\\
\end{array}
\right.
\end{equation}

A link between $w_3^\epsilon$ and $u_3^\epsilon$ as $\epsilon \to 0$ is the following result, which is originally due to \cite{Evans_1989}. 
\begin{lemma}\label{lem:underlinec}
Let $K, K'$ be any compact sets such that $K \subset {\rm Int}\,K' \subset K'.$ If
$w_3^\epsilon \to 0$ uniformly on $K'$ as $\epsilon \to 0$, then
 \begin{equation*}
\liminf_{\epsilon\to 0} \inf_{K}
 u_3^\epsilon \geq 1-a_{31}\limsup_{\epsilon\to 0} \sup_{K'}u_1^\epsilon-a_{32}\limsup_{\epsilon\to 0} \sup_{K'}u_2^\epsilon.
\end{equation*}
In particular, $$\liminf_{\epsilon\to 0} \inf_{K}
 u_3^\epsilon \geq 1- a_{31} - a_{32} >0.$$
\end{lemma}
\begin{proof}
The proof is analogous to \cite[Lemma 3.1]{LLL20192} and we omit the details.
\end{proof}

Next, we apply the the half-relaxed limit method, due to Barles and Perthame \cite{BP1987}, to 
pass to the (upper and lower) limits of $w_3^\epsilon$. More precisely, we define
\begin{equation}\label{eq:w_star3}
w_3^*(t,x):=\hspace{-.5cm}\limsup\limits_{\scriptsize \begin{array}{c}\epsilon \to 0\\ (t',x') \to (t,x)\end{array}} \hspace{-.5cm}w_3^\epsilon (t',x')\,\,\, \, \mathrm{and}\, \, \,\,w_{3,*}(t,x):=\hspace{-.5cm}\liminf\limits_{\scriptsize \begin{array}{c}\epsilon \to 0\\ (t',x') \to (t,x)\end{array}} \hspace{-.5cm} w_3^\epsilon (t',x').
\mathfrak{\mathfrak{}}\end{equation}
\begin{remark}\label{rmk:w1w3}
Let $w_3^*$ and $w_{3,*}$ be defined in \eqref{eq:w_star3}. Then for any $ c\in\mathbb{R}$,
\begin{equation}\label{eq:w_star300}
  w_3^*(t,ct)=tw_3^*(1,c) \,\,\text{ and }\,\, w_{3,*}(t,ct)=tw_{3,*}(1,c). 
\end{equation}
Indeed, by 
\eqref{scaling} and \eqref{eq:w3epsilon}, the first equality in \eqref{eq:w_star300} is due to the following observation:
\begin{equation*}
\begin{split}
  tw_3^*(1,c)&=-t\hspace{-.5cm}\limsup\limits_{\scriptsize \begin{array}{c}\epsilon \to 0\\ (t',x') \to (1,c)\end{array}} \hspace{-.5cm}\left[\epsilon\log{u_3\left(\frac{t'}{\epsilon},\frac{x'}{\epsilon}\right)}\right]\\
  &=-\hspace{-.5cm}\limsup\limits_{\scriptsize \begin{array}{c}\epsilon \to 0\\ (t'',x'') \to (t,ct)\end{array}} \hspace{-.5cm}\left[(\epsilon t)\log{u_3\left(\frac{t''}{\epsilon t},\frac{x''}{\epsilon t}\right)}\right]=w_3^*(t,ct),
  \end{split}
\end{equation*}
where $(t'',x'')=(t',x')t$. The second equality in \eqref{eq:w_star300} follows by the same argument.
\end{remark}

The following lemma, for which the proof can be found in \cite[Lemma 3.2]{LLL2019} and \cite[Lemma 3.2]{LLL20192}, says that $w^*_3$ and $w_{3,*}$ are well-defined.
\begin{lemma}\label{lem:3-1}
Let $({\rm H}_{c_1,c_2,\lambda})$ hold for some $\lambda \in (0,\infty]$ and let $w_3^\epsilon$ be the solution of \eqref{eq:w3}.
\begin{itemize}
\item[{\rm(i)}] If ${\lambda\in(0,\infty)}$, then there exists some constant $Q>0$ independent of $\epsilon$  such that 
\begin{equation*}
  \max\{\lambda x- Q(t+\epsilon),0\} \leq {w}_3^{\epsilon}(t,x)
\leq\lambda x+Q(t+\epsilon)\,\,\text{ for }\,\,(t,x) \in [0,\infty)\times [0,\infty);
\end{equation*}
\item[{\rm(ii)}]  If $\lambda=\infty$, 
then  for each  compact subset $K\subset[(0,\infty)\times [0,\infty)]$, there is some constant $Q(K)$  independent of $\epsilon $ such that for any $(t,x) \in K$,
$$0 \leq {w}_3^{\epsilon}(t,x)\leq Q(K) \,\, \text{ for }\,\,\epsilon \in (0, 1/Q(K)].
$$
\end{itemize}
\end{lemma}

\begin{remark}\label{rmk:w1w20}
By comparison, we see that
$\lim_{t\to\infty} u_3(t,0)\geq (1-a_{31}-a_{32})/2>0,$
so that definition \eqref{eq:w_star3} implies
$w_3^*(t,0)=w_{3,*}(t,0)=0$ for $t\geq 0$. 
Also, if $\lambda \in (0,\infty)$, then  one can take $t=0$ and let $\epsilon \to 0$ in Lemma \ref{lem:3-1}{\rm(i)} to deduce
$w^*_3(0,x)=w_{3,*}(0,x)= \lambda  x$ for $x\geq0$.
\end{remark}

\subsection{Definitions and Preliminaries}\label{subsect:3.1}
Recall 
that, throughout this section, the assumption 
$({\rm H}_{c_1,c_2,\lambda})$ holds for some $c_1>c_2$ and $\lambda \in (0,\infty]$ (see Definition \ref{def:h3}). 
We proceed to define several quantities based on the parameters
$d_3,\, r_3, \,a_{21},\, a_{31}, \,a_{32},\, c_{\rm{LLW}}, 
\,c_1, \,c_2, \,\lambda.$ 
We list the objects, and where they are defined in Table \ref{tab:1} for quick reference later.

 \begin{table}[!htbp]
\centering
\caption{List of Auxiliary Objects}
\label{tab:1}
\vspace{-.2cm}
\centering
\footnotesize
\begin{tabular}{|c|c|c|c|}
\hline
\cline{1-4}
Object (s)&Defined ~in& Used in & Property \\
\hline
\cline{1-4}
$\alpha_3$& Section \ref{secbetamu3} & Section \ref{S3},\,\ref{S4},\,\ref{S5},\,\ref{S6} & $\alpha_3=2\sqrt{d_3r_3}$\\
$c_{\rm{LLW}}$& Definition \ref{c3LLW} & Section \ref{S3},\,\ref{S4},\,\ref{S6} &$ c_{\rm{LLW}}\leq \alpha_3$\\
$w_3^*, w_{3,*}$&  Section \ref{Section2}, \eqref{eq:w_star3}& Section \ref{Section2} &Remark \ref{rmk:w1w3}\\
$\rho^\mu_{\rm{nlp}}$& Section \ref{rhomu} & Section \ref{S3} & Lemma \ref{w3nllp}\\
$s^\mu_{{\rm nlp}}$&Section \ref{snlpmu}& Section \ref{S3} & \eqref{charact_wnlpmu}  \\ 
$\beta_3^\mu$& Section \ref{secbetamu3} & Section \ref{subsection5_1} &Lemma \ref{lem:betamu3}\\
$s^\mu(\hat{c})$& Section \ref{smuhatc} & Section \ref{subsection5_1} &
Lemma \ref{lemma:shatc}\\
$\nu^{\mu}_2(\hat c),\nu^{\mu}_3$& Section \ref{smuhatc}& Section \ref{subsection5_1} &Remark \ref{rmk:wmu3}\\
$\mathcal{E}$&Section \ref{subsection5_1}, \eqref{E} &  Proposition \ref{prop:estamoverc3} & Proposition \ref{prop:overc3} 
\\
$\rho_{\rm{nlp}}$ & Proposition \ref{underc3} & Section \ref{S4},\,\ref{S6} & $\rho_{\rm{nlp}}=\rho^1_{\rm{nlp}}$ for $\mu=1$ in \eqref{eq:w3nlpmu}\\
$s_{\rm{nlp}}$&Section \ref{S1}, \eqref{charact_snlp} & Section  \ref{subsection5_2},\,\ref{S4},\,\ref{S6} &Proposition \ref{charactsnlpc}\\ 
$\beta_3$& Section  \ref{subsection5_2} & Section  \ref{subsection5_2},\,\ref{S4}  &$\beta_3 = \max\{c_{\rm LLW}, s_{{\rm{nlp}}}\}$\\
$\underline{\beta}_3$& Section \ref{S1}, \eqref{definitionunderbata}& Section \ref{S4} &Proposition \ref{underc3}\\
\hline
\cline{1-4}
\end{tabular}
\end{table}

\subsubsection{Definition of $\rho^\mu_{{\rm nlp}}(s)$ for $\mu \in [0,1]$}\label{rhomu}

For each $\mu\in[0,1]$, we define function $\rho^\mu_{{\rm nlp}}:[0,\infty)\times [0,\infty)\rightarrow[0,\infty)$ as the unique viscosity solution of the variational inequality
\begin{align}\label{eq:w3nlpmu}
\left\{\begin{array}{l}
\smallskip
\min\{\rho-s\rho'+d_3|\rho'|^2+\mathcal{R}^\mu(s),\rho\}=0 \,\,\text{ in } (0,\infty),\\
\rho(0)=0,\quad\lim\limits_{s\to\infty}\frac{\rho(s)}{s}=\lambda,
\end{array}
\right.
\end{align}
where $\mathcal{R}^\mu(s)=r_3(1-\mu a_{31}\chi_{\{c_2< s\leq c_1\}}-a_{32}\chi_{\{s\leq c_2\}})$, and $\lambda\in(0,\infty]$ is given in $({\rm H}_{c_1,c_2,\lambda})$. 
The existence and uniqueness of $\rho^\mu_{{\rm nlp}}$ is guaranteed by  Lemma \ref{lemA1rho}. 
(When $\mu=0$, the species $u_1$ and $u_3$ do not compete. 
We will be using it as a starting case to bootstrap to $\mu=1$.)

\begin{lemma}\label{lemma_w3nlp}
For any $\mu\in[0,1]$, $\rho^\mu_{{\rm nlp}}(s)$ is Lipschitz continuous and  non-decreasing  with respect to $\mu\in [0,1]$, and is non-decreasing with respect to $s\in [0,\infty)$. Moreover,
\begin{equation}\label{eq:ssnlp}
\{s \geq 0: \rho^\mu_{\rm nlp}(s) = 0\} = [0,s_0]
\end{equation}for some $s_0 \geq 2\sqrt{d_3r_3(1-a_{31}-a_{32})}$ depending on  $\mu$ and $c_1, c_2$,  which are given in $({\rm H}_{c_1,c_2, \lambda})$.
\end{lemma}
\begin{proof}
\noindent {\bf Step 1.}  We prove the continuity and monotonicity of $\rho^\mu_{{\rm nlp}}$ with respect to $\mu$.
 Given any $0\leq\mu_1\leq \mu_2\leq 1$, let $\rho^{\mu_1}_{{\rm nlp}}$ and $\rho^{\mu_2}_{{\rm nlp}}$ be the viscosity solutions of \eqref{eq:w3nlpmu} with $\mu=\mu_1$ and $\mu=\mu_2$, respectively. It suffices to show that
\begin{equation}\label{goal0}
 0 \leq {\rho^{\mu_2}_{{\rm nlp}}(s)-\rho^{\mu_1}_{{\rm nlp}}(s)} \leq r_3a_{31}(\mu_2-\mu_1)\,\,\text{ for any } s\in[0,\infty).
\end{equation}

To this end, we first apply Lemma \ref{lemA1rho} with 
 $c_{g}=c_1$ to deduce that
$\rho^{\mu_1}_{{\rm nlp}}(s)=\rho^{\mu_2}_{{\rm nlp}}(s)$ for $s\in[c_1,\infty).$
It remains to prove \eqref{goal0} for $s\in[0,c_1]$. In such a case, $\rho^{\mu_1}_{{\rm nlp}}$ defines the unique viscosity solution of the problem
\begin{align}\label{eq:rhobound}
\left\{\begin{array}{l}
\smallskip
\min\{\rho-s\rho'+d_3|\rho'|^2+\mathcal{R}^{\mu_1}(s),\rho\}=0 \,\,\text{ in } (0,c_1),\\
\rho(0)=0,\quad\rho(c_1)=\rho^{\mu_2}_{{\rm nlp}}(c_1).
\end{array}
\right.
\end{align}
It is straightforward to check that ${\rho^{\mu_2}_{{\rm nlp}}}$ and {$\rho^{\mu_2}_{{\rm nlp}}-r_3a_{31}(\mu_2-\mu_1)$} are, respectively, viscosity super- and sub-solutions of \eqref{eq:rhobound}. Since the boundary conditions can be verified readily, by comparison arguments in Lemma \ref{lemA2rho},  we can deduce \eqref{goal0}. 

\smallskip
\noindent {\bf Step 2.} We show that $\rho^\mu_{{\rm nlp}}$ is non-decreasing in $s \in [0,\infty)$.
Suppose to the contrary that there exists some $s_0\in (0,\infty)$ such that $\rho^\mu_{{\rm nlp}}-0$ attains a local maximum at $s_0$ and $\rho^\mu_{{\rm nlp}}(s_0)>0$. By definition of viscosity solutions (see Definition \ref{defviscosity} and \cite[Proposition 3.1]{BarlesLect}), we have
$$0 \geq \rho^\mu_{{\rm nlp}}(s_0)-s_0\cdot 0+d_3|0|^2+\mathcal{R}^\mu(s_0)=\rho^\mu_{{\rm nlp}}(s_0)+\mathcal{R}^\mu(s_0),$$
which is a contradiction to $\mathcal{R}^\mu\geq0$. Step 2 is completed.

\smallskip
\noindent {\bf Step 3.} We show that $\rho^\mu_{\rm nlp}(s) \leq \max\{ \frac{s^2}{4d_3} - r_3(1-a_{31} - a_{32}),0\}$ for $s \in [0,\infty)$.

Observe that $\overline\rho_2(s):=  \max\{ \frac{s^2}{4d_3} - r_3(1-a_{31} - a_{32}),0\}$ is continuous, nonnegative,  and 
a classical super-solution for \eqref{eq:rhobound} whenever $s \not \in 2\sqrt{d_3r_3(1-a_{31}-a_{32})}$.  Let $\phi\in C^1(0,\infty)$ be any test function such that  $\overline{\rho}_2 - \phi$ attains a strict local minimum at  $\hat{s} = 2\sqrt{d_3r_3(1-a_{31}-a_{32})}$.
Then at $s={\hat{s}}$, direct calculation yields
\begin{align*}
\overline{\rho}_2(\hat s)-\hat s\phi' +d_3|\phi'|^2 + (\mathcal{R}^\mu)^*(\hat s)&=  -\hat s\phi' +d_3 | \phi'|^2 +  r_3(1 -a_{32})\\
&\geq  d_3\left[\phi'-\sqrt{\tfrac{r_3(1-a_{31}-a_{32})}{d_3}}\right]^2 \geq 0.
\end{align*}
Therefore, $\overline{\rho}_2$ defined above is a viscosity super-solution of \eqref{eq:rhobound}.
Observing also that
$$
\rho^\mu_{\rm nlp}(0) \leq \overline\rho_2(0) \quad\text{  and }\quad  \lim_{s\to\infty} \frac{\rho^\mu_{\rm nlp}(s)}{s} = \lambda \leq \infty =  \lim_{s\to\infty} \frac{\overline\rho_2(s)}{s},
$$
we apply comparison principle in Corollary \ref{corA3rho} to complete Step 3.

Finally, since $\rho^\mu_{\rm nlp}$ is nonnegative, non-decreasing in $s$ and $\rho^\mu_{\rm nlp}(0) = 0$, we deduce that \eqref{eq:ssnlp} holds for some $s_0 \geq 2\sqrt{d_3r_3(1-a_{31}-a_{32})} >0$.
\end{proof}

\subsubsection{Definition of $s^\mu_{\rm nlp}$ for $\mu \in [0,1]$}\label{snlpmu}
For each for $\mu \in [0,1]$, we define the speed $s^\mu_{{\rm nlp}}$ by
\begin{equation}\label{charact_wnlpmu}
  s^\mu_{{\rm nlp}}:= \sup \{s:\rho^\mu_{{\rm nlp}}(s) = 0\}.
\end{equation}
By Lemma \ref{lemma_w3nlp}, we  have $s^\mu_{{\rm nlp}}\in[2\sqrt{d_3r_3(1-a_{31}-a_{32})} ,\infty)$, and is non-increasing in $\mu$.
\subsubsection{Definition of $\alpha_3$ and $\beta^\mu_3$}\label{secbetamu3}
We define 
\begin{equation}\label{def:betamu}
\alpha_3 := 2\sqrt{d_3 r_3} \quad \text{ and }\quad  \beta^\mu_3:=\max\{s^\mu_{ {\rm nlp}},c_{\rm LLW}\} \quad \text{ for }\,\,\mu \in [0,1],
\end{equation}
where $c_{\rm LLW}$ is defined in Definition \ref{c3LLW}.

\begin{lemma}\label{lem:betamu3}
Let $\beta^\mu_3$ be defined by \eqref{def:betamu}. 
Then 
  $\beta^\mu_3 \leq \sigma_3<c_2$ for all $\mu \in [0,1]$.
\end{lemma}
\begin{proof}
Recall that $\sigma_3=d_3(\lambda \wedge \sqrt{r_3/d_3})+\frac{r_3}{\lambda \wedge\sqrt{r_3/d_3}}$. 
Since $c_{\rm LLW} \leq \alpha_3$ (see Definition \ref{c3LLW}) and $\sigma_3<c_2$ (see 
{\rm(i)} in Definition \ref{def:h3}),
it follows that $c_{\rm{LLW}}\leq\alpha_3\leq\sigma_3<c_2$. 

It remains to show $s^\mu_{ {\rm nlp}}\leq\sigma_3$.
To this end, we define ${\underline{\rho}_2}$ as the unique viscosity solution of
\begin{align}\label{eq:underlinrho1}
\left\{\begin{array}{l}
\smallskip
\min\{\rho-s\rho'+d_3|\rho'|^2+r_3,\rho\}=0 \,\,\text{ in } (0,\infty),\\
\rho(0)=0,\quad\lim\limits_{s\to\infty}\frac{\rho(s)}{s}=\lambda,
\end{array}
\right.
\end{align}
which is clearly a viscosity sub-solution of \eqref{eq:w3nlpmu}. We apply Corollary \ref{corA3rho} to deduce that
\begin{equation}\label{comparisonresult11}
  {\underline{\rho}_2}(s)\leq \rho^\mu_{ {\rm nlp}}(s) \quad \quad\text{for}\,\, s\in (0,\infty).
\end{equation}

We first consider the case $\lambda> \sqrt{r_3/d_3}$. In this case, $\sigma_3 = 2\sqrt{d_3 r_3}$ and $\lambda >
\frac{\sigma_3}{2d_3}$. A direct application of Lemma \ref{lemA1rho} for \eqref{eq:underlinrho1} with $c_g=\sigma_3$ and $g=0$ yields
$$
\underline{\rho}_2(s)=\frac{s^2}{4d_3}-r_3 = \frac{s^2 - (\sigma_3)^2}{4d_3} \quad \text{ for }s >\sigma_3\text{ and }s \approx \sigma_3.
$$ 
Thus, by \eqref{comparisonresult11} we arrive at $\rho^\mu_{\rm nlp}(s) \geq \underline{\rho}_2(s)>0$ for $s > \sigma_3$ such that $s \approx \sigma_3$. The definition of $s^\mu_{ {\rm nlp}}$  in \eqref{charact_wnlpmu} implies $s^\mu_{ {\rm nlp}}\leq\sigma_3$  as desired. 

It remains to consider the case $\lambda \leq \sqrt{r_3/d_3}$. In this case, 
$\sigma_3=d_3\lambda +\frac{r_3}{\lambda }$ and $\lambda \leq \frac{\sigma_3}{2d_3}$. We apply Lemma \ref{lemA1rho} for \eqref{eq:underlinrho1}  with $c_g=\sigma_3$ and $g=0$ 
again to deduce that 
$$
\underline{\rho}_2(s) 
=\lambda \left[s-(d_3\lambda +\frac{r_3}{\lambda })\right] = \lambda(s - \sigma_3) \quad \text{ for }s >\sigma_3\text{ and }s \approx \sigma_3.
$$
Hence, $\rho^\mu_{\rm nlp}(s) \geq \underline{\rho}_2(s)>0$ for $s > \sigma_3$ such that $s \approx \sigma_3$, which implies $s^\mu_{\rm nlp} \leq \sigma_3$.
%
\end{proof}

\begin{remark}\label{rmk:defbeta3}
Let $\mu=1$ in \eqref{eq:w3nlpmu}, \eqref{charact_wnlpmu} and \eqref{def:betamu}. It is easily seen that
$$\rho^1_{ {\rm nlp}}=\rho_{ {\rm nlp}},\,\,\,\, s^1_{ {\rm nlp}}=s_{{\rm nlp}}(c_1,c_2,\lambda)\,\,\,\text{ and }\,\,\,\beta^1_3=\beta_3:=\max\{c_{\rm LLW}, s_{\rm nlp}(c_1,c_2,\lambda)\},$$ 
where $\rho_{ {\rm nlp}}$ and $s_{{\rm nlp}}(c_1,c_2,\lambda)$ are defined in \eqref{eq:w3nlp} and \eqref{charact_snlp}. With this in mind, we drop the superscript $1$ in the notations $w^1_{ {\rm nlp}}$, $s^1_{ {\rm nlp}}$ and $\beta^1_3$ when we consider the case $\mu=1$.
\end{remark}

\begin{lemma}\label{w3nllp}
For any $\mu\in[0,1]$, let $\rho^\mu_{{\rm nlp}}$ be the unqiue viscosity solution of \eqref{eq:w3nlpmu}. If  $s^{\mu}_{{\rm nlp}}>\alpha_3\sqrt{1-a_{32}}$, then
\begin{equation*}
  \rho^\mu_{{\rm nlp}}(c_2)= \lambda^\mu_{{\rm nlp}}(c_2- s^\mu_{{\rm nlp}}),
\end{equation*}
where  $\lambda^\mu_{{\rm nlp}}=\frac{s^\mu_{{\rm nlp}} -\sqrt{(s^\mu_{{\rm nlp}})^2-\alpha_3^2(1-a_{32})}}{2d_3}$ and ${s}^\mu_{{\rm nlp}}$ is defined by \eqref{charact_wnlpmu}.
\end{lemma}
\begin{proof}
We only prove $\rho^\mu_{{\rm nlp}}(c_2)\leq \lambda^\mu_{{\rm nlp}}(c_2- s^\mu_{{\rm nlp}})$ in detail, and then $ \rho^\mu_{{\rm nlp}}(c_2)\geq \lambda^\mu_{{\rm nlp}}(c_2- {s^\mu_{{\rm nlp}}})$ follows from a similar argument. To this end, let us argue by contradiction, by assuming
$
 \rho^\mu_{{\rm nlp}}(c_2)>\lambda^\mu_{{\rm nlp}}(c_2- s^\mu_{{\rm nlp}}). 
$
By continuity there exists some $\hat{s}\in(\alpha_3\sqrt{1-a_{32}}, s^\mu_{{\rm nlp}} )$ such that
\begin{equation}\label{inequalityw3c}
  \rho^\mu_{{\rm nlp}}(c_2)> \nu_0(\hat{s})\cdot(c_2- \hat{s}),
\end{equation}
where  $\nu_0(\hat{s})=\frac{1}{2d_3}(\hat{s} -\sqrt{\hat{s}^2-\alpha_3^2(1-a_{32})})$. Note that \eqref{inequalityw3c} holds due to $\nu_0(\hat{s})\to\lambda^\mu_{{\rm nlp}}$  as $\hat s\to s^\mu_{{\rm nlp}}$.
By \eqref{inequalityw3c}, we can check 
$ \rho^\mu_{{\rm nlp}}$ is a viscosity super-solution of
\begin{align}\label{asubsolution}
\left\{
\begin{array}{ll}
\medskip
\min\{\rho-s\rho'+d_3|\rho'|^2+r_3(1-a_{32}),\rho\}=0 &\text{for}~ s\in(0,c_2),\\
\rho(0)=0, \,\,\,\, \rho(c_2)=\nu_0(\hat{s})\cdot(c_2- \hat{s}).
\end{array}
\right.
\end{align}
Define $\underline{\rho}_{\hat{s}}(s):=\max\{\nu_0(\hat{s}) \cdot(s- \hat{s}),0\}$.
It is straightforward to verify that $\underline{\rho}_{\hat{s}}$ is a viscosity sub-solution of \eqref{asubsolution}.  (It is in fact a viscosity solution of \eqref{asubsolution}.) By  Lemma \ref{lemA2rho} again, we have
$\rho^\mu_{{\rm nlp}}(s)\geq \underline{\rho}_{\hat{s}}(s)$ for $s\in[0, c_2]$. 
Therefore, we deduce that
$$[0,s^{\mu}_{\rm nlp}]= \{s:\rho^\mu_{{\rm nlp}}(s) = 0\}\subset\{s:\underline{\rho}_{\hat{s}}(s) = 0\}=[0, \hat{s}],$$
where the first equality follows from the definition \eqref{charact_wnlpmu} of $s^{\mu}_{\rm nlp}$. This implies  that $s^\mu_{{\rm nlp}}\leq \hat{s}$, a contradiction
to $\hat{s}\in(\alpha_3\sqrt{1-a_{32}}, s^\mu_{{\rm nlp}} )$. Lemma \ref{w3nllp} is thus proved.
\end{proof}

\subsubsection{Definition of $\underline{\rho}^\mu_{\ell}(s)$ for $\mu \in [0,1]$ and $\ell >0$}
For given $\ell>0$, we define
\begin{equation}\label{eq:nu1}
\nu_{1}(\ell):=\frac{1}{2d_3}\left(\ell + \sqrt{\ell^2 - \alpha_3^2(1-a_{32})}\right).
\end{equation}
 \begin{lemma}\label{lemma:sub-solution}
 For any $\mu\in[0,1]$ and $\ell',\ell$ such that $0<\ell<{\ell'}\leq c_2$, the function $\underline{\rho}^\mu_{\ell}: [\ell,\ell']\rightarrow[0,\infty)$ defined by
$\underline{\rho}^\mu_{\ell}(s):= \min\{  \rho^\mu_{\rm nlp}(s), \, \nu_{1}(\ell) \cdot (s-\ell)\}$
is a viscosity sub-solution of
 \begin{align}\label{underwell}
\left \{
\begin{array}{ll}
\medskip
\min\{\rho-s\rho'+d_3|\rho'|^2+r_3\left(1-a_{32}\right), \rho\}=0 &\text{in } ({\ell},{\ell'}),\\
\rho({{\ell}})=0,\,\,\quad\,\,\rho({\ell'})= \rho^\mu_{{\rm nlp}}({\ell'}),
\end{array}
\right.
\end{align}
where $\nu_1(\ell)$ is given by \eqref{eq:nu1}.
Furthermore, if $\nu_1(\ell) \cdot  (\ell' - \ell)\geq \rho^\mu_{{\rm nlp}}({\ell'})$, then $\underline{\rho}^{\mu}_{\ell}$ defines  the unique viscosity solution of \eqref{underwell}.
  \end{lemma}
\begin{proof}
 We first verify that $\underline{\rho}^{\mu}_{\ell}$ is a viscosity sub-solution of \eqref{underwell}. First, it is easy to see that $\rho^\mu_{\rm nlp}(s)$ and $\nu_1(\ell)(s-\ell)$ are viscosity solutions of the first equation of \eqref{underwell} and satisfies \eqref{underwell} wherever they are differentiable. 
  They are both Lipschitz continuous (as they satisfy $\rho - s \rho' + d_3 |\rho'|^2 \leq 0$ in viscosity sense so that their Lipschitz bounds are bounded locally \cite[Proposition 1.14]{IshiiLect}). By Rademacher's theorem, they are both differentiable a.e., and hence satisfy the first equation of \eqref{underwell} a.e.. Hence, $\underline{\rho}^{\mu}_{\ell}$ is also Lipschitz continuous and satisfies the first equation of \eqref{underwell} a.e.. By the convexity of the Hamiltonian, we can apply \cite[Proposition 5.1]{Bardi1997} to conclude that it is in fact a viscosity sub-solution of the first equation of \eqref{underwell}. Since it is clear that the boundary conditions are satisfied, $\underline{\rho}^{\mu}_{\ell}$ is a viscosity sub-solution of \eqref{underwell}.

If $\nu_1(\ell)\cdot({\ell'}- {{\ell}})\geq \rho^\mu_{{\rm nlp}}({\ell'})$,
then $\rho^\mu_{{\rm nlp}}$ and $\nu_1(\ell)\cdot(s-{{\ell}})$ are both viscosity super-solution of \eqref{underwell}. Upon taking their minimum, the resulting function $\underline{\rho}^{\mu}_{\ell}$ is also a viscosity super-solution of \eqref{underwell} (see, e.g. \cite[Proof of Theorem 7.1]{BarlesLect}). Since $\underline{\rho}^{\mu}_{\ell}$ is already a  sub-solution, it is therefore a viscosity solution. Finally, the uniqueness 
follows from Lemma \ref{lemA1rho}.
%
\end{proof}

\subsubsection{Definition of $s^\mu(\hat{c})$, $\nu^{\mu}_2(\hat{c})$, and $\nu^{\mu}_3$}\label{smuhatc}
For given $\mu \in [0,1]$ and $\hat{c} \in (\beta^\mu_3, c_2]$, we define
\begin{equation}\label{eq:smu}
s^{\mu}(\hat{c}):= 
\begin{cases}
\smallskip
d_3\nu^{\mu}_2(\hat{c})  + \frac{r_3 (1-a_{32}(1-a_{21})}{\nu^{\mu}_2(\hat{c})}  &\text{ if }\nu^{\mu}_2(\hat{c})\leq \frac{r_3 (1-a_{32}(1-a_{21})}{d_3\nu^{\mu}_3}\text{ and }  \nu^{\mu}_2(\hat{c}) \leq \nu^{\mu}_3,\\
\beta^\mu_3 &\text{ otherwise},
\end{cases}
\end{equation}
where $\beta^\mu_3< c_2$ (as proved in Lemma \ref{lem:betamu3}) and
\begin{equation}\label{eq:nu2}
\nu^{\mu}_2(\hat{c}): = 
\begin{cases}
\smallskip
 \frac{1}{2d_3}\left\{\hat{c} - \sqrt{\hat{c}^2 - 4d_3[r_3(1-a_{32}(1-a_{21})) + \rho^\mu_{\rm nlp}(\hat{c})]}\right\} \quad \\
  \qquad   \qquad   \qquad \text{ if }\hat{c}^2 \geq  4d_3[r_3(1-a_{32}(1-a_{21})) + \rho^\mu_{\rm nlp}(\hat{c})],\\
 \infty  \qquad \qquad \text{ otherwise},
\end{cases}
\end{equation}
and
\begin{equation}\label{qq:3.20}
\nu^{\mu}_3 := \frac{1}{2d_3} \left[  \beta^{\mu}_3 + \sqrt{(  \beta^{\mu}_3)^2 - \alpha_3^2(1-a_{32}(1-a_{21}))}\right],
\end{equation}
the latter is well-defined since $\beta^\mu_3 \geq c_{\rm{LLW}} \geq \alpha_3\sqrt{1-a_{32}(1-a_{21})}$ due to \eqref{def:betamu}.
\begin{remark}\label{rmk:wmu3}
By  the construction of $\nu^{\mu}_2(\hat{c})$ and $\nu^{\mu}_3$, we can rewrite $\beta^\mu_3$ as
\begin{equation}\label{betamu3}
 \beta^\mu_3=d_3 \nu^{\mu}_3+\frac{r_3(1-a_{32}(1-a_{21}))}{\nu^{\mu}_3}.
\end{equation}
In case $\hat{c}^2\geq 4d_3[r_3(1-a_{32}(1-a_{21})) + \rho^\mu_{\rm nlp}(\hat{c})]$, we can rewrite $\rho^\mu_{{\rm nlp}}(\hat c)$ as
\begin{equation}\label{wmu3nlplambda1}
  \rho^\mu_{{\rm nlp}}(\hat c)=\hat c\nu^{\mu}_2(\hat c)-d_3\left(\nu^{\mu}_2(\hat c)\right)^2-r_3(1-a_{32}(1-a_{21})).
\end{equation}
Furthermore, if $s^\mu(\hat{c})> \beta^\mu_3$, then it follows from \eqref{eq:smu} and \eqref{wmu3nlplambda1} that
\begin{equation}\label{wmu3nlplambda2}
  \rho^\mu_{{\rm nlp}}(\hat c)=\nu^{\mu}_2(\hat c)\cdot(\hat c-s^\mu(\hat{c})).
\end{equation}
\end{remark}
\begin{lemma}\label{lemma:shatc}
Let $\hat{c} \in (\beta^\mu_3, c_2]$ and $s^\mu(\hat{c})$ be defined by \eqref{eq:smu}. Then 
$s^\mu(\hat{c})\in[\beta^\mu_3,\hat c)$.
\end{lemma}

\begin{proof}
 First, we show $s^\mu(\hat c)  \geq \beta^\mu_3$.
By definition \eqref{eq:smu}, it suffices to verify  $s^\mu(\hat{c})\geq \beta^\mu_3$ when $\nu^{\mu}_2(\hat c)\leq\nu^{\mu}_3\text{ and }\nu^{\mu}_2(\hat c)\leq\frac{r_3(1-a_{32}(1-a_{21}))}{d_3\nu^{\mu}_3}$. In such a case, by \eqref{betamu3}, direct calculation yields
\begin{align*}
s^\mu(\hat{c})- \beta^\mu_3=&d_3(\nu^{\mu}_2(\hat c)-\nu^{\mu}_3) + r_3(1-a_{32}(1-a_{21}))\left[\frac{1}{\nu^{\mu}_2(\hat c)}- \frac{1}{\nu^{\mu}_3}\right]\\
=&\frac{\nu^{\mu}_2(\hat c)-\nu^{\mu}_3}{\nu^{\mu}_2(\hat c)}\left[d_3\nu^{\mu}_2(\hat c)-\frac{r_3(1-a_{32}(1-a_{21}))}{\nu^{\mu}_3}\right]\geq 0,
\end{align*}
which proves $s^\mu(\hat c)  \geq \beta^\mu_3$.

It remains to show $s^\mu(\hat c) < \hat{c}$.  Since $\hat c \in (\beta^\mu_3, c_2]$, there is nothing to prove in case $s^\mu(\hat c) = \beta^\mu_3$. Next, assume $s^\mu(\hat c) > \beta^\mu_3$.
Since $\hat c>\beta^\mu_3\geq s^\mu_{{\rm nlp}}$ by the definition in \eqref{def:betamu}, using \eqref{charact_wnlpmu} we have $ \rho^\mu_{{\rm nlp}}(\hat c)>0$. In view of \eqref{wmu3nlplambda2},  $ \rho^\mu_{{\rm nlp}}(\hat c)>0$ implies that $s^\mu(\hat{c})<\hat c$.
\end{proof}

%



\subsection{Estimating $\overline{c}_3$ from above}\label{S3}
The purpose of this subsection is to prove $\overline c_3\leq \max\{s_{{\rm nlp}},c_{\rm{LLW}}\}$ as stated in \eqref{estamitec3}. 
Recall that throughout the section, we have fixed  $d_i, r_i, a_{ij}$ and the initial conditions $u_{i,0}$  in such a way that 
$({\rm H}_{c_1,c_2,\lambda})$ hold for some $c_1>c_2$ and $\lambda \in (0,\infty]$. 

\subsubsection{Estimating $\overline{c}_3$ for given $\mu\in [0,1]$}\label{subsection5_1}
In this subsection, we show that $\overline{c}_3 \leq \beta^\mu_3$ for any $\mu\in[0,1]$ satisfying $w_{3,*}(1,c_2)\geq  \rho^\mu_{{\rm nlp}}(c_2)$, where $\beta^\mu_3$ is given in \eqref{def:betamu} and $w_{3,*}$ is defined by \eqref{eq:w_star3}.  
See Proposition \ref{prop:overc3_1} below.

\begin{lemma}\label{lemma:5-4} 
Let $(u_i)_{i=1}^3$  be any solution of \eqref{eq:1-1} such that $({\rm H}_{c_1,c_2,\lambda})$ holds.
 Fix any $\hat c\in(\beta^\mu_3,c_2]$ and $\mu\in[0,1]$. Suppose that
 \begin{equation}\label{condition216}
   w_{3,*}(1,\hat c)\geq \rho^\mu_{{\rm nlp}}(\hat c)>0.
 \end{equation}
Then we have
$\overline{c}_3\leq s^\mu(\hat{c})$, where $s^\mu(\hat{c})$ is defined by \eqref{eq:smu}.
\end{lemma}

\begin{proof}
Observe from \eqref{condition216} that $w_{3,*}(1,\hat{c})>0$ so that (by  definition \eqref{eq:w_star3} of $w_{3,*}$) we have $u_3(t,\hat{c}t) \to 0$ as $t\to\infty$, i.e. $\hat c\in (\overline{c}_3,c_2]$.
By \eqref{eq:H.c1c2}, 
we can choose  a sequence $\hat{c}_j\in (\overline{c}_3,c_2)$ such that
 $\hat{c}_j\to\hat{c}$ as $j \to \infty$ and
\begin{equation}\label{conditionc_j}
\lim\limits_{t\to\infty} u_2(t,\hat{c}_jt)\geq\frac{1-a_{21}}{2}\, \quad \text{ for each }\,\,\,j \in \mathbb{N}.
\end{equation}

%
Fix $j \in \mathbb{N}$ large such  that $\hat\mu_j :=w_{3,*}(1,\hat{c}_j)\wedge \rho^\mu_{{\rm nlp}}(\hat{c}_j)>0$. Denote by $(\underline u_2,\overline u_3)$ the unique solution of the problem
\begin{equation}\label{eq:subu2u3}
\left\{
\begin{array}{ll}
\medskip
\partial_t \underline u_2-\partial_{xx}\underline u_2=\underline u_2(1-a_{21}-\underline u_2-a_{23}\overline u_3) &\text{ for }0<x<\hat c_jt,\,t>t_0,\\
\medskip
\partial_t \overline u_3-d_3\partial_{xx}\overline u_3=r_3\overline u_3(1-a_{32} \underline u_2- \overline u_3) & \text{ for } 0<x<\hat c_jt,\,t>t_0,
\end{array}
\right.
\end{equation}
with the initial-boundary condition
$$\underline u_2 = \min\left\{u_2, \frac{1-a_{21}}{2}\right\} \,\,\text{ and }\,\, \overline u_3 = u_3 \,\,\text{ on }\,\, \partial\{(t,x):\, t>t_0,\, x \in \{0, \hat{c}_j t\}\}.$$
In view of \eqref{conditionc_j}, we have $ \lim\limits_{t\to\infty}\underline u_2(t,\hat{c}_jt)=\frac{1-a_{21}}{2}$.
 Obviously,
 $(u_2,u_3)$ defined by \eqref{eq:1-1} is a classical  super-solution of \eqref{eq:subu2u3}, so that by comparison we derive that
$$ u_2\geq \underline u_2\,\,\text{ and } \,\, u_3\leq \overline u_3 \quad 
\text{ for } 0\leq x\leq\hat c_jt,\,\,t\geq t_0.$$

 By the definition of $w_{3,*}$ in \eqref{eq:w_star3} and $w^\epsilon_3(1,\hat{c}_j)=-\epsilon \log{u_3^\epsilon(1,\hat{c}_j)}$, for  small $\epsilon>0$, 
 we have
\begin{equation*}
-\epsilon\log{u_3\left(\frac{1}{\epsilon},\frac{\hat{c}_j}{\epsilon}\right)}\geq w_{3,*}(1,\hat{c}_j)+o(1) \geq  \hat\mu_j+o(1),
\end{equation*}
that is
$$u_3\left(\frac{1}{\epsilon},\frac{\hat{c}_j}{\epsilon}\right) \leq \exp\left(-\frac{\hat\mu_j + o(1)}{\epsilon}\right).  
$$
Since $\overline{u}_3(t,\hat{c}_j t) = u_3(t,\hat{c}_j t)$ for all $t$, this implies
\begin{equation}\label{condition1}
  \overline{u}_3(t,\hat{c}_jt)=u_3(t,\hat{c}_jt)\leq \exp\{-(\hat\mu_j+o(1))t\} \quad \text{ for } t\gg1.
\end{equation}
We can apply Lemma \ref{lem:appen1} to $(\underline u_2,\overline u_3)$ to yield
 $$
\lim_{t\to\infty} \sup_{ct<x<\hat{c}_jt}  u_3(t,x)\leq\lim_{t\to\infty} \sup_{ct<x<\hat{c}_j t}  \overline{u}_3(t,x) = 0  \quad \text{ for each }c >  s_{\hat{c}_j}.
$$
Here $s_{\hat{c}_j}$ can be expressed by 
\begin{equation*}
 s_{\hat{c}_j}=
\begin{cases}
\medskip
 c_{\rm{LLW}} & \text{ if } \hat\mu_j\geq \lambda_{\rm LLW}(\hat{c}_j-  c_{\rm{LLW}}),\\
 \hat{c}_j-\frac{2d_3\hat\mu_j}{\hat{c}_j-\sqrt{\hat{c}_j^2-4d_3[\hat\mu_j+r_3(1-a_{32}(1-a_{21}))]}} &  \text{ if } \hat\mu_j< \lambda_{\rm LLW}(\hat{c}_j- c_{\rm{LLW}}),
\end{cases}
\end{equation*}
where $0<\lambda_{\rm LLW}\leq \sqrt{r_3(1-a_{32}(1-a_{21}))/d_3}$ is the smaller positive root of $\lambda c_{\rm{LLW}} - d_3\lambda^2 - r_3(1-a_{31}(1-a_{21}))=0$ (see Remark \ref{rmk:lambdaLLW}).
Hence,  $\overline{c}_3 \leq s_{\hat{c}_j}$ for all $j \gg 1$. 
Recalling that $\hat\mu_j :=w_{3,*}(1,\hat{c}_j)\wedge \rho^\mu_{{\rm nlp}}(\hat{c}_j)>0$ and that
 $w_{3,*}(1,\hat c)\geq \rho^\mu_{{\rm nlp}}(\hat c)$, we arrive at $\hat\mu_j\to\rho^\mu_{{\rm nlp}}(\hat{c})$ as $j \to \infty$ (note that $\rho^\mu_{{\rm nlp}}$ is continuous and $w_{3,*}$ is lower semicontinuous).
Therefore, letting $j \to \infty$, we obtain $\overline{c}_3 \leq s_{\hat{c}}$. Here $s_{\hat{c}}$ is given by
\begin{equation}\label{eq:chatc}
 s_{\hat{c}}=
\begin{cases}
\medskip
 c_{\rm{LLW}} & \text{ if }  \rho^\mu_{{\rm nlp}}(\hat{c})\geq \lambda_{\rm LLW}(\hat{c}-  c_{\rm{LLW}}),\\
\hat{c}-\frac{ \rho^\mu_{{\rm nlp}}(\hat{c})}{\nu^{\mu}_2(\hat c)} &  \text{ if } \rho^\mu_{{\rm nlp}}(\hat{c})< \lambda_{\rm LLW}(\hat{c}- c_{\rm{LLW}}),
\end{cases}
\end{equation}
where we used  the definition \eqref{eq:nu2} of  $\nu^{\mu}_2(\hat c)$.
It remains to verify $ s_{\hat{c}}\leq s^\mu(\hat{c})$.

 If $\rho^\mu_{{\rm nlp}}(\hat c)\geq \lambda_{\rm LLW}(\hat c-  c_{\rm{LLW}})$, then by \eqref{eq:chatc} we obtain $s_{\hat{c}}=c_{\rm{LLW}}$.
Since $s^\mu(\hat{c})\geq \beta^\mu_3\geq c_{\rm{LLW}}$ by    Lemma \ref{lemma:shatc} and \eqref{def:betamu}, we have $s_{\hat{c}} = c_{\rm{LLW}} \leq s^\mu(\hat{c})$.

It remains to prove $s_{\hat{c}}\leq s^\mu(\hat{c})$ if $\rho^\mu_{{\rm nlp}}(\hat c)< \lambda_{\rm LLW}(\hat c-  c_{\rm{LLW}})$.
We use Remark \ref{rmk:lambdaLLW} to derive
\begin{align}\label{eq:uuuu}
\rho^{\mu}_{\rm nlp}(\hat{c})&<\lambda_{\rm LLW}(\hat c-  c_{\rm{LLW}}) \notag \\
&=\lambda_{\rm LLW}\hat{c} - d_3 \lambda_{\rm LLW}^2 - r_3(1-a_{32}(1-a_{21})).
\end{align}
Completing the square, $\lambda_{\rm LLW}\hat{c} - d_3 \lambda_{\rm LLW}^2 \leq \frac{\hat{c}^2}{4d_3}$ so that  we arrive at
$\rho^{\mu}_{\rm nlp}(\hat{c})< \frac{\hat{c}^2}{4d_3} -r_3(1-a_{32}(1-a_{21}))$, 
i.e. $\hat{c}^2> 4d_3[r_3(1-a_{32}(1-a_{21})) + \rho^\mu_{\rm nlp}(\hat{c})]$,
whence we can invoke \eqref{wmu3nlplambda1} and the second part of \eqref{eq:chatc} to derive that
\begin{equation}\label{eq:5-9}
 s_{\hat{c}}= \frac{\hat c \nu^\mu_2(\hat c) - \rho^\mu_{\rm nlp}(\hat c)}{\nu^\mu_2(\hat c)}=
d_3\nu^{\mu}_2(\hat c) + \frac{r_3(1-a_{32}(1-a_{21}))}{\nu^{\mu}_2(\hat c)}.
\end{equation}

Next, we claim that 
\begin{equation}\label{lambda1.1}
\nu^\mu_2(\hat{c}) < \lambda_{\rm LLW} \leq \frac{\beta^\mu_3}{2d_3}.
\end{equation}
Indeed, by \eqref{eq:nu2} and Remark \ref{rmk:lambdaLLW}, respectively, we have
\begin{equation}\label{lambda1}
0<\nu^{\mu}_2(\hat c) \leq \frac{\hat{c}}{2d_3} \quad \text{ and }\quad
0<\lambda_{\rm LLW}
\leq \frac{c_{\rm{LLW}}}{2d_3}\leq \frac{\beta_3^\mu}{2d_3} \leq \frac{\hat{c}}{2d_3},
\end{equation}
This yields the second inequality of \eqref{lambda1.1}. Next, we 
compare \eqref{eq:uuuu} and \eqref{wmu3nlplambda1} to obtain
\begin{equation}\label{lambda3LLW}
  \hat c\nu^{\mu}_2(\hat c)-d_3\left(\nu^{\mu}_2(\hat c)\right)^2< \lambda_{\rm LLW}\hat c-d_3\lambda_{\rm LLW}^2.
\end{equation}
Since \eqref{lambda1} says that 
$\nu^{\mu}_2(\hat c)$ and $\lambda_{\rm LLW}$ belong to the interval $I=(0, \frac{\hat{c}}{2d_3}]$, on which $s \mapsto\hat{c}s - d_3 s^2$ is monotone,
this completes the proof of \eqref{lambda1.1}.

Since $\nu^{\mu}_3 \geq \frac{\beta^{\mu}_3}{2d_3}$ (see \eqref{qq:3.20}), we deduce from \eqref{lambda1.1} that
%
\begin{equation}\label{eq:3.15b}
\nu^{\mu}_2(\hat c)\leq\nu^{\mu}_3.\end{equation}
Next, we verify $s_{\hat{c}}\leq s^\mu(\hat{c})$ by dividing into the following two cases:
\begin{itemize}
  \item [{\rm (i)}] If $\nu^{\mu}_2(\hat c)\leq\frac{r_3(1-a_{32}(1-a_{21}))}{d_3\nu^{\mu}_3}$, then since $\nu^{\mu}_2(\hat{c}) \leq \nu_3^{\mu}$ in \eqref{eq:3.15b},  it follows from \eqref{eq:smu} and
\eqref{eq:5-9} that
 $s_{\hat{c}}$ exactly equals $s^\mu(\hat{c})$;
  \item [{\rm (ii)}] If $\nu^{\mu}_2(\hat c) >\frac{r_3(1-a_{32}(1-a_{21}))}{d_3\nu^{\mu}_3}$, then
$s^\mu(\hat c)  = \beta^\mu_3$ by \eqref{eq:smu}.
We directly calculate that
\begin{align*}
 s^\mu(\hat{c})- s_{\hat{c}} 
&=d_3(\nu^{\mu}_3-\nu^{\mu}_2(\hat c))+ r_3(1-a_{32}(1-a_{21}))\left[\frac{1}{\nu^{\mu}_3}-\frac{1}{\nu^{\mu}_2(\hat c)}\right]\\
&=\frac{\nu^{\mu}_3-\nu^{\mu}_2(\hat c)}{\nu^{\mu}_2(\hat c)}\left[d_3\nu^{\mu}_2(\hat c)- \frac{r_3(1-a_{32}(1-a_{21}))}{\nu^{\mu}_3}\right]\geq 0,
\end{align*}
where we used \eqref{betamu3} and \eqref{eq:5-9} for the first equality, and used $\frac{r_3(1-a_{32}(1-a_{21}))}{d_3\nu^{\mu}_3}<\nu^{\mu}_2(\hat c)\leq \nu^{\mu}_3$ for the last inequality.
\end{itemize}
The proof is thereby completed.
 \end{proof}

To set up the proof by continuity, 
let $\rho^\mu_{{\rm nlp}}$ be the unique solution of \eqref{eq:w3nlpmu}, and define
\begin{equation}\label{E}
    \mathcal{E}:=\left\{\mu\in[0,1]:w_{3,*}(1,c_2)\geq  \rho^\mu_{{\rm nlp}}(c_2)\right\},
\end{equation}
We establish in the next two propositions that $\overline{c}_3\leq \beta^\mu_3$ for all $\mu\in\mathcal{E}$.
\begin{proposition}\label{prop:overc3}
If $\mu \in \mathcal{E}$, then  either $\overline{c}_3\leq \beta^\mu_3$ or $ w_{3,*}(1,c) \geq \rho^{\mu}_{\rm nlp}(c)$ for all $c \in [\beta^{\mu}_3, c_2]$.
\end{proposition}
\begin{proof}
Fix $\mu \in \mathcal{E}$ and define
\begin{equation}\label{D}
    \mathcal{D}_{\mu}:=\left\{c'\in[\beta^\mu_3,c_2]:w_{3,*}(1,c)\geq  \rho^\mu_{{\rm nlp}}(c)
    \,\,\text{ for all }\,\,c\in[c',c_2]
    \right\}.
\end{equation}
First we observe that $\mathcal{D}_{\mu}$ is closed, since $\rho^\mu_{{\rm nlp}}$ is continuous and $w_{3,*}$ is lower semicontinuous. 
Also, $\mathcal{D}_{\mu}$ is non-empty by the hypothesis $c_2\in\mathcal{D}_{\mu}$ (which is in fact equivalent to $\mu \in \mathcal{E}$).
Define $\hat c:=\inf\mathcal{D}_{\mu}$, then $\hat{c} \in \mathcal{D}_{\mu}$ and $\hat{c} \in [\beta^\mu_3, c_2]$. Suppose to the contradiction that Proposition \ref{prop:overc3} fails. Then we have  $\hat c \in (\beta^\mu_3,c_2]$ and $\overline{c}_3> \beta^\mu_3$. 

\smallskip
\noindent {\bf Step 1.} We show that $\overline{c}_3\leq s^\mu(\hat c)$, where $s^\mu(\hat c)\in[\beta^\mu_3,c_2)$ is defined by \eqref{eq:smu}. Taking $\overline{c}_3 > \beta^\mu_3$ into account, this implies in particular  $s^\mu(\hat c)>\beta^\mu_3$.

Since $\hat c\in (\beta^\mu_3,c_2]$ and $\hat c\in\mathcal{D}_{\mu}$, we see that $w_{3,*}(1,\hat c)\geq \rho^\mu_{{\rm nlp}}(\hat c)>0$.
(To see that
the last  term is positive, note that $\beta^\mu_3 \geq s^\mu_{\rm nlp}$ by definition, so that $\hat c>\beta^\mu_3\geq s^\mu_{{\rm nlp}}$. Hence $\rho^\mu_{\rm nlp}(\hat c)>0$ follows from the definition of $s^\mu_{\rm nlp}$ in \eqref{charact_wnlpmu}.)
Then we may apply Lemma \ref{lemma:5-4} to deduce $\overline{c}_3\leq s^\mu(\hat c)$. 
This completes Step 1.

To derive a contradiction to $\hat{c}=\inf \mathcal{D}_\mu$, we will
 find some $\delta=\delta(\hat c)>0$ such that $\hat c-\delta\in\mathcal{D}_{\mu}$
in the following three steps.

\smallskip

\noindent {\bf Step 2.} We show that $w_{3,*}(1,s)\geq {\rho}_{1}(s)$ for all $s\in [s^\mu(\hat c),\hat c ]$, where ${\rho}_{1}$ defines the unique viscosity solution of (for uniqueness see Lemma \ref{lemA2rho})
 \begin{align}\label{underlineW}
\left\{\begin{array}{ll}
\medskip
\min\{\rho-s\rho'+d_3|\rho'|^2+r_3\left(1-a_{32}\right), \rho\}=0 &\mathrm{for}~s\in (s^\mu(\hat c),\hat c),\\
\rho(s^\mu(\hat c))=0,\,\,\,\,\rho(\hat c)= \rho^\mu_{{\rm nlp}}(\hat c ).
\end{array}
\right.
\end{align}
 By Step 1, we have $\overline{c}_3\leq s^\mu(\hat c)$. Thus applying \eqref{eq:H.c1c2} 
 yields
$$ \liminf_{\substack{(t',x')\to (t,x)\\ \epsilon\to 0}}u_1^\epsilon(t',x')\geq \chi_{\{c_2t<x<c_1t\}}\,\,\text{ and }\,\,
\liminf_{\substack{(t',x')\to (t,x)\\ \epsilon\to 0}} u_2^\epsilon(t',x')\geq\chi_{\{s^\mu(\hat c)t<x<c_2t\}}.
$$
Letting $\epsilon \to 0$ in \eqref{eq:w3},  
it is standard \cite[Sect. 6.1]{BarlesLect} (also \cite[Propositions 3.1 and  3.2]{Barles_1990})  to see that $w_{3,*}$ is a viscosity super-solution of
\begin{equation}\label{eq:equ}
  \min\left\{\partial_t w+d_3|\partial_xw|^2+\overline{\mathcal{R}}_3(x/t),w\right\}=0 \,\,\,\,\mathrm{in }~ (0,\infty)\times\mathbb{R},
\end{equation} where
$\overline{\mathcal{R}}_3(s)=r_3(1-a_{31}\chi_{\{c_2<s< c_1\}}-a_{32}\chi_{\{s^\mu(\hat c)<s<c_2\}}).$
(We note that \eqref{eq:spreadingly111} in Proposition \ref{thm:u1v1111} 
is used crucially in the derivation, to deal with the discontinuity of  Hamiltonian function.)

We claim that $w_{3,*}$ is also a viscosity super-solution of
\begin{align}\label{underlineW0}
\left\{\begin{array}{ll}
\medskip
\min\{\partial_t w+d_3|\partial_x w|^2+r_3\left(1-a_{32}\right), w\}=0 &\text{for}~ s^\mu(\hat c)t<x<\hat c t,\\
w(t,s^\mu(\hat c)t)=0,\,\,\,\,w(t,\hat ct)=t \rho^\mu_{{\rm nlp}}(\hat c ) &\text{for}~ t\geq 0.
\end{array}
\right.
\end{align}
First, we check the boundary conditions. Indeed, it follows that
$w_{3,*}(t,s^\mu(\hat c)t)\geq0$ and
$w_{3,*}(t,\hat ct) = tw_{3,*}(1,\hat{c})\geq t \rho^\mu_{{\rm nlp}}(\hat c)$, where the first equality is due to \eqref{eq:w_star300} in Remark \ref{rmk:w1w3}, and the last inequality is due to $\hat{c} \in \mathcal{D}_\mu$. Next, observe that the first part of
\eqref{underlineW0} is the restriction of \eqref{eq:equ} to a subdomain, as
 $\overline{\mathcal{R}}_3(t,x)=r_3\left(1-a_{32}\right)$ when $s^\mu(\hat c)t<x<\hat ct.$
As a result, $w_{3,*}$, being a super-solution of \eqref{eq:equ}, automatically qualifies as a super-solution of \eqref{underlineW0}. Then
we  apply Proposition \ref{lemA3rho}, which exploits the connection between \eqref{underlineW} and \eqref{underlineW0}, to deduce
$$w_{3,*}(t,x)\geq t{\rho}_{1}\left(x/t\right)\quad \text{for }\,\, s^\mu(\hat c)t\leq x\leq\hat c t,$$
so that $w_{3,*}(1,s)\geq {\rho}_{1}(s)$ for all $s\in [s^\mu(\hat c),\hat c ]$.
Step 2 is thus completed.

\smallskip
\noindent {\bf Step 3.} To proceed further,  we show
\begin{equation}\label{goalinequality}
 0<\rho^\mu_{{\rm nlp}}(\hat c)<\nu_4^{\mu}(\hat{c}) \cdot (\hat c -s^\mu(\hat c)),
\end{equation}
where we define (consistently with definition of $\nu_1(\ell)$ in \eqref{eq:nu1})
\begin{equation}\label{eq:3.20b}
\nu_4^{\mu}(\hat{c}):=\nu_1(s^\mu(\hat c))=\frac{s^\mu(\hat c) +\sqrt{(s^\mu(\hat c))^2-\alpha^2_3(1-a_{32})}}{2d_3}.
\end{equation}
Since $s^\mu(\hat c)>\beta^\mu_3$ according to Step 1, the first alternative in \eqref{eq:smu} holds, and we deduce  
$$
s^\mu(\hat c)=d_3\nu^{\mu}_2(\hat c) + \frac{r_3(1-a_{32}(1-a_{21}))}{\nu^{\mu}_2(\hat c)}, \quad  \nu^{\mu}_2(\hat c)\leq\nu^{\mu}_3\,\, \text{ and } \,\, \nu^{\mu}_2(\hat c)\leq\frac{r_3(1-a_{32}(1-a_{21}))}{d_3\nu^{\mu}_3}.$$ This implies
$\nu^{\mu}_2(\hat c)\leq\frac{r_3(1-a_{32}(1-a_{21}))}{d_3\nu^{\mu}_2(\hat c)}=\frac{s^\mu(\hat c)}{d_3}-\nu^{\mu}_2(\hat c)$, so that by \eqref{eq:3.20b} we derive that
\begin{equation*}
  \nu^{\mu}_2(\hat c)\leq\frac{s^\mu(\hat c)}{2d_3}<\nu_4^{\mu}(\hat{c}).
\end{equation*}
This, together with \eqref{wmu3nlplambda2} in Remark \ref{rmk:wmu3}, implies 
\begin{equation*}
0<\rho^\mu_{{\rm nlp}}(\hat c)=\nu^{\mu}_2(\hat c)\cdot(\hat c-s^\mu(\hat c))<\nu_4^{\mu}(\hat{c})\cdot(\hat c -s^\mu(\hat c)).
\end{equation*}
(Note that  $\rho^\mu_{{\rm nlp}}(\hat c)>0$ since $\hat{c} > \beta^{\mu}_3 \geq s^{\mu}_{\rm nlp}$ as in Step 1.) We have proved
 \eqref{goalinequality}.

\smallskip
 \noindent {\bf Step 4.} We show that there exists some $\delta>0$ such that $\hat c-\delta\in\mathcal{D}_{\mu}$, which contradicts $\hat c=\inf\mathcal{D}_{\mu}$ and  completes the proof of Proposition \ref{prop:overc3}.

First, we apply Lemma \ref{lemma:sub-solution} with ${\ell'}=\hat c$, ${{\ell}}=s^\mu(\hat c)$ to conclude that
\begin{equation}\label{inequalityunderw}
{\rho}_{3}(s):=\min\left\{ \rho^\mu_{{\rm nlp}}(s), \,\nu_4^{\mu}(\hat{c})\cdot(s-s^\mu(\hat c))\right\}
\end{equation}
 is a viscosity sub-solution of \eqref{underlineW}, where $\nu_4^{\mu}(\hat{c})=\nu_1(s^\mu(\hat c))$ is defined in \eqref{eq:3.20b}. See Figure \ref{figure1} for a typical profile of ${\rho}_{3}$.
Since ${\rho}_{1}$ is a viscosity solution of \eqref{underlineW} by definition, we can apply
comparison principle in Lemma \ref{lemA2rho} to  deduce that
 \begin{equation}\label{inequaliywmu}
   {\rho}_{1}(s)\geq {\rho}_{3}(s)\,\quad\text{ for }s\in[s^\mu(\hat c),\hat c].
 \end{equation}
By \eqref{goalinequality}, it follows by continuity that there exists $\delta \in (0, \hat{c} - s^{\mu}(\hat{c}))$ such that
$$
0<\rho^\mu_{{\rm nlp}}(s) <\nu_4^{\mu}(\hat{c})\cdot(s- s^\mu(\hat{c})) \quad \text{ for } s\in[ \hat{c} -\delta ,\hat{c}].
$$
On account of \eqref{inequalityunderw}, we have ${\rho}_{3}(s) = \rho^{\mu}_{\rm nlp}(s)$ in $[\hat{c}-\delta,\hat{c}]$, so that
\begin{equation}\label{eq:uu?}
w_{3,*}(1,s)\geq {\rho}_{1}(s)\geq {\rho}_{3}(s)
=\rho^\mu_{{\rm nlp}}(s)\,\,\text{ for all }s\in [\hat c -\delta, \hat c ],
\end{equation}
where the first inequality follows from Step 2 and the second one is due to \eqref{inequaliywmu}.
Since $\hat{c} \in \mathcal{D}_\mu$, we already have $w_{3,*}(1,s) \geq \rho^\mu_{\rm nlp}(s)$ for $s \in [\hat{c},c_2]$. Taking \eqref{eq:uu?} into account,
we  thus arrive at  $\hat c-\delta\in\mathcal{D}_{\mu}$,   a contridiction. Step 4 is  completed and Proposition \ref{prop:overc3}  is proved.
\end{proof}
\begin{figure}[http!!]
  \centering
\includegraphics[height=1.4in]{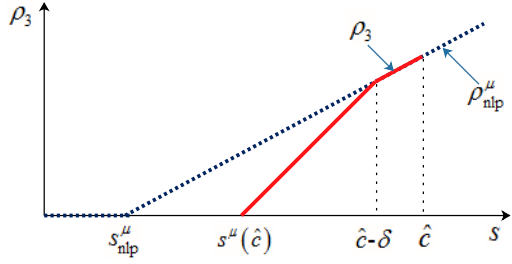}
  \caption{A typical profile of  $\rho_3$.}
  \label{figure1}
  \end{figure}

We improve Proposition \ref{prop:overc3} by removing one alternative in its conclusion.

\begin{proposition}\label{prop:overc3_1}
 Assume that $({\rm H}_{c_1,c_2,\lambda})$ holds. If
 $\mu \in \mathcal{E}$, 
 then
$\overline{c}_3 \leq \beta^\mu_3$, 
where  $\beta^\mu_3=\max\{s^\mu_{{\rm nlp}},c_{\rm{LLW}}\}$ with $s^\mu_{{\rm nlp}}$ given by \eqref{charact_wnlpmu}.
\end{proposition}

\begin{proof}
If Proposition \ref{prop:overc3_1} fails, i.e. $\overline{c}_3 > \beta^\mu_3$, then by Proposition \ref{prop:overc3},  we deduce
$$
w_{3,*}(1,c)\geq  \rho^\mu_{{\rm nlp}}(c) \quad \text{ for all } \,\,\, c \in [\beta^{\mu}_3, c_2].
$$
Since $\beta^\mu_3 \geq s^\mu_{{\rm nlp}}$ (see the definition of  $\beta^\mu_3$ in \eqref{def:betamu}),  we have $(\beta^{\mu}_3,c_2) \subset (s^\mu_{{\rm nlp}}, c_2)$. It follows from the definition of $s^\mu_{{\rm nlp}}$ in  \eqref{charact_wnlpmu} that
\begin{equation}\label{w3>0}
    w_{3,*}(1,c)\geq \rho^\mu_{{\rm nlp}}(c)>0 \,\,\text{ for any } c\in (\beta^\mu_3,c_2).
\end{equation}
Recalling \eqref{eq:w_star3},  we see that for each $c\in (\beta^\mu_3,c_2)$,
$$ \lim_{\epsilon \to 0} u_3\left(\frac{1}{\epsilon},\frac{ c}{\epsilon}\right) \leq  \lim_{\epsilon \to 0} \exp\left(-\frac{ w_{3,*}(1, c) + o(1)}{\epsilon}\right)=0,
$$
so that by  \eqref{w3>0} we derive that
$$
\lim_{t\to\infty} \sup_{ct<x <c_2t}  u_3(t,x) = 0  \, \text{ for each }c\in (\beta^\mu_3,c_2).
$$
Therefore, we reach $\overline{c}_3 \leq \beta^\mu_3$, a contradiction. This completes the proof.
\end{proof}

\subsubsection{Bootstrapping  up to $\mu=1$}\label{subsection5_2}
We proceed to prove $\overline{c}_3 \leq \beta_3$ in this subsection,  where  $\beta_3=\max\{s_{{\rm nlp}},c_{\rm{LLW}}\}$. In view of Proposition \ref{prop:overc3_1} (see also Remark \ref{rmk:defbeta3}), it is enough to show that $1 \in \mathcal{E}$.
We will argue with a continuity argument.

\begin{lemma}\label{lemma:0inE}
Assume that $({\rm H}_{c_1,c_2,\lambda})$ holds. 
Then $0\in\mathcal{E}$.
\end{lemma}

\begin{proof}
Observe from \eqref{eq:H.c1c2}
that
$$ 
\liminf_{\substack{(t',x')\to (t,x)\\ \epsilon\to 0}}u_2^\epsilon(t',x')\geq\chi_{\{\overline{c}_3t<x<c_2t\}}.$$
By a standard verification,  we assert that $w_{3,*}$ is a viscosity super-solution of
\begin{equation*}
\left \{
\begin{array}{ll}
\smallskip
\min\{\partial_t w+d_3|\partial_xw|^2+r_3(1-a_{32}\chi_{\{\overline{c}_3 t<x<c_2 t\}}),w\}=0 &\text{in }(0,\infty)\times (0,\infty),\\
\smallskip
w(t, 0)=0, \qquad w(0,x)=\lambda x & \text{for }t\geq 0,\, x\geq 0,
\end{array}
\right.
\end{equation*}
where $\lambda\in(0,\infty]$ is given in $({\rm H}_{c_1,c_2,\lambda})$, and the boundary conditions have been verified in Remark \ref{rmk:w1w20}.
A direct application of Proposition \ref{lemA3rho} yields
 \begin{equation}\label{coparisonresult}
  w_{3,*}(t,x) \geq t{\rho}_{4}\left(x/t\right)\,\,\text{ in }~ [0,\infty)\times [0,\infty), 
  \end{equation}
where ${\rho}_{4}(s)$  is the unique viscosity solution of
\begin{align}\label{eq:sub-solutionw2}
\left\{\begin{array}{l}
\smallskip
\min\{\rho-s\rho'+d_3|\rho'|^2+r_3(1-a_{32}\chi_{\{\overline{c}_3 <s<c_2\}}),\rho\}=0 \,\,\text{ in } (0,\infty),\\
\rho(0)=0,\quad\lim\limits_{s\to\infty}\frac{\rho(s)}{s}=\lambda.
\end{array}
\right.
\end{align}

We recall from \eqref{eq:w3nlpmu} that $\rho^0_{{\rm nlp}}$ defines the unique viscosity solution of
\begin{align}\label{w03nlp}
\left\{\begin{array}{l}
\smallskip
\min\{\rho-s\rho'+d_3|\rho'|^2+r_3(1-a_{32}\chi_{\{s\leq c_2\}}),\rho\}=0 \,\,\text{ in } (0,\infty),\\
\rho(0)=0,\quad\lim\limits_{s\to\infty}\frac{\rho(s)}{s}=\lambda.
\end{array}
\right.
\end{align}
Regarding \eqref{eq:sub-solutionw2} and \eqref{w03nlp},  we apply Lemma \ref{lemA1rho} with $c_{g}=c_2$ and $g=r_3a_{32}\chi_{\{\overline{c}_3 <s<c_2\}}$ or $g=r_3a_{32}\chi_{\{s\leq c_2\}}$ 
to deduce that
${\rho}_{4}(c_2)=\rho^0_{{\rm nlp}}(c_2)$,
by which we  deduce from \eqref{coparisonresult} that  $w_{3,*}(1,c_2)\geq {\rho}_{4}(c_2)=\rho^0_{{\rm nlp}}(c_2)$, so that $0\in\mathcal{E}$ by the definition of $\mathcal{E}$ in \eqref{E}.
\end{proof}
We now state the main result of this section.
\begin{proposition}\label{prop:estamoverc3}
Let $(u_i)_{i=1}^3$  be any solution of \eqref{eq:1-1} such that $({\rm H}_{c_1,c_2,\lambda})$ holds. 
Then
$$\overline{c}_3 \leq \beta_3=\max\{s_{{\rm nlp}},c_{\rm{LLW}}\}, $$
where  
$s_{ {\rm nlp}}=s_{ {\rm nlp}}^\mu \big|_{\mu=1}$, and $s_{ {\rm nlp}}^\mu$ is given by \eqref{charact_wnlpmu}.
\end{proposition}
\begin{remark}\label{rmk:u1upper}
By Proposition \ref{prop:estamoverc3}, we have $\overline{c}_3 \leq \beta_3= \max\{ s_{\rm nlp}, c_{\rm{LLW}}\}$. Hence species $u_1$ is controlled by species $u_2$ in the region $\{(t,x):\beta_3 t \leq x \leq c_2 t\}$ for $t \gg 1$. Precisely, one may apply \eqref{eq:H.c1c2} 
to deduce that
$$
\limsup_{\substack{(t',x')\to (t,x)\\ \epsilon\to 0}}  u_1^\epsilon(t',x')\leq \chi_{\{c_2t\leq x\leq c_1t\}}+\chi_{\{x\leq\beta_3t\}},
$$
where $u_1^\epsilon$ is defined by \eqref{scaling}.
\end{remark}

\begin{proof}[Proof of Proposition {\rm\ref{prop:estamoverc3}}]
 In order to apply Proposition \ref{prop:overc3_1}, we will show $1 \in \mathcal{E}$ by a continuity argument.
First, we claim that $\mathcal{E}$ is closed and non-empty. It is closed since
$\rho^\mu_{{\rm nlp}}(c_2)$ is continuous in $\mu$ (see Lemma \ref{lemma_w3nlp}). The set
$\mathcal{E}$ is non-empty because of $0\in\mathcal{E}$, which is proved in Lemma \ref{lemma:0inE}.
Define $\mu_{\hspace{-.5pt}M}=\sup\mathcal{E}$ such that $\mu_{\hspace{-.5pt}M}\in[0,1]$. By the closedness of $\mathcal{E}$,
we have $\mu_{\hspace{-.5pt}M}\in\mathcal{E}$, so that Proposition \ref{prop:overc3_1} implies
\begin{equation}\label{eatmatover}
  \overline{c}_3 \leq \beta^{\mu_{\hspace{-.5pt}M}}_3,
\end{equation}
where  $\beta^{\mu_{\hspace{-.5pt}M}}_3=\max\{s^{\mu_{\hspace{-.5pt}M}}_{{\rm nlp}},c_{\rm{LLW}}\}$ with $s^{\mu_{\hspace{-.5pt}M}}_{{\rm nlp}}\in(0,c_2)$ given by \eqref{charact_wnlpmu}. If $s^{\mu_{\hspace{-.5pt}M}}_{{\rm nlp}}\leq c_{\rm{LLW}}$, then Proposition \ref{prop:estamoverc3} can be established by \eqref{eatmatover}, since
$$\overline{c}_3 \leq \beta^{\mu_{\hspace{-.5pt}M}}_3=c_{\rm{LLW}}\leq \max\{s_{\rm nlp}, c_{\rm{LLW}}\}=\beta_3.$$
 Therefore, it remains to consider the case 
 $s^{\mu_{\hspace{-.5pt}M}}_{{\rm nlp}}> c_{\rm{LLW}}$ and prove $\mu_{\hspace{-.5pt}M}=1$.

Suppose to the contrary that $\mu_{\hspace{-.5pt}M}<1$ and  $s^{\mu_{\hspace{-.5pt}M}}_{{\rm nlp}} > c_{\rm LLW}$.  Then
\begin{equation}\label{eq:ssssnlp}
\overline{c}_3\leq\beta^{\mu_{\hspace{-.5pt}M}}_3= \max\{s^{\mu_{\hspace{-.5pt}M}}_{{\rm nlp}}, c_{\rm{LLW}}\}=s^{\mu_{\hspace{-.5pt}M}}_{{\rm nlp}}.
\end{equation} 
Again by \eqref{eq:H.c1c2}, 
we deduce that
$$ \liminf_{\substack{(t',x')\to (t,x)\\ \epsilon\to 0}}u_1^\epsilon(t',x')\geq \chi_{\{c_2t<x<c_1t\}}\,\,\text{ and }\,\,\liminf_{\substack{(t',x')\to (t,x)\\ \epsilon\to 0}}u_2^\epsilon(t',x')\geq\chi_{\{s^{\mu_{\hspace{-.5pt}M}}_{{\rm nlp}}<x<c_2t\}},$$
where the second inequality follows from \eqref{eq:ssssnlp}.
 Letting $\epsilon \to 0$ in \eqref{eq:w3}, by \eqref{eq:spreadingly111} in Proposition \ref{thm:u1v1111} and  Remark \ref{rmk:w1w20},
we may check that $w_{3,*}$ is a viscosity super-solution of  
 \begin{equation}\label{eq:super-solution00}
\left \{
\begin{array}{ll}
\smallskip
\min\{\partial_t w+d_3|\partial_xw|^2+\overline{\mathcal{R}}_{3}(\frac{x}{t}),w\}=0 &\text{for }x>s^{\mu_{\hspace{-.5pt}M}}_{{\rm nlp}}t,\\
\smallskip
w(0,x)=\lambda x & \text{for } x\geq 0,\\
w(t,s^{\mu_{\hspace{-.5pt}M}}_{{\rm nlp}} t)=0 & \text{for } t\geq 0,
\end{array}
\right.
\end{equation}
where $\overline{\mathcal{R}}_{3}(s)=r_3(1-a_{31}\chi_{\{c_2<s<c_1\}}-a_{32}\chi_{\{s^{\mu_{\hspace{-.5pt}M}}_{{\rm nlp}}<s<c_2\}})$. By Proposition \ref{lemA3rho}, we deduce that
\begin{equation}\label{comparisonresult20}
  w_{3,*}(t,x)\geq t{\rho}_{5}\left(x/t\right)\,\,\text{ for } \,\,x\geq s^{\mu_{\hspace{-.5pt}M}}_{{\rm nlp}}t,
\end{equation}
 where ${\rho}_{5}$ defines the unique viscosity solution of
 \begin{equation}\label{eq:super-solution}
\left \{
\begin{array}{ll}
\smallskip
\min\{\rho-s\rho'+d_3|\rho'|^2+\overline{\mathcal{R}}_{3}(s),\rho\}=0 &\text{for } s\in(s^{\mu_{\hspace{-.5pt}M}}_{{\rm nlp}},\infty),\\
\rho(s^{\mu_{\hspace{-.5pt}M}}_{{\rm nlp}})=0,\quad\lim\limits_{s\to\infty}\frac{\rho(s)}{s}=\lambda.
\end{array}
\right.
\end{equation}

In what follows, we will show that there exists some  ${\mu_{\sharp}}\in(\mu_{\hspace{-.5pt}M},1)$ such that ${\mu_{\sharp}}\in\mathcal{E}$. This is in contradiction to $\mu_{\hspace{-.5pt}M} = \sup \mathcal{E}$.

\smallskip

\noindent {\bf Step 1.} We choose some ${\mu_{\sharp}}\in(\mu_{\hspace{-.5pt}M},1)$ such that
\begin{equation}\label{c**smallc2}
  \rho^{{\mu_{\sharp}}}_{{\rm nlp}}(c_2) < \nu_1( s^{\mu_{\hspace{-.5pt}M}}_{\rm nlp}) \cdot (c_2-s^{\mu_{\hspace{-.5pt}M}}_{{\rm nlp}}),
\end{equation}
where (see \eqref{eq:nu1} for the definition of $\nu_1(\ell)$)
$$\nu_1( s^{\mu_{\hspace{-.5pt}M}}_{\rm nlp})= \nu_1(\ell)\big|_{\ell = s^{\mu_{\hspace{-.5pt}M}}_{\rm nlp}} = \frac{s^{\mu_{\hspace{-.5pt}M}}_{\rm nlp} + \sqrt{(s^{\mu_{\hspace{-.5pt}M}}_{\rm nlp})^2 - \alpha_3^2(1-a_{32})}}{2d_3}.$$

Indeed,  notice that  $s^{\mu_{\hspace{-.5pt}M}}_{{\rm nlp}}> c_{\rm{LLW}}\geq\alpha_3\sqrt{1-a_{32}}$.  By  Lemma \ref{w3nllp} and the definition of $\lambda^{\mu_{\hspace{-.5pt}M}}_{{\rm nlp}}$ there, we arrive at
$$\rho^{\mu_{\hspace{-.5pt}M}}_{{\rm nlp}}(c_2)= \lambda^{\mu_{\hspace{-.5pt}M}}_{{\rm nlp}}(c_2- {s^{\mu_{\hspace{-.5pt}M}}_{{\rm nlp}}})<\frac{s^{\mu_{\hspace{-.5pt}M}}_{{\rm nlp}}}{2d_3}(c_2-s^{\mu_{\hspace{-.5pt}M}}_{{\rm nlp}})<  \nu_1( s^{\mu_{\hspace{-.5pt}M}}_{\rm nlp}) \cdot(c_2-s^{\mu_{\hspace{-.5pt}M}}_{{\rm nlp}}).$$
Hence, from the continuity of $\rho^{\mu}_{{\rm nlp}}$ in $\mu$ as stated in Lemma \ref{lemma_w3nlp},  we may choose ${\mu_{\sharp}}\in(\mu_{\hspace{-.5pt}M},1)$ to be sufficiently close to $\mu_{\hspace{-.5pt}M}$, so that \eqref{c**smallc2} holds.

\smallskip
\noindent {\bf Step 2.}  Let ${\mu_{\sharp}}\in(\mu_{\hspace{-.5pt}M},1)$  be chosen as in Step 1. It follows 
from Lemma \ref{lemma:sub-solution} that
$$
{\rho}_{6}(s)
:=\min\left\{\rho^{{\mu_{\sharp}}}_{{\rm nlp}}(s), \, \nu_1( s^{\mu_{\hspace{-.5pt}M}}_{\rm nlp}) \cdot (s-s^{\mu_{\hspace{-.5pt}M}}_{{\rm nlp}})\right\}$$
is  a viscosity sub-solution of
 \begin{align*}
\left \{
\begin{array}{ll}
\medskip
\min\{\rho-s\rho'+d_3|\rho'|^2+r_3\left(1-a_{32}\right), w\}=0 &\mathrm{in}~ (s^{\mu_{\hspace{-.5pt}M}}_{{\rm nlp}},c_2),\\
\rho(s^{\mu_{\hspace{-.5pt}M}}_{{\rm nlp}})=0,\,\,\,\,\rho(c_2)= \rho^{{\mu_{\sharp}}}_{{\rm nlp}}(c_2 ),
\end{array}
\right.
\end{align*}
where $\nu_1( s^{\mu_{\hspace{-.5pt}M}}_{\rm nlp})= \nu_1(\ell)\big|_{\ell = s^{\mu_{\hspace{-.5pt}M}}_{\rm nlp}}$ is defined in Step 1.

\smallskip
\noindent {\bf Step 3.} By \eqref{c**smallc2}, there exists  $\delta>0$ such that 
${{\rho}_{6}}(s)=\rho^{{\mu_{\sharp}}}_{{\rm nlp}}(s)$ for $s \in [c_2-\delta,c_2]$.
Define
\begin{equation}\label{inequalityunderwmax}
{\rho}_{7}(s):=
\begin{cases}
\smallskip
 \rho^{{\mu_{\sharp}}}_{{\rm nlp}}(s) & \text{ for }\,\,  s\in ( c_2,\infty),\\
 \smallskip
 \rho^{{\mu_{\sharp}}}_{{\rm nlp}}(s) = {\rho}_{6}(s) & \text{ for }\,\,  s\in[ c_2-\delta,c_2],\\
{{\rho}_{6}}(s)&  \text{ for } \,\, s\in [s^{\mu_{\hspace{-.5pt}M}}_{{\rm nlp}}, c_2 - \delta). 
\end{cases}
\end{equation}
We claim that $\rho_7$ is a a viscosity sub-solution of \eqref{eq:super-solution} in the entire region 
 $(s^{\mu_{\hspace{-.5pt}M}}_{\rm nlp},\infty)$.

Indeed, since ${\mu_{\sharp}}<1$,
we see that $\rho^{{\mu_{\sharp}}}_{{\rm nlp}}$ is a viscosity sub-solution of \eqref{eq:super-solution} in $(c_2 -\delta,\infty)$. Moreover, it is straightforward to check that
${{\rho}_{6}}$ is a viscosity sub-solution of \eqref{eq:super-solution} in 
$(s^{\mu_{\hspace{-.5pt}M}}_{\rm nlp}, c_2)$. By noting that viscosity solution is a local property, we can deduce that ${\rho}_{7}$ as given in \eqref{inequalityunderwmax}
 is a viscosity sub-solution of \eqref{eq:super-solution} in 
 $(s^{\mu_{\hspace{-.5pt}M}}_{\rm nlp},\infty)$.

\smallskip
\noindent {\bf Step 4.} We claim that $\rho_5(s) \geq \rho_7(s)$  
for $s\in [s^{\mu_{\hspace{-.5pt}M}}_{{\rm nlp}}, c_1]$.

To see that, we apply Lemma \ref{lemA1rho} with $c_{g}=c_1$, and deduce from  the definitions of ${\rho}_{5}$ (as the unique viscosity solution of \eqref{eq:super-solution}) and $\rho^{{\mu_{\sharp}}}_{{\rm nlp}} = \rho^{\mu}_{{\rm nlp}}\big|_{\mu = {\mu_{\sharp}}}$ that
$$
{\rho}_{5}(s)= \rho^{{\mu_{\sharp}}}_{{\rm nlp}}(s)={\rho}_{7}(s) \quad  \text{ for }s\geq c_1.
$$ 
Observe that ${\rho}_{5}$ and ${\rho}_{7}$ are a pair of super- and sub-solutions to \eqref{eq:super-solution} in the bounded interval $(s^{\mu_{\hspace{-.5pt}M}}_{{\rm nlp}}, c_1)$, with boundary values,  
$$
\rho_5(s^{\mu_{\hspace{-.5pt}M}}_{{\rm nlp}}) \geq  0 = \rho_7(s^{\mu_{\hspace{-.5pt}M}}_{{\rm nlp}})\quad \text{ and }\quad \rho_5(c_1)= \rho_7(c_1).
$$
By standard comparison principle \cite[Theorem 2]{Tourin_1992},
 we arrive at
${\rho}_{5}(s)\geq{\rho}_{7}(s)$ for $s\in [s^{\mu_{\hspace{-.5pt}M}}_{{\rm nlp}}, c_1]$.

\smallskip
\noindent {\bf Step 5.} We claim that  ${\mu_{\sharp}}\in\mathcal{E}$,  which  contradicts the definition of $\mu_{\hspace{-.5pt}M}$. 

Step 4 together with \eqref{comparisonresult20} implies that
$$ w_{3,*}(1,c_2)\geq{\rho}_{5}(c_2)\geq {\rho}_{7}(c_2)=\rho^{{\mu_{\sharp}}}_{{\rm nlp}}(c_2),$$
 so that ${\mu_{\sharp}}\in\mathcal{E}$, which is impossible as ${\mu_{\sharp}}>\mu_{\hspace{-.5pt}M}= \sup \mathcal{E}$.

Thus $\mu_{\hspace{-.5pt}M} = 1$, and $\mathcal{E} = [0,1]$. We  now take $\mu=1$ in Proposition \ref{prop:overc3_1} to establish $\overline{c}_3 \leq \beta^{\mu}_3 \big|_{\mu =1} = \beta_3$. This completes the proof.
\end{proof}
\subsection{Estimate $\underline c_3$ from below}\label{S4}
Let $({\rm H}_{c_1,c_2,\lambda})$ be satisfied for some $c_1>c_2$ and $\lambda \in (0,\infty]$. Define the continuous function $\underline{\rho}_{\rm nlp}:[0,\infty)\rightarrow[0,\infty)$ as  the unique viscosity solution of
\begin{equation}\label{eq:sub-solutionw3}
\left \{
\begin{array}{ll}
\smallskip
\min\{\rho-s\rho'+d_3|\rho'|^2+\underline{\mathcal{R}}(s),\rho\}=0 &\text{ in }(0,\infty),\\
\rho(0)=0,\quad\lim\limits_{s\to\infty}\frac{\rho(s)}{s}=\lambda,
\end{array}
\right.
\end{equation}
where 
\begin{equation}\label{qq:3.55}
\underline{\mathcal{R}}(s)=r_3(1-a_{31}\chi_{\{c_2\leq s\leq c_1\}}-a_{32}\chi_{\{s\leq c_2\}})-r_3a_{31}\chi_{\big\{s\leq \max\{s_{{\rm nlp}}, 
c_{\rm{LLW}}\}\big\}},
\end{equation} 
with $s_{\rm nlp}(c_1,c_2,\lambda)$ and $c_{\rm LLW}$ being given in \eqref{charact_snlp} and Definition \ref{c3LLW} respectively.

By arguing similarly as in Lemma \ref{lemma_w3nlp}, we see that $\underline{\rho}_{\rm nlp}$ is continuous and non-decreasing in $s \in [0,\infty)$. Hence, we can 
define the speed 
\begin{equation}\label{definitionunderbata}
 \underline{\beta}_3 := \sup \{s:{\underline\rho}_{\rm nlp}(s) = 0\}.
\end{equation}

The main result of this subsection is to establish a lower bound of $\underline c_3$, and show that this lower bound coincides with the upper bound showed in Proposition \ref{prop:estamoverc3} 
in  case $s_{\rm nlp} \geq c_{\rm{LLW}}$.

\begin{proposition}\label{underc3}
Let $(u_i)_{i=1}^3$  be any solution of \eqref{eq:1-1} such that $({\rm H}_{c_1,c_2,\lambda})$ holds. 
Then we have 
$
\underline{c}_{3}\geq \underline{\beta}_3, 
$
where $\underline{\beta}_3$ is  given  by \eqref{definitionunderbata}
and satisfies
$\underline{\beta}_3\in [\alpha_3\sqrt{1-a_{31} - a_{32}},s_{{\rm nlp}}]$. Furthermore, 
$\underline{\beta}_3= s_{{\rm nlp}}$ if $ s_{{\rm nlp}}\geq c_{\rm{LLW}}$,
 where $s_{{\rm nlp}} = s_{{\rm nlp}}^\mu \big|_{\mu =1}$. 
\end{proposition}
\begin{remark}\label{rmk:estamoverc3}
In case $s_{\rm nlp} \geq c_{\rm{LLW}}$, Propositions \ref{prop:estamoverc3} and \ref{underc3} together imply
$$
s_{\rm nlp} \leq \underline{c}_3 \leq \overline{c}_3 \leq \max\{s_{\rm nlp}, c_{\rm{LLW}}\} = s_{\rm nlp}.
$$
Thus $\overline{c}_3 = \underline{c}_3 = s_{\rm nlp}$ and Propositions \ref{prop:estamoverc3} and \ref{underc3} are sharp in such a case.
\end{remark}

We establish a lemma before proving the Proposition \ref{underc3}.

\begin{lemma}\label{lem:sub-solutionw3}
Let $(u_i)_{i=1}^3$  be any solution of \eqref{eq:1-1} such that $({\rm H}_{c_1,c_2,\lambda})$ holds.
Then
$$w_3^*(1,s)\leq{\underline\rho}_{\rm nlp}(s)\,\,\text{ for }s\in[0,\infty).$$
where ${\underline\rho}_{\rm nlp}$ is defined as the unique viscosity solution of \eqref{eq:sub-solutionw3}.
\end{lemma}
\begin{proof}
Using \eqref{eq:H.c1c2}, 
we observe that
$$
\limsup_{\substack{(t',x')\to (t,x)\\ \epsilon\to 0}}  u_1^\epsilon(t',x')\leq \chi_{\{c_2t\leq x\leq c_1t\}}+\chi_{\{x\leq\beta_3t\}} \quad \text{ and }\quad 
\limsup_{\substack{(t',x')\to (t,x)\\ \epsilon\to 0}}  u_2^\epsilon(t',x')\leq \chi_{\{x\leq c_2t\}}. 
$$
 Letting $\epsilon \to 0$ in \eqref{eq:w3} and use Remark \ref{rmk:w1w20} to verify boundary conditions, it is standard to
verify that $w_3^*$ is  a viscosity  sub-solution of
\begin{equation}\label{eq:subsub}\min\{\partial_t w+d_3|\partial_xw|^2+\underline{\mathcal{R}}_1(x/t),w\}=0 \quad\text{in}\quad (0,\infty)\times (0,\infty),\end{equation}
where at $s = c_2$, we can only estimate both $u_1^\epsilon$ and $u_2^\epsilon$ from above by $1$, so that
$$\underline{\mathcal{R}}_1(s)=\begin{cases}
\underline{\mathcal{R}}(s)& \text{for } s\neq c_2,\\
1-a_{31}-a_{32} & \text{for } s= c_2,
\end{cases}
$$
and $\underline{\mathcal{R}}(s)$ is defined in \eqref{qq:3.55}. 

Note that $\underline{\mathcal{R}}_1(s) \leq \underline{\mathcal{R}}(s)$, so
we cannot directly apply comparison directly, and need to proceed with care. Since $w_3^*(t,x)=tw_3^*(1,\frac{x}{t})$ as stated in Remark \ref{rmk:w1w3},  by arguing as in Lemma \ref{lemA2rho}, it can be verified that $\rho_3^*(s):=w_3^*(1,s)$ satisfies, in the viscosity sense,
\begin{equation}\label{eq:w3upperstar}
\min\{\rho-s\rho'+d_3|\rho'|^2+\underline{\mathcal{R}}_1(s),\rho\}\leq 0 \quad \text{in}\quad(0,\infty),
\end{equation}
and satisfies
\begin{equation*}
\rho_3^*(0)=0  \quad \text{ and }\quad \lim\limits_{s\to\infty}\frac{\rho_3^*(s)}{s}=\lambda.
\end{equation*}

Now, we claim  $\rho_3^* \in {\rm Lip}_{\mathrm{loc}}([0,\infty))$.  Indeed,
since $\underline{\mathcal{R}}_1(s) \geq 0$, one can easily verify that $\rho_3^*(s)$ is a viscosity sub-solution of $\rho -s\rho' + d_3|\rho'|^2 = 0$ on $(0,\infty)$. Fix an arbitrary $s_0>0$, and choose $M = M(s_0)>0$ such that $\rho(s_0) -s_0 M + \frac{d_3}{2}M^2 >0$, then a direct application of \cite[Proposition 1.14]{IshiiLect} yields that  $\rho_3^*$ is Lipschitz continuous in $[0,s_0]$, for arbitrary $s_0>0$.

It follows from the Rademacher's theorem that $\rho_3^*$ is differentiable a.e. on $[0,\infty)$. Being a viscosity sub-solution of \eqref{eq:w3upperstar}, it thus satisfies the differential inequality \eqref{eq:w3upperstar} a.e. on $[0,\infty)$. Since $\underline{\mathcal{R}}_1(s) = \underline{\mathcal{R}}(s)$ a.e. we have proved that $\rho_3^*$ satisfies
\begin{equation}\label{eq:w3upperstar123}
\min\{\rho-s\rho'+d_3|\rho'|^2+\underline{\mathcal{R}}(s),\rho\}\leq 0 \quad \text{a.e. in}\quad(0,\infty).
\end{equation}
However, by the convexity of the Hamiltonian, we can apply \cite[Proposition 5.1]{Bardi1997} to conclude that $\rho_3^*$ is in fact a viscosity sub-solution of \eqref{eq:sub-solutionw3}, for which $\underline\rho_{\rm nlp}$ is the unique viscosity solution. The lemma thus follows by a standard comparison via Proposition \ref{lemA3rho} (where we realize that $w_3^*(t,x) = t\rho_3^*(x/t)$ is now a sub-solution of \eqref{eq:subsub} with $\underline{\mathcal{R}}_1$ replaced by $\underline{\mathcal{R}}$).
\end{proof}
\begin{remark}\label{rmk:3.21}
The equivalence in {concepts of the viscosity sub-solutions} for \eqref{eq:w3upperstar} and \eqref{eq:w3upperstar123} is guaranteed by convexity of the Hamiltonian functions. This is however not true for {viscosity} super-solutions. 
This is why the condition \eqref{thm:u1v1111} was essential in  
Step 2 in the proof of Proposition \ref{prop:overc3} as well as in the verification that $w_{3,*}$ being a viscosity super-solution of \eqref{eq:super-solution00}.
Heuristically, this suggests that the ``gap" created by the succession of $u_1$ by $u_2$ may speed up, but never slow down, the invasion $u_3$. 
\end{remark}

We are ready to prove Proposition \ref{underc3}.


\begin{proof}[Proof of Proposition {\rm\ref{underc3}}]
\mbox{}

\smallskip
\noindent {\bf Step 1.} We show $\underline{c}_{3}\geq \underline{\beta}_3$.
By Lemma \ref{lem:sub-solutionw3},  $0\leq w_3^*(1,s)\leq{\underline\rho}_{\rm nlp}(s)$ in $(0,\infty)$,
which implies 
 \begin{equation*}
\{s:w_3^*(1,s)=0\}\supset\{s:{\underline\rho}_{\rm nlp}(s) = 0\}=[0, \underline{\beta}_3].
\end{equation*}
Therefore, by \eqref{eq:w_star300}, $w_3^*(t,x)=0$ in $ \{(t,x):0 \leq x <\underline{\beta}_3 t\}$.
Recalling the definition  of $w_3^*$ in \eqref{eq:w_star3},  we see that 
$w_3^\epsilon(t,x) = -\epsilon \log{u_3^\epsilon(t,x)} \to 0$ \ locally uniformly in $\{(t,x):0 \leq x <\underline{\beta}_3 t\}$ as  $\epsilon \to 0$. 
 For each small $\eta>0$, we can choose the compact sets $K,K'$ in Lemma \ref{lem:underlinec} by $$K=\{(1,s): 2\eta\leq s \leq  \underline{\beta}_3 - 2\eta \}\,\,\text{ and }\,\, K'=\{(1,s): \eta\leq s \leq  \underline{\beta}_3 - \eta \}.$$
Since $0 \leq u_i \leq 1$ for all $(t,x)$ and $i=1,2$,
$$\sup\limits_{(t,x)\in K'} u_1^\epsilon(t,x)\leq 1\,\,\text{ and }\sup\limits_{(t,x)\in K'} u_2^\epsilon(t,x)\leq 1.$$
We may apply Lemma \ref{lem:underlinec} to deduce that 
\begin{equation*}
\liminf_{t\to \infty}\inf_{2\eta t\leq x \leq  (\underline{\beta}_3 - 2\eta)t } u_3 (t,x)=\liminf_{\epsilon\to 0}\inf_{K} u_3^\epsilon (t,x)>0.
\end{equation*}
This implies $\underline{c}_3\geq  \underline{\beta}_3$, and  Step 1 is completed.

\smallskip

\noindent {\bf Step 2.}  We  claim $\underline{\beta}_3\leq s_{{\rm nlp}}$.
It is straightforward to check  that $\rho_{\rm nlp} = \rho_{\rm nlp}^{\mu}\big|_{\mu=1}$ (as given by \eqref{eq:w3nlpmu} with $\mu=1$)  is a viscosity sub-solution of \eqref{eq:sub-solutionw3}. 
By Corollary \ref{corA3rho} once again we get
\begin{equation*}
 {\underline\rho}_{\rm nlp}(s)\geq  \rho_{{\rm nlp}}(s)\,\,\text{ in } \,[0,\infty).
\end{equation*}
By definition of $\underline{\beta}_3$ and $s_{\rm nlp}$ (see \eqref{definitionunderbata} and \eqref{charact_wnlpmu} with $\mu=1$), we deduce
\begin{equation}\label{charactwnlp1}
\underline{\beta}_3 = \sup \{s :{\underline\rho}_{\rm nlp}(s)=0\}  \leq \sup \{s :{\rho}_{\rm nlp}(s)=0\} = s_{\rm nlp}.
\end{equation}

\noindent {\bf Step 3.}  We  claim $\underline{\beta}_3\geq \alpha_3\sqrt{1-a_{31} - a_{32}}$.
In this case, it suffices to note that ${\rho}_{8}(s):=\max\{\frac{s^2}{4d_3} - r_3(1-a_{31} - a_{32}),0\}$ defines a viscosity super-solution of  \eqref{eq:sub-solutionw3}, so that we can proceed as in Step 2 to yield the  inequality $\underline{\beta}_3 \geq \alpha_3\sqrt{1-a_{31} - a_{32}}$.


\smallskip

\noindent {\bf Step 4.} We show $\underline{\beta}_3= s_{{\rm nlp}}$ if $s_{{\rm nlp}}\geq c_{\rm LLW}$.
Assume $s_{{\rm nlp}}\geq c_{\rm LLW}$, then  $\beta_3=s_{{\rm nlp}}$.
It suffices to show that $\rho_{{\rm nlp}}$ is a viscosity solution of \eqref{eq:sub-solutionw3}. Indeed, if that is the case, then by uniqueness in Lemma \ref{lemA1rho} we deduce that ${\underline\rho}_{\rm nlp} = \rho_{{\rm nlp}}$ in $[0,\infty)$.
Hence the equality  in \eqref{charactwnlp1} holds, and we derive $\underline{\beta}_3=s_{{\rm nlp}}$.

To show that $\rho_{{\rm nlp}}$ is a viscosity solution of \eqref{eq:sub-solutionw3}, noting that since $\rho_{{\rm nlp}}$ is already a viscosity sub-solution of \eqref{eq:sub-solutionw3}, it is enough to verify that it is a viscosity super-solution of \eqref{eq:sub-solutionw3} in $(0,\infty)$. To this end, suppose that $\rho_{{\rm nlp}} - \phi$ attains a strict local minimum at $s_0>0$. In view of $\rho_{\rm nlp} \geq 0$, it suffices to show that
\begin{equation}\label{eq:wnlpissuper}
\rho_{{\rm nlp}}(s_0)-s_0\phi'(s_0)+d_3|\phi'(s_0)|^2 + (\underline{\mathcal{R}})^*(s_0) \geq 0,
\end{equation}
where $(\underline{\mathcal{R}})^* (s)= \limsup_{s' \to s}\underline{\mathcal{R}}(s') $ is the upper envelope of $\underline{\mathcal{R}}$. Let $\mathcal{R}^1(s)$ be as given below \eqref{eq:w3nlpmu}, then
$\underline{\mathcal{R}}(s) = \mathcal{R}^1(s) - r_3a_{31}\chi_{\{ s \leq \beta_3\}}$ for $s \geq 0$.  
Since $\mathcal{R}^1(s)$ is continuous at 
$\beta_3 \in (0, c_2)$, 
\begin{equation}\label{eq:envelope}
(\underline{\mathcal{R}})^* (s) = (\mathcal{R}^1)^*(s) \quad \text{ for }s \geq \beta_3.
\end{equation}

If $s_0 \geq s_{\rm nlp}$, then noting that $\beta_3=s_{\rm nlp}$, 
we have $s_0 \geq \beta_3$. Then
 \eqref{eq:wnlpissuper}  follows from \eqref{eq:envelope} and the fact that  $w_{\rm nlp}$ is a viscosity solution of \eqref{eq:w3nlpmu}.

 If $s_0 < s_{\rm nlp}$, then by the definition of $s_{\rm nlp}$, we see that $\rho_{\rm nlp}(s)$ vanishes in a neighborhood of $s_0$, so that $ \phi'(s_0) = 0$, and thus
 $$
\rho_{{\rm nlp}}(s_0)-s_0\phi'(s_0)+d_3|\phi'(s_0)|^2 +  (\underline{\mathcal{R}})^*(s_0) = \underline{\mathcal{R}}(s_0) \geq r_3(1-a_{31} - a_{32}) >0,
$$
which implies  \eqref{eq:wnlpissuper} holds. 
Therefore,  $\rho_{{\rm nlp}}$ is a viscosity super-solution of \eqref{eq:sub-solutionw3}. 
\end{proof}


 We are now in a position to prove {Theorem A}.
\begin{proof}[Proof of {Theorem {\rm A}}]
The estimate \eqref{estamitec3} in {Theorem A} is a direct consequence of  Proposition \ref{prop:estamoverc3}, Proposition \ref{underc3}, and Lemma \ref{lem:betamu3}.  The fact $\underline{c}_3=\overline{c}_3=s_{{\rm nlp}}$  when $s_{{\rm nlp}}\geq c_{\rm{LLW}}$ is proved in Remark \ref{rmk:estamoverc3}.  Therefore, it remains to show that 
\begin{equation}\label{spreadingfinial}
    \liminf\limits_{t \rightarrow \infty} \inf\limits_{0 \leq  x < (\underline{c}_3 -\eta)t} u_3(t,x)>0 \quad \text{for each small }\,\,\eta>0,
\end{equation}
and then the spreading property \eqref{eq:spreadingly} follows from \eqref{eq:H.c1c2} in the assumption $({\rm H}_{c_1,c_2,\lambda})$. 
%
%
Observe from $u_1\leq 1$ and $u_2 \leq 1$ that $u_3$ is a classical super-solution of
$$
\left\{
\begin{array}{ll}
\smallskip
\partial_t u = d_3\partial_{xx}u +r_3 u(1-a_{31}-a_{32} - u) &\text{ for }(t,x) \in (0,\infty)\times \mathbb{R},\\
u(0,x) =  u_3(0,x) &\text{ for } x \in \mathbb{R}.
\end{array}
\right.
$$
By the classical results in  Fisher \cite{Fisher_1937} or Kolmogorov et al. \cite{Kolmogorov_1937}, we have
\begin{equation}\label{eq:u3kpp1}
\lim_{t \to \infty} \inf_{|x| < (\underline{\sigma}_3 - \eta)t} u(t,x) \geq (1-a_{31}-a_{32})/2>0 \quad \text{ for small }\,\,\eta >0,
\end{equation}
where $\underline\sigma_3:= \alpha_3\sqrt{1-a_{31}-a_{32}}$. Hence, to prove \eqref{spreadingfinial}, it suffices to claim that
for each $\eta>0$ small, there exist some $T>0$ and $\delta>0$ such that
\begin{equation}\label{positiveu3}
  u_3(t,x) \geq \delta\quad \text{ in }\,\,\,\{(t,x): t \geq T,\,  (-\underline\sigma_3 + \eta)t \leq  x \leq (\underline{c}_3 - \eta) t\}.
\end{equation}
Fix a small $\eta>0$.   By definition of  $\underline{c}_3$,   there exist some $c'_3\in (\underline{c}_3-{\eta}, \underline{c}_3)$ and $T>0$ such that
\begin{equation}\label{lowerboundu3_2}
  \inf\limits_{t\geq T}u_3(t,c'_3t)>0.
\end{equation}
Since $\{(T,x): (-\underline{\sigma}_3 +\eta) T \leq x \leq c'_3 T\}$ is compact, we apply \eqref{eq:u3kpp1} to get
\begin{equation}\label{eq:c3'}
  \inf\limits_{(-\underline{\sigma}_3 +\eta) T \leq x \leq c'_3 T}u_3(T,x)>0.
\end{equation}
By \eqref{eq:u3kpp1}, \eqref{lowerboundu3_2} and \eqref{eq:c3'}, we deduce that
\begin{equation*}
  \begin{split}
   \delta := \min\Big\{\inf\limits_{t \geq T}  u_3(t,c'_3 t),&\,\,\,\inf\limits_{t \geq T}  u_3(t,(-\underline{\sigma}_3 + \eta) t),\,\,(1-a_{31}-a_{32})/2, \inf\limits_{\scriptscriptstyle (-\underline{\sigma}_3 + \eta) T\leq x \leq c'_3 T} u_3(T,x) \Big\}
  \end{split}
\end{equation*}
is positive. Then $u_3$ is a super-solution to the KPP-type equation
$\partial_t  u=d_3 \partial_{xx} u + r_3u(1-a_{31}-a_{32}-u)$
such that $u_3 \geq \delta$ on the parabolic boundary. By the parabolic maximum principle, we derive \eqref{positiveu3} and the proof of {Theorem A} is complete.
\end{proof}

\section{Asymptotic behaviors of the final zone}\label{S5}

The purpose of this section is to prove {Theorem B}, which characterizes the asymptotic  profile of the final zone $\{(t,x):x < \underline{c}_3 t\}$.

\begin{proof}[Proof of {Theorem {\rm B}}]
By \eqref{eq:u3kpp1} and the definition of $\underline{c}_3$, it is obvious that $\underline{c}_3\geq \alpha_3\sqrt{1-a_{31}-a_{32}}$.
Hence, it remains to prove \eqref{spreadings}. We divide the proof into several steps.

\smallskip
\noindent{\bf Step 1.} We show that, if
\begin{equation}\label{upperboundui_B}
  \lim\limits_{t\to\infty}\sup\limits_{(-\underline{\sigma}_3 +\eta)t<x<(\underline{c}_3-\eta)t} u_i\leq B_{i} \,\,\text{ for } i=1,2, \text{ and each small }\eta>0, 
\end{equation}
with some constant $B_{i}\in[0,1]$, then
\begin{equation}\label{lowerboundu3}
  \lim\limits_{t\to\infty}\inf\limits_{(-\underline{\sigma}_3 +\eta)t<x<(\underline{c}_3-\eta)t} u_3\geq A \quad \text{ for each small }\eta>0,
\end{equation}
where $A=1-a_{31}B_{1}-a_{32}B_{2}$.
 Suppose {that} \eqref{lowerboundu3} fails. Then there exists $(t_n,x_n)$ such that
\begin{equation}\label{eq:4.6b}
        c_n:= x_n/t_n \to c \in (-\underline\sigma_3, \underline{c}_3) \,\text{ and }\, \lim\limits_{n\to\infty}u_3(t_n,x_n)<A.
\end{equation}
 Denote $(u_{1,n},u_{2,n},u_{3,n})(t,x):= (u_1,u_2,u_3)(t_n + t, x_n + x)$.
In view of
$0 \leq u_{i,n} \leq 1$  in $ [-t_n,\infty)\times\mathbb{R}$ for $i=1,2,3,$
  by parabolic estimates we assert that $(u_{1,n},u_{2,n},u_{3,n})$ is precompact in $C^2_{\mathrm{loc}}(K)$ for each compact subset $K \subset \mathbb{R}^2$.
Passing to a subsequence if necessary,   we assume that $u_{3,n} \to \hat u_3$ in $C^2_{\mathrm{loc}}(\mathbb{R}^2)$, which satisfies  
 $
\partial_t \hat u_3-d_3\partial_{xx}\hat u_3\geq r_3\hat u_3(A-\hat u_3)$ in $\mathbb{R}^2$ due to \eqref{upperboundui_B}.

Observe from \eqref{positiveu3} that $ \hat u_3(t,x) >\delta$ in $\mathbb{R}^2$. Let $\underline{U}_3(t)$ denote the solution of 
$$ U'_3 = r_3U_3(A-U_3 )\quad\text{and}\quad U_3(0)=\delta,$$
which satisfies $\underline{U}_3(\infty)=\lim\limits_{t\to\infty}\underline{U}_3(t) = A$.
Note that for each $T_1>0$, $\hat u_3(-T_1,x)\geq \underline{U}_3(0)$ for all $x \in \mathbb{R}$. By comparison, we have
$
 \hat u_3(t,x)\geq  \underline{U}_3(t+T_1) $ for $ (t,x)\in [-T_1,0] \times \mathbb{R},
$
and thus  
$ \hat u_3(0,0)\geq \underline{U}_3(T_1)$ for each $T_1>0.$
Letting $T_1 \to \infty$, we obtain $\hat u_3(0,0)\geq A$, i.e.
$\lim\limits_{n \to \infty} u_3(t_n,x_n) \geq A,$
 contradicting  \eqref{eq:4.6b}. Therefore, \eqref{lowerboundu3} is established.

\smallskip
\noindent{\bf Step 2.} We show that, if
\begin{equation*}
  \lim\limits_{t\to\infty}\inf\limits_{(-\underline\sigma_3 + \eta)t<x<(\underline{c}_3-\eta)t} u_3\geq A \quad \text{ for each small }\eta>0,
\end{equation*}
with some $A\in[0,1]$, then
\begin{equation}\label{upperboundui}
  \lim\limits_{t\to\infty}\sup\limits_{(-\underline\sigma_3 + \eta)t<x<(\underline{c}_3-\eta)t} u_i\leq B_{i}\,\,\text{ for }i=1,2,  \text{ and each small }\eta>0,
\end{equation}
where $ B_{i}=\max\{1-a_{i3}A,0\}$.

Since this is analogous to the arguments in Step 1, we omit the details.

\smallskip
%
%

\noindent{\bf Step 3.} We show that if $1< a_{23}\leq a_{13}$, then for each small $\eta>0$,
\begin{equation}\label{eq:spreadu1}
\lim\limits_{t\to\infty}\sup\limits_{(-\underline\sigma_3 + \eta)t<x<(\underline{c}_3-\eta)t} |u_1(t,x)|=0;
\end{equation}
If $1<a_{13}\leq a_{23}$, then for each small  $\eta>0$,
\begin{equation}\label{eq:spreadu2}
\lim\limits_{t\to\infty}\sup\limits_{(-\underline\sigma_3 + \eta)t<x<(\underline{c}_3-\eta)t} |u_2(t,x)|=0.
\end{equation}

We only treat the case $1< a_{23}\leq a_{13}$ and prove \eqref{eq:spreadu1}, as \eqref{eq:spreadu2} follows by switching the roles of $u_1$ and $u_2$. We shall define $B_{1,j}, B_{2,j}, A_j$ inductively by applying Steps 1 and 2. First, define $B_{1,1} = B_{2,1} = 1$ and apply Step 2, so that  \eqref{lowerboundu3} holds for $A=A_1=1-a_{31}-a_{32}$. Then letting $A=A_1$ in Step 2, we deduce \eqref{upperboundui} with $B_{i}=B_{i,2}=\max\{1-a_{i3}A_1,0\}$ for $i=1,2$. Recurrently, if $1-a_{13}A_m>0$ for some $m>1$, then $1-a_{23}A_m>0$ (by $a_{13}\geq a_{23}$) and 
\begin{equation}\label{recursion}
\begin{split}
  A_{m+1}&=1-a_{31}(1-a_{13}A_{m})-a_{32}(1-a_{23}A_{m})\\
  &=A_1+(a_{31}a_{13}+a_{32}a_{23})A_{m}=\sum_{n=0}^m (a_{31}a_{13}+a_{32}a_{23})^n A_1,
  \end{split}
\end{equation}
whence \eqref{lowerboundu3} holds for $A=A_{m+1}$. Notice from \eqref{recursion} that $A_{m+1}>A_{m}$. We shall claim that there exists some $m_0>1$  such that $1-a_{13}A_{m_0}\leq 0$, and then applying  \eqref{upperboundui} in Step 3 with $A=A_{m_0}$,  we deduce  \eqref{eq:spreadu1}.

To this end, we  argue by contradiction and assume that $1-a_{13}A_m>0$ for all $m>1$, so that \eqref{recursion} holds for all $m$. We can reach a contradiction by  the following two cases:
\begin{itemize}
  \item [{\rm(i)}]  If  $a_{31}a_{13}+a_{32}a_{23}\geq1$, then by choosing some $m_0\geq \frac{1}{a_{13}A_1}$, it follows from \eqref{recursion} that $1-a_{13}A_{m_0}\leq1-a_{13} m_0A_1\leq 0$, which is a contradiction;

\smallskip
  \item [{\rm(ii)}]If $a_{31}a_{13}+a_{32}a_{23}<1$, then letting $m\nearrow\infty$ in \eqref{recursion} gives
 $$A_{\infty}=\frac{A_1}{1-(a_{31}a_{13}+a_{32}a_{23})}=\frac{1-a_{31}-a_{32}}{1-(a_{31}a_{13}+a_{32}a_{23})}\geq 1,$$
 where the inequality follows from  $a_{13}>1$ and $a_{23}>1$. Hence, we can choose $m_0$ large such that $1-a_{13}A_{m_0}\leq0$, which is also a contradiction.
\end{itemize}
Therefore, \eqref{eq:spreadu1} is established.

\smallskip
\noindent{\bf Step 4.}  We show \eqref{spreadings}. The proof is based on classification of entire solutions of \eqref{eq:1-1}. We only consider the case  $1< a_{23}\leq a_{13}$, since for the case  $1<a_{13}\leq a_{23}$, \eqref{spreadings} can be proved by a same way. By \eqref{eq:spreadu1} in Step 3, it remains to prove
\begin{equation}\label{lowerboundu2u3}
  \lim\limits_{t\to\infty}\sup\limits_{(-\underline\sigma_3 + \eta)t<x<(\underline{c}_3-\eta)t}(|u_2(t,x)|+|u_3(t,x)-1|)=0 \quad \text{ for each small }\eta>0.
\end{equation}
 Suppose that \eqref{lowerboundu2u3} fails. Then
 there exists $(t_n,x_n)$ such that
\begin{equation*}
        c_n:=x_n/t_n \to c \in (-\underline\sigma_3, \underline{c}_3) \quad\text{and}\quad \lim\limits_{n\to\infty}u_2(t_n,x_n)>0 \text{ or } \lim\limits_{n\to\infty}u_3(t_n,x_n)<1.
\end{equation*}
As before,  we also denote $(u_{1,n},u_{2,n},u_{3,n})(t,x):= (u_1,u_2,u_3)(t_n + t, x_n + x)$.
By parabolic estimates,
   we may assume that $(u_{1,n},u_{3,n})$ converges  to $(\hat u_{2},\hat u_{3})$ in $C^2_{\mathrm{loc}}(\mathbb{R}^2)$ by passing to a subsequence. By \eqref{eq:spreadu1},  we see that $(\hat u_{2},\hat u_{3})$ satisfies
 \begin{equation*}
\left\{\begin{array}{ll}
\smallskip
\partial_t \hat u_{2}-\partial_{xx}\hat u_{2}= \hat u_{2}(1-\hat u_{2} - a_{23}\hat u_{3}) & \text{ in } \mathbb{R}^2,\\
\partial_t \hat u_{3}-d_3\partial_{xx}\hat u_{3}= r_3\hat u_{3}(1-a_{32} \hat u_{2} -\hat u_{3}) & \text{ in } \mathbb{R}^2.\\
\end{array}
\right.
\end{equation*}
 Again by \eqref{positiveu3} in  the proof of {Theorem A}, we have $ (\hat u_2,\hat u_3)(t,x) \preceq (1,\delta)$ for all $(t,x)\in\mathbb{R}^2$. Let $(\overline{U}_2,\underline{U}_3)$ denote the solution of  ODEs
$$
 U'_2 = U_2(1-U_{2} - a_{23}U_3 ) \,\,\text{ and }\,\,U'_3 = r_3U_3(1-a_{32} U_2 -U_3 ),
$$
with initial data $(\overline{U}_2,\underline{U}_3)(0)=(1,\delta)$, so that $(\overline{U}_2,\underline{U}_3)(\infty) =(0,1)$ due to $a_{32}<1<a_{23}$.

Analogue to Step 1,  by comparison we can arrive at
$(\hat u_2,\hat u_3)(0,0)\preceq(\overline{U}_2,\underline{U}_3)(\infty)=(0,1),$
so that $u_2(t_n,x_n)\to 0$ and $u_3(t_n,x_n) \to 1$ as $n \to \infty$,
which is a contradiction. Therefore, \eqref{lowerboundu2u3} is established.
The proof of {Theorem B} is now complete.
\end{proof}

\begin{remark}
Let hypothesis $({\rm H}_{c_1,c_2,\lambda})$ hold with $c_1>c_2$. Assume $a_{13} > a_{31}$ and $a_{23} < a_{32}$ (instead of $a_{13}, a_{23} >1$ in {Theorem B}), then we claim that for each small $\eta>0$,
\begin{equation*}\label{spreadingsss}
\lim\limits_{t\to\infty}\sup\limits_{0<x<(\underline{c}_3-\eta)t}(|u_1(t,x) - U_1^*|+|u_2(t,x) - U_2^*|+|u_3(t,x)-U_3^*|)=0,
\end{equation*}
where $(U_1^*,U_2^*,U_3^*)$ is the unique positive equilibria of system \eqref{eq:1-1}. In this case, \cite[Proposition 1]{Cham_2010} can be applied to yield a strictly convex Lyapunov function for  system \eqref{eq:1-1}. One can then proceed similarly as in \cite[Lemma 7.7]{Xiao_2019} to fully classify the positive entire solutions of the three-species competition system \eqref{eq:1-1}. We omit the details.
\end{remark}

\section{Properties  of $s_{\rm nlp}(c_1,c_2,\lambda)$}\label{S6}
This section is devoted to deriving  Remark \ref{rmk:1.3} and the proof of Proposition \ref{coro:1estamoverc001}. 
Since $s_{\rm nlp}\leq \sigma_3(\lambda)$ was established in 
Lemma \ref{lem:betamu3}, 
Remark \ref{rmk:1.3} follows from  the following result.
\begin{proposition}\label{charactsnlpc}
Let $s_{{\rm nlp}}$ be defined by \eqref{charact_snlp} for 
$c_1>c_2$ and $\lambda\in (0,\infty]$. 
Then $s_{\rm{nlp}}\geq \alpha_3\sqrt{1-a_{32}}.$
Furthermore, if $a_{31} < a_{32}$ and $\alpha_3 < c_2 < c_1 < \alpha_3(\sqrt{a_{32}} + \sqrt{1-a_{32}})$, then 
$s_{\rm nlp} > \alpha_3\sqrt{1-a_{32}}$, where $\alpha_3=2\sqrt{d_3 r_3}$.
\end{proposition}
\begin{proof}
\noindent {\bf Step 1.} We prove  $s_{\rm{nlp}}\geq \alpha_3\sqrt{1-a_{32}}.$  The proof depends on the construction of a viscosity super-solution and an application of Lemma \ref{lemA2rho}.
Define $\overline{\rho}: [0,c_1]\to[0,\infty)$ by
\begin{equation*}
\overline{\rho}{(s)} := 
\begin{cases}
\smallskip
\overline{\lambda}s- d_3 \overline{\lambda}^2 +\overline{r} &\text{ for }{c_2} < s\leq {c_1},\\
\smallskip
\frac{s^2}{4d_3}-r_3(1-a_{32})&\text{ for } \alpha_3\sqrt{1-a_{32}} < {s} \leq  {c_2},\\
0 &\text{ for } 0\leq {s}  \leq \alpha_3\sqrt{1-a_{32}},
\end{cases}
\end{equation*}
where $\overline{\lambda}=\frac{{c_1}+{c_2}}{4d_3}+\frac{r_3(1-a_{32})}{{c_1}-{c_2}}$ and $\overline{r}=d_3\left[\frac{{c_1}-{c_2}}{4d_3}-\frac{r_3(1-a_{32})}{{c_1}-{c_2}}\right]^2$.
Let us show that  ${\overline{\rho}}$ is a viscosity super-solution of \eqref{eq:w3nlp} in the interval $(0,c_1)$. 
Set $A:= \frac{c_1- c_2}{4d_3}$ and $B:= \frac{r_3(1-a_{32})}{c_1-c_2}$. We can verify ${\overline{\rho}}$  is continuous  in $[0,{c_1}]$ by the following calculations at $s=c_2$:
\begin{align*}
\quad\overline{\lambda}{c_2}- d_3 \overline{\lambda}^2 +\overline{r} &={c_2}\left[A + B+\frac{{c_2}}{2d_3}\right]-d_3\left[A + B+\frac{{c_2}}{2d_3}\right]^2+\overline{r}\\
&={c_2}\left[A + B\right]+\frac{(c_2)^2}{2d_3}-d_3\left[A +B\right]^2 -{c_2}\left[A  +B\right]  -\frac{(c_2)^2}{4d_3}+\overline{r}\\
&=\frac{(c_2)^2}{4d_3}-d_3\left[A  +B\right]^2+d_3\left[A - B\right]^2\\
&=\frac{(c_2)^2}{4d_3}-4d_3 AB = \frac{(c_2)^2}{4d_3}-r_3(1-a_{32}).
\end{align*}
Observe that ${\overline{\rho}}$ is a classical super-solution for \eqref{eq:w3nlp} in the set $(0,c_1)\setminus\{{c_2},\alpha_3\sqrt{1-a_{32}} \}$.
Since ${\overline{\rho}} \geq 0$ by construction, it remains to consider the case when ${\overline{\rho}} - \phi$ attains a strict local minimum at  $\hat{s} = {c_2}$ or $ \hat{s}=\alpha_3\sqrt{1-a_{32}}$, where $\phi\in C^1(0,\infty)$ is a test function.  In case ${\hat{s}} = {c_2}$, direct calculation at $s={\hat{s}}$ yields that
\begin{align*}
{\overline{\rho}(\hat s)}-{\hat s}\phi' +d_3|\phi'|^2 +\mathcal{R}^*(\hat s)
&\geq \frac{(c_2)^2}{4d_3}- {c_2} \phi' +d_3 | \phi'|^2 = d_3\left( \phi'-\frac{{c_2}}{2d_3}\right)^2\geq 0.
\end{align*}
On the other hand, if ${\hat{s}} = \alpha_3\sqrt{1-a_{32}}$, then at $s={\hat{s}}$, we calculate that
\begin{align*}
{\overline{\rho}(\hat s)}-\hat s\phi' +d_3|\phi'|^2 + \mathcal{R}^*(\hat s)&=  -\alpha_3\sqrt{1-a_{32}} \phi' +d_3 | \phi'|^2 +  r_3(1 -a_{32})\\
&= d_3\left[ \phi'-\sqrt{\tfrac{r_3(1-a_{32})}{d_3}}\right]^2\geq 0.
\end{align*}
Therefore, ${\overline{\rho}}$ defined above is a viscosity super-solution of \eqref{eq:w3nlp}.

Recall that $\rho_{\rm nlp}$ denotes  the unique viscosity solution of \eqref{eq:w3nlp}.
Notice that $\rho_{\rm{nlp}}(0)=0={\overline{\rho}}(0)$.
To apply Lemma \ref{lemA2rho}, let us verify the boundary condition 
${\overline{\rho}}({c_1})\geq \rho_{\rm{nlp}}({c_1}).$
First, by Lemma \ref{lemA1rho} with  $c_g=c_1$,  it follows easily that $\rho_{\rm{nlp}}({c_1})\leq c^2_1/(4d_3)-r_3(1-a_{31})$. Writing $\overline\lambda = \frac{c_1}{2d_3} - D$ and $\overline{r} = d_3 D^2$, where $D = \frac{{c_1}-{c_2}}{4d_3}-\frac{r_3(1-a_{32})}{{c_1}-{c_2}}$, we verify that
 \begin{align*}
 {\overline{\rho}}({c_1})= \overline{\lambda}{c_1}- d_3 \overline{\lambda}^2 +\overline{r}
 &={c_1}\left[   \frac{c_1}{2d_3} - D   \right] -d_3\left[ \frac{c_1}{2d_3} - D\right]^2+d_3 D^2\\
&=\frac{c_1^2}{2d_3}-{c_1}D-  d_3  \left[  \frac{c_1^2}{4d_3^2}- \frac{c_1D}{d_3} +D^2 \right]  +  d_3 D^2=\frac{c_1^2}{4d_3}\geq\rho_{\rm{nlp}}({c_1}).
\end{align*}
By applying  Lemma \ref{lemA2rho} with $c_{b}=c_1$, we  deduce ${\overline{\rho}}(s)\geq \rho_{\rm{nlp}}(s)$ for $s\in [0,{c_1}]$. 
In particular, $0 \leq \rho_{\rm nlp} (\alpha_3\sqrt{1-a_{32}})\leq {\overline{\rho}}(\alpha_3\sqrt{1-a_{32}})=0$. Hence by definition, we have  $s_{\rm{nlp}} = \sup\{s \geq 0: \rho_{\rm nlp}(s)=0\} \geq \alpha_3\sqrt{1-a_{32}}$, which completes Step 1.

\smallskip
\noindent {\bf Step 2.} We show  $s_{\rm{nlp}}> \alpha_3\sqrt{1-a_{32}}$ if  $a_{31} < a_{32}$ and $\alpha_3 < c_2 < c_1 < \alpha_3(\sqrt{a_{32}} + \sqrt{1-a_{32}})$.
Due to $a_{31} < a_{32}$,  it can be verified that  $\rho_{\rm nlp}$ is a viscosity sub-solution of
\begin{align}\label{eq:w3nlp-1-1}
\left\{\begin{array}{l}
\smallskip
\min\{\rho-s\rho'+d_3|\rho'|^2+r_3(1-a_{32}\chi_{\{s\leq c_1\}}),\rho\}=0 \,\,\text{ in } (0,\infty),\\
\rho(0)=0,\quad\lim\limits_{s\to\infty}\frac{\rho(s)}{s}=\lambda.
\end{array}
\right.
\end{align}
Let $\overline\rho$ be the unique viscosity solution of \eqref{eq:w3nlp-1-1}. Corollary \ref{corA3rho} implies that 
$\rho_{\rm nlp}(s)\leq\overline\rho(s)$ for all $s\in(0,\infty)$. 
By the same arguments as in the proof of Lemma \ref{lemA2rho}, we can verify that $\overline{w}(t,x):=t\overline\rho(\frac{x}{t})$ is the viscosity solution of
\begin{align}\label{hatwinfty-1-1}
\left\{\begin{array}{ll}
\medskip
\min\{\partial_tw+ d_3|\partial_xw|^2+r_3(1-a_{32}\chi_{\{x \leq c_1 t\}}),w\}=0 \,\,\text{ in }(0,\infty)\times(0,\infty),\\
 w(0,x)=\lambda x, 
\,\,\,
w(t,0)=0 \quad\text{ for } \,\, x\in  [0,\infty), \,\, t\in (0,\infty).
\end{array}
\right.
\end{align}

Define $\underline{s}_{\rm nlp}>0$ such that $\{(t,x): \overline{w}(t,x)=0\}=\{(t,x):t>0 \,\text{ and }\,x\leq \underline{s}_{\rm nlp} t\}$. Since $c_1 < \alpha_3(\sqrt{a_{32}} + \sqrt{1-a_{32}})$, it is shown in \cite[(1.6) in Theorem 1.3]{LLL20192} that  
$$\underline{s}_{\rm nlp}> \alpha_3\sqrt{1-a_{32}}.$$
(To apply \cite[Theorem 1.3]{LLL20192}, we consider the transformation   $\tilde{w}(s,y):=\overline{w}\left(\frac{s}{r_3},\sqrt{\frac{d_3}{r_3}}y\right)$.)  
In view of ${\overline{\rho}}(\underline{s}_{\rm nlp})=\overline{w}(1,\underline{s}_{\rm nlp})=0$, we arrive at 
$0 \leq \rho_{\rm nlp} (\underline{s}_{\rm nlp})\leq {\overline{\rho}}(\underline{s}_{\rm nlp})=0$. By definition,   $s_{\rm{nlp}} = \sup\{s \geq 0: \rho_{\rm nlp}(s)=0\} \geq \underline{s}_{\rm nlp}> \alpha_3\sqrt{1-a_{32}}$ as desired.
\end{proof}

Next we  prove   Proposition \ref{coro:1estamoverc001}. We first prepare the following result:
\begin{lemma}\label{lemma_snlp}
Let $s_{\rm nlp}(c_1,c_2,\lambda)$ be defined by \eqref{charact_snlp}. Then $s_{{\rm nlp}}$ is continuous and  non-increasing  with respect to $\lambda\in (0,\infty]$.
\end{lemma}
\begin{proof}
Similar to Lemma \ref{lemma_w3nlp}, we can prove that $\rho_{\rm nlp}$ is non-increasing and continuous in $\lambda$. This implies that $s_{\rm nlp}$ is non-increasing and continuous in $\lambda$. We omit the details. 
%
\end{proof}

\begin{proof}[Proof of  Proposition {\rm\ref{coro:1estamoverc001}}]
The proof is divided into 
two steps.

\smallskip
\noindent {\bf Step 1.} We show that there exist some $\delta>0$ and $\lambda\in (0,\infty)$ such that {\rm (i)} $\sigma_3(\lambda)=d_3\lambda + \frac{r_3}{\lambda}<\hat{s}_{\rm nlp}(c_1)$ and {\rm (ii)} $s_{\rm nlp}(c_1,  \hat{s}_{\rm nlp}(c_1), \lambda) > c_{\rm LLW} 
$ for all $a_{31}\in [0,\delta)$. Here $\hat{s}_{\rm nlp}$ and $s_{\rm nlp}$ are defined in \eqref{qq:1.7a} and \eqref{charact_snlp}, respectively.

First, we consider the case $a_{31}=0$. We claim that there exists $\overline\lambda \in (0, \sqrt{r_3/d_3})$ such that  
\begin{equation}\label{eq:qq5.1}
\alpha_3 <\overline\sigma_3 = \hat{s}_{\rm nlp}(c_1)  = {s}_{\rm nlp}(c_1,\hat{s}_{\rm nlp}(c_1),\overline\lambda), \quad \text{ where }\,\,\, \overline\sigma_3:= d_3\overline{\lambda} + \frac{r_3}{\overline{\lambda}}.
\end{equation}
Indeed, since $\alpha_3=2\sqrt{d_3 r_3}< \hat{s}_{\rm nlp}(c_1)$ (see \eqref{eq:prop1.14a}), we can choose the unique $\overline\lambda \in (0, \sqrt{r_3/d_3})$ so that the first equality in 
 \eqref{eq:qq5.1} holds. Also, the first inequality in \eqref{eq:qq5.1} follows from $\overline\lambda \in (0, \sqrt{r_3/d_3})$. Next, we show $\overline\sigma_3={s}_{\rm nlp}(c_1,\hat{s}_{\rm nlp}(c_1),\overline\lambda)$. To this end, we observe that 
 $$\overline{\rho}(s):=\max\left\{\overline{\lambda}\cdot(s-\overline{\sigma}_3), \, 0\right\}$$
 is the unique viscosity solution of 
  \begin{align}\label{eq:w3nlp-1}
 \min\{\rho-s\rho'+d_3|\rho'|^2+\mathcal{R}(s),\rho\}=0 \,\,\text{ in } (0,\infty),
 \end{align}
 where $\mathcal{R}(s)=r_3(1-a_{32}\chi_{\{s\leq  \hat{s}_{\rm nlp}(c_1)\}})$. However, by the fact that $a_{31}=0$ and the first equality of \eqref{eq:qq5.1}, it follows that $\overline\rho$ is also the unique viscosity solution of \eqref{eq:w3nlp} with $c_2=\hat{s}_{\rm nlp}(c_1)$ and $\lambda = \overline\lambda$. Hence $\overline\sigma_3 = s_{\rm nlp}(c_1,\hat{s}_{\rm nlp},\overline\lambda)$. Now, if we consider $\lambda = \overline\lambda+ \epsilon$ for $\epsilon>0$ small, then 
 $$
 \sigma_3(\lambda) < \overline\sigma_3 = \hat{s}_{\rm nlp}(c_1) \quad \text{and} \quad s_{\rm nlp}(c_1,  \hat{s}_{\rm nlp}(c_1), \lambda)  > \alpha_3,
 $$
where we used the continuous dependence in $\lambda$ (Lemma \ref{lemma_snlp}). 
This proves Step 1 in the case $a_{31}=0$. Since all of the desired inequalities are strict, the case $0<a_{31}\ll 1$ follows by continuous dependence on $a_{31}$.

\smallskip
\noindent {\bf Step 2.} 
We show \eqref{eq:spreadingly},  \eqref{eq:spreadings0}, and  \eqref{spreadings} hold for the  chosen $\lambda$ as in Step 1. First, 
the hierarchy conditions \eqref{eq:compactorder1}, \eqref{eq:1--1b}, and $a_{21}a_{12} <1$ imply  $\hat{c}_{\rm LLW} = 2\sqrt{1-a_{21}}$ as stated in Remark \ref{rmk:1--2a}.  Since $\sigma_3(\lambda) < \hat{s}_{\rm nlp}(c_1)$ for the chosen $\lambda$, we can apply Theorem \ref{thm:suff1.2} to deduce that $({\rm H}_{c_1,c_2,\lambda})$ holds with $c_1 = 2\sqrt{d_1 r_1}$ and $c_2 = \hat{s}_{\rm nlp}(c_1)$. 

Since \eqref{eq:compactorder1} and  \eqref{eq:1--1b} also hold, {Theorem A} applies. 
In particular, \eqref{eq:spreadingly} holds. 
In view of  $s_{\rm nlp}(c_1,  \hat{s}_{\rm nlp}(c_1), \lambda)> \alpha_3 \geq c_{\rm LLW}$ (see Remark \ref{rem1.3} for the last inequality), it follows that 
\eqref{eq:spreadings0} holds. Finally, due to  $a_{13}>1$ and $a_{23}>1$, \eqref{spreadings} is a direct consequence of {Theorem B}. 
Proposition \ref{coro:1estamoverc001} is proved.
\end{proof}

\section*{Acknowledgment}
 KYL is partiallly supported by  NSF grant DMS-1853561. QL (201706360310) and SL (201806360223) would like to thank the China Scholarship Council for financial support during the period of their overseas study and  express their gratitude to the Department of Mathematics, The Ohio State University for the warm hospitality. 

\begin{appendices}

\section{Some Useful Lemmas 
}\label{appendixlemmas}
In this section, we include some lemmas used in this paper. The first result called linear determinacy is based on  \cite[Theorem 2.1]{Lewis_2002}. Related results can be found in \cite{Huang_2010,Huang_2011}.
\begin{lemma}\label{lem:1estamoverc30}
 Let $c_{\rm{LLW}}$ be given by Definition {\rm \ref{c3LLW}}.  Suppose that
 \begin{equation}\label{lineardetermained}
 d_3 \geq \frac{1}{2}, \quad a_{32}(1-a_{21}) < 1 < \frac{a_{23}}{1-a_{21}},\quad\text{and} \quad a_{32}a_{23} <1.
  \end{equation}
 Then 
$c_{\rm{LLW}}=\alpha_3\sqrt{1-a_{32}(1-a_{21})}.$
 \end{lemma}
 
\begin{proof}
Let $(u,v)$ be a solution of \eqref{eq:1-2'}. 
Then 
  $(U, V)(s,y):=\left(\frac{u}{1-a_{21}},v\right)\left(\frac{s}{r_3},\sqrt{\frac{d_3}{r_3}}y\right)$
 satisfies
 \begin{equation*}
\left\{
\begin{array}{ll}
\medskip
\partial_s U-\hat{d}_3\partial_{yy}U=\hat{r}_3U(1-U-\hat{a}_{23} V)& \text{ in }(0,\infty)\times \mathbb{R},\\
\partial _s V-\partial_{yy}V= V(1-\hat{a}_{32} U-V) & \text{ in }(0,\infty)\times \mathbb{R},\\
\end{array}
\right .
\end{equation*}
where
$\hat{d}_3=\frac{1}{d_3}$, $\hat{r}_3=\frac{1-a_{21}}{r_3}$, $\hat{a}_{23}=\frac{a_{23}}{1-a_{21}}$, and $\hat{a}_{32}=a_{32}(1-a_{21})$.
Under these notations, we observe that  \eqref{lineardetermained} is  equivalent to
$$\hat{d}_3\leq 2, \quad\hat{a}_{32}<1<\hat{a}_{23}, \quad \text{and}\quad\hat{a}_{23}\hat{a}_{32}<1. 
$$
Thus Lemma \ref{lem:1estamoverc30} is a direct consequence of \cite[Theorem 2.1]{Lewis_2002}.
\end{proof}

The next result  will be used in the proof of 
Lemma \ref{lemma:5-4}.

\begin{lemma}\label{lem:appen1}
Fix any $\hat{c}>0$. Let $(u, v)$ be a solution of
\begin{equation*}
\left\{
\begin{array}{ll}
\smallskip
\partial_t   u-\partial_{xx} u=  u(1-a_{21}-  u-a_{23}  v)& 0< x< \hat{c}t,\, t>t_0,\\
\partial _t   v-d_3\partial_{xx}  v=r_3   v(1-a_{32}  u-  v)&  0< x<\hat{c}t,\, t>t_0.\\
\end{array}
\right.
\end{equation*}
Suppose that there exists some $\hat \mu>0$ such that 
\begin{itemize}
\item[{\rm(i)}]  $\lim\limits_{t\to\infty}(  u,  v)(t,\hat{c}t)=(1-a_{21},0)$;
\item[{\rm(ii)}]  $\lim\limits_{t\to\infty} e^{\mu t}  v(t,\hat{c}t)=0~$ for each $\mu \in [0,\hat \mu)$.
\end{itemize}
Then there exists $s_{\hat{c}}>0$ such that
$$
\lim_{t\to\infty} \sup_{ct<x<\hat{c} t}  v(t,x) = 0   \quad \text{ for each }c > s_{\hat{c}},
$$
where
\begin{equation*}
s_{\hat{c}}=
\begin{cases}
\medskip
c_{\rm{LLW}}& \text{if } \hat\mu\geq   \lambda_{\rm LLW} (\hat{c} - c_{\rm{LLW}}),\\
\hat{c}-\frac{2d_3\hat\mu}{\hat{c}-\sqrt{\hat{c}^2-4d_3[\hat\mu+r_3(1-a_{32}(1-a_{21}))]}} & \text{if }\hat\mu<\lambda_{\rm LLW} (\hat{c} - c_{\rm{LLW}}). \end{cases}
\end{equation*}
Here  $c_{\rm{LLW}}$ is defined in Definition {\rm \ref{c3LLW}} and
$\lambda_{\rm LLW}=\frac{c_{\rm{LLW}}-\sqrt{(c_{\rm{LLW}})^2-\alpha^2_3(1-a_{32}(1-a_{21}))}}{2d_3}.$
\end{lemma}
The proof of Lemma \ref{lem:appen1} can be found in \cite[Lemma 2.4]{LLL2019} and is omitted.

\begin{remark}\label{rmk:lambdaLLW}
We mention that $\lambda_{\rm LLW}$ defined above satisfies
 \begin{equation*}
   \lambda_{\rm LLW}c_{\rm{LLW}}=d_3\lambda_{\rm LLW}^2+r_3(1-a_{32}(1-a_{21}))\quad\text{and} \quad\lambda_{\rm LLW}\leq\frac{c_{\rm{LLW}}}{2d_3}.
\end{equation*}
\end{remark}

The following result is associated to condition \eqref{eq:1--1b} and will be used in the proof of   Propositions \ref{prop:overc3} and \ref{prop:estamoverc3}. 
\begin{proposition}\label{thm:u1v1111}
Let $(u_i)_{i=1}^3$  be any solution of \eqref{eq:1-1} such that $({\rm H}_{c_1,c_2,\lambda})$  holds. 
If \eqref{eq:1--1b} holds, 
then for each $\eta>0$ small,
 \begin{equation}\label{eq:spreadingly111}
\lim\limits_{t\rightarrow \infty} \sup\limits_{(\underline{c}_2-\eta) t< x<(\overline{c}_2+\eta) t} (a_{31}u_1(t,x)+a_{32}u_2(t,x))\geq \min\{a_{31},a_{32}\}.
\end{equation}
 \end{proposition}
 \begin{proof}
 Let $v(t,x)=a_{31}u_1(t,x)+a_{32}u_2(t,x)$ and denote
$$\kappa:=\max\left\{1,\frac{a_{32}a_{21}}{a_{31}}\right\}\quad\text{and}\quad \ell:=\max\left\{a_{13}a_{31},\frac{a_{23}a_{32}}{\kappa}\right\}.$$
Since  $a_{31}\leq\frac{a_{32}}{a_{12}}$ and $a_{21}<1<a_{12}$, we have $\kappa a_{31}\leq a_{32}$. Due to $d_1=1$,  by \eqref{eq:1-1} we calculate  
 \begin{align}\label{super-solution1-1}
\partial_t v-\partial_{xx}v=&a_{31}r_1u_1(1-u_1-a_{12}u_2 - a_{13}u_3)+a_{32}u_2(1-a_{21} u_1- u_2 - a_{23}u_3)\notag\\
=&a_{31}r_1u_1(1-u_1)+a_{32}u_2(1-u_2)
-\left(a_{31}a_{12}r_1+a_{32}a_{21}\right)u_1u_2\notag\\
&-\left(a_{13}a_{31}r_1u_1+a_{23}a_{32}u_2\right)u_3\\
\geq & a_{31}r_1u_1(1-u_1)+\kappa a_{31}u_2 - \kappa a_{32} u_2^2
-\left(a_{32}r_1+\kappa a_{31}\right)u_1u_2-(r_1u_1+\kappa u_2) \ell u_3\notag\\
= &(r_1u_1+\kappa u_2)\left(a_{31}-\ell u_3-v\right).\notag
\end{align}
By  \eqref{eq:H.c1c2}, 
it follows that for each small $\eta>0$,
$$\lim_{t\to\infty}v(t,(c_1-3\eta)t)=a_{31} \quad\text{and}\quad \lim_{t\to\infty}v(t,(c_2-3\eta)t)=a_{32}>a_{31},$$
and moreover,
$$\lim\limits_{t\to \infty} \sup\limits_{(\underline{c}_2-2\eta)t < x < (\sigma_1-2\eta)t} |u_3(t,x)| =0.$$
Using a  similar argument as in Step 1 of the proof of {Theorem B}, we can show that there exist some $T>0$ and $\delta\in (0,a_{31})$
such that
\begin{equation}\label{positivep}
 v(t,x) \geq \delta\quad \text{in}\quad \Omega_T:=\{(t,x): t \geq T,\,  (\underline{c}_2-3\eta)t \leq  x \leq (\sigma_1-3\eta) t\}.
\end{equation}
Let $\underline{v}(t,x)$ denote the unique solution of
$$\partial_t \underline{v}-\partial_{xx}\underline{v}=(r_1u_1+\kappa u_2)\left(a_{31}-\ell  u_3-\underline{v}\right) \text{ in }\Omega_T\quad\text{and}\quad \underline{v}=\delta   \text{ on } \partial\Omega_T,$$
for which $v(t,x)$ is  obviously 
a super-solution due to \eqref{super-solution1-1} and \eqref{positivep}. By comparison, we arrive at 
$\underline{v}\leq v$ in $\bar{\Omega}_T$, so that noting that $a_{31}\leq a_{32}$, 
it suffices to show 
\begin{equation}\label{liu112}
  \lim\limits_{t\rightarrow \infty} \sup\limits_{(\underline{c}_2-\eta) t< x<(\overline{c}_2+\eta) t} \underline{v}(t,x)= a_{31}.
\end{equation}

The parabolic maximum principle implies  that $\underline{v}\leq a_{31}$ in $\bar{\Omega}_T$.
 Suppose \eqref{liu112} fails. Then there exists $(t_n,x_n)$ such that
\begin{equation}\label{liu113}
        c_n:= x_n/t_n \to c \in (\underline{c}_2-2\eta, \overline{c}_2+2\eta) \quad\text{ and }\quad \lim\limits_{n\to\infty}\underline{v}(t_n,x_n)<a_{31}.
\end{equation}
Denote $\underline{v}_{n}(t,x):= \underline{v}(t_n + t, x_n + x)$ and $(u_{1,n},u_{2,n},u_{3,n})(t,x):= (u_1,u_2,u_3)(t_n + t, x_n + x)$. By parabolic estimates we see that $\underline{v}_{n}$ and $(u_{1,n},u_{2,n},u_{3,n})$ are precompact in $C^2_{\mathrm{loc}}(K)$ for each compact subset $K \subset \mathbb{R}^2$.
Note that
$u_{3,n} \to 0$ as $n\to\infty$ and by \eqref{positivep}, for some $\tilde{\delta}>0$,
$$\liminf_{n\to\infty} (r_1u_{1,n}+\kappa u_{2,n})\geq \min\left\{\frac{r_1}{a_{31}},\frac{\kappa }{a_{32}}\right\}\liminf_{n\to\infty} v(t_n + t, x_n + x)>\tilde{\delta}.$$
 Passing to a subsequence if necessarily,   we assume that $\underline{v}_{n} \to \hat{\underline{v}}$ in $C^2_{\mathrm{loc}}(\mathbb{R}^2)$, which satisfies
 $$ \delta \leq \hat{\underline{v}} \leq a_{31}\quad \text{ and }\quad \partial_t \hat{\underline{v}}-\partial_{xx}\hat{\underline{v}}\geq\tilde{\delta}\left(a_{31}-\hat{\underline{v}}\right) \quad\text{in }\,\,\mathbb{R}^2.$$
By the parabolic maximum principle, we deduce that $\hat{\underline{v}} \equiv   a_{31}$ in $\mathbb{R}^2$, and particularly,
  $\lim\limits_{n\to\infty}\underline{v}(t_n,x_n)=\hat{\underline{v}}(0,0) = a_{31},$
   which  contradicts \eqref{liu113}.  Therefore, \eqref{liu112} is established. 
\end{proof}

\begin{remark}
We mention that condition \eqref{eq:1--1b} 
in {Theorem A} is nearly necessary 
to guarantee
\eqref{eq:spreadingly111}.
Indeed, for any traveling wave solution $(\tilde{u}_1,\tilde{u}_2)$ of \eqref{eq:2sp},
the lower bound of $a_{31}\tilde{u}_1+a_{32}\tilde{u}_2$ 
was considered in \cite{ChenHung_2016}. 
Setting $\underline{\tilde{u}}_1=1$ and $\underline{\tilde{u}}_2=\frac{1}{a_{12}}$ in  \cite[Theorem 1.2]{ChenHung_2016} yields 
$$a_{31}\tilde{u}_1+a_{32}\tilde{u}_2\geq \min\left\{a_{31},\frac{a_{32}}{a_{12}}\right\}\frac{\min(d_1,1)}{\max(d_1,1)}.$$
To ensure $(\tilde{u}_1,\tilde{u}_2)$ satisfies \eqref{eq:spreadingly111},  we require
$$\min\left\{a_{31},\frac{a_{32}}{a_{12}}\right\}\frac{\min(d_1,1)}{\max(d_1,1)}\geq \min\{a_{31},a_{32}\},$$
which in turn implies that $d_1=1$  and $a_{31}\leq\frac{a_{32}}{a_{12}}$.
\end{remark}

\section{Proofs of Lemma \ref{lemA1rho} and Proposition \ref{lemA3rho}}
\label{sec:B}


\begin{proof}[Proof of Lemma {\rm\ref{lemA1rho}}]
We only show uniqueness, as the existence of $\hat\rho$ is standard \cite[Theorem 2]{Crandall_1986}.
We divide the proof into  two steps by distinguishing the cases $\hat \lambda\in(0,\infty)$ and   $\hat \lambda=\infty$.

\noindent {\bf Step 1.} 
We prove Lemma \ref{lemA1rho} when $\hat \lambda\in(0,\infty)$. In this case the uniqueness is proved in Lemma \ref{lemA2rho}.
It remains to show that assertions {\rm{(a)}} and {\rm{(b)}} hold.
%

To this end, we first define $\underline{\rho}_1\in C(0,\infty)$ as follows:
\begin{itemize}
\item[{\rm{(i)}}] If $\hat\lambda\leq \sqrt{\frac{\hat{r}}{\hat{d}}}$, then
${\underline{\rho}_1}(s):= \max\left\{\hat\lambda s-(\hat d\hat\lambda^2+\hat r),0\right\};$
\item[{\rm{(ii)}}] If $\hat\lambda> \sqrt{\frac{\hat{r}}{\hat{d}}}$, then
$${\underline{\rho}_1}(s):=
\begin{cases}
\smallskip
\hat\lambda s-(\hat d\hat\lambda^2+\hat r) & \text{ for } s\geq 2\hat{d}\hat\lambda,\\
\smallskip
\frac{s^2}{4\hat d}-\hat r & \text{ for } 2\sqrt{\hat{d}\hat{r}} \leq s< 2\hat{d}\hat\lambda,\\
0 & \text{ for }0\leq s<2\sqrt{\hat{d}\hat{r}}.
\end{cases}
$$
\end{itemize}
It is straightforward to verify  ${\underline{\rho}_1}$ is a viscosity sub-solution (in fact a viscosity solution) of
$$\min\{\rho-s\rho'+\hat d|\rho'|^2+\hat r,\rho\}=0 \,\,\text{ in }(0,\infty).$$
In view of $g\geq 0$, we conclude that ${\underline{\rho}_1}$ is a viscosity sub-solution of \eqref{hatrho0}.

Set $g_{\max}:=\max\left\{\sup\limits_{(0,\infty)}g, \frac{(c_{g})^2}{4\hat d}\right\}$.  We define ${\overline{\rho}_1}\in C(0,\infty)$ as follows:
\begin{itemize}
\item[\rm{(i)}]
If $\hat\lambda>\frac{c_{g}}{2\hat{d}}$, then
$${\overline{\rho}_1}:=
\begin{cases}
\smallskip
\hat\lambda s-(\hat d\hat\lambda^2+\hat r) & \text{ for }s\geq 2\hat{d}\hat\lambda,\\ 
\smallskip
\frac{s^2}{4\hat d}-\hat r & \text{ for }c_{g}\leq s< 2\hat{d}\hat\lambda, \\
\hat{\lambda}_1s-(\hat d \hat{\lambda}_{1}^2+\hat{r}-g_{\max}) &  \text{ for }0\leq s< c_{g}, \\  
\end{cases}
  $$
with $\hat{\lambda}_1=\frac{c_{g}}{2\hat d}-\sqrt{\frac{g_{\max}}{\hat d}}$; 
  \item[\rm{(ii)}]  If $\hat\lambda\leq\frac{c_{g}}{2\hat{d}}$, then
$${\overline{\rho}_1}:=
\begin{cases}
\smallskip
\hat\lambda s-(\hat d\hat\lambda^2+\hat r) & \text{ for }s\geq c_{g},\\
\hat{\lambda}_2s-(\hat{d}\hat{\lambda}_2^2+\hat{r}-g_{\max}) & \text{ for }0\leq s< c_{g}, \\
\end{cases}
  $$
   with $\hat{\lambda}_2=\frac{c_{g}-\sqrt{(c_{g}-2\hat{d}\hat\lambda)^2+4\hat{d}g_{\max}}}{2\hat{d}}$.
  \end{itemize}
We shall verify that ${\overline{\rho}_1}$ defined above is  a viscosity super-solution of \eqref{hatrho0} for case \rm{(i)}, and then a similar verification can be made for case \rm{(ii)}. Since $\mathrm{spt}\, g\subset\left[0, c_{g}\right]$, by the definition of $g_{\max}$, it suffices to check ${\overline{\rho}_1}$ is a viscosity super-solution of
\begin{equation}\label{auxillaryrho}
  \min\left\{\rho-s\rho'+\hat d|\rho'|^2+\hat r-g_{\max}\chi_{\{0<s<c_{g}\}},\rho\right\}=0 \,\,\text{ in }\,\,(0,\infty).
\end{equation}

By construction, ${\overline{\rho}_1}$ is continuous and nonnegative in $[0,\infty)$. We see that ${\overline{\rho}_1}$ is a classical (and thus viscosity) solution of \eqref{auxillaryrho} whenever $s\neq c_{g}$. It remains to consider the case when ${\overline{\rho}_1} - \phi$ attains a strict local minimum at $s=c_{g}$, where $\phi\in C^1(0,\infty)$ is any test function.

In such a case, noting that $(\hat r-g_{\max}\chi_{\{0<s<c_{g}\}})^* = \hat r$ at $s = c_g$, we calculate  at ${s}=c_{g}$ that
\begin{align*} 
{\overline{\rho}_1}-c_{g}\phi' +\hat{d}|\phi'|^2 + \hat{r} &=\frac{(c_{g})^2}{4\hat{d}}-\hat{r}-c_{g}\phi' +\hat{d}| \phi'|^2+\hat{r} = \hat{d}\left( \phi'-\frac{c_{g}}{2\hat{d}}\right)^2 \geq 0.
\end{align*}
Hence {$\overline\rho_1$} is a viscosity super-solution of \eqref{auxillaryrho}, and thus of \eqref{hatrho0}.

Observe also that $$\limsup\limits_{s\to\infty}\frac{{\underline{\rho}_1}(s)}{s}=\limsup\limits_{s\to\infty}\frac{\hat\rho(s)}{s}=\liminf\limits_{s\to\infty}\frac{{\overline{\rho}_1}(s)}{s}=\hat\lambda.$$ To apply Lemma \ref{lemA2rho},
 we shall verify  $ {\underline{\rho}_1}(0)\leq \hat\rho(0)\leq {\overline{\rho}_1}(0)$. For the case $\hat\lambda>\frac{c_{g}}{2\hat{d}}$, we calculate
\begin{equation}\label{calculation0}
\begin{split}
{\overline{\rho}_1}(0)=&-(\hat d \hat{\lambda}_1^2+\hat{r}-g_{\max})\\
\geq&\hat{\lambda}_1(c_{g}-\hat d\hat{\lambda}_1)-(\hat{r}-g_{\max})\\
=&\hat d\left(\frac{c_{g}}{2\hat d}-\sqrt{\frac{g_{\max}}{\hat d}}\right)\left(\frac{c_{g}}{2\hat d}+\sqrt{\frac{g_{\max}}{\hat d}}\right)-(\hat{r}-g_{\max})\\
=&\frac{(c_{g})^2}{4\hat d}-\hat r\geq 0=\hat\rho(0)={\underline{\rho}_1}(0),
\end{split}
\end{equation}
where the first inequality is due to $\hat{\lambda}_1\leq0$, and similar verification can be performed for the  case $\hat\lambda\leq\frac{c_{g}}{2\hat{d}}$.

Therefore,   $\overline\rho_1$ and  $\underline\rho_1$ defined above  are a pair of viscosity super- and sub-solutions of \eqref{hatrho0}. Observe from the expressions of ${\underline{\rho}_1}$ and ${\overline{\rho}_1}$ that ${\overline{\rho}_1}={\underline{\rho}_1}$ in $[c_{g},\infty)$ and satisfies assertions {\rm{(a)}} and {\rm{(b)}}. Let $\hat\rho$ be any viscosity solution of \eqref{hatrho0}. Since
 ${\underline{\rho}_1}\leq \hat\rho \leq {\overline{\rho}_1}$ in $[0,\infty)$
by Lemma \ref{lemA2rho}, the assertions {\rm{(a)}} and {\rm{(b)}} hold for $\hat\rho$ in $[c_{g},\infty)$. 
Step 1 is thus completed.

\smallskip
\noindent {\bf Step 2.} We prove Lemma \ref{lemA1rho} for the case $\hat \lambda=\infty$.
First, we show that for any viscosity solution $\hat \rho$  of \eqref{hatrho0} with $\hat \lambda=\infty$, it follows that
\begin{equation}\label{goal}
  \hat\rho(s)=
\frac{s^2}{4\hat d}-\hat r\,\,\text{ for }\,\,s\geq c_{g}.
\end{equation}

To this end, we shall adopt the same strategy  as in Step 1 by constructing suitable viscosity super- and sub-solutions of \eqref{hatrho0}.
For any $\lambda> \sqrt{\hat{r}/\hat{d}}$, we define $\underline{\rho}_\lambda\in C(0,\infty)$ by
\begin{equation}\label{rholambda}
  \underline{\rho}_\lambda:=
\begin{cases}
\smallskip
\lambda s-(\hat d\lambda^2+\hat r) & \text{ for } s\geq 2\hat{d}\lambda,\\
\smallskip
\frac{s^2}{4\hat d}-\hat r& \text{ for } 2\sqrt{\hat{d}\hat{r}} \leq s< 2\hat{d}\lambda,\\
0& \text{ for }0\leq s<2\sqrt{\hat{d}\hat{r}},
\end{cases}
\end{equation}
which can be verified directly to be a viscosity sub-solution of \eqref{hatrho0}. In view of $\underline{\rho}_\lambda(0)=0=\hat\rho(0)$ and
$\limsup\limits_{s\to\infty}\frac{\underline{\rho}_\lambda(s)}{s}=\lambda<\infty=\limsup\limits_{s\to\infty}\frac{\hat\rho(s)}{s}$, we apply Lemma \ref{lemA2rho} with $c_{b}=\infty$ 
to deduce 
$$\begin{array}{l}
\underline{\rho}_\lambda\leq \hat\rho\quad\text{ in }[0,\infty)\quad\text{ for all }\,\,\,\lambda> \sqrt{\hat{r}/\hat{d}},
\end{array}
$$
where letting $\lambda\to\infty$, together with the expression of $\underline{\rho}_\lambda$ in \eqref{rholambda}, gives
\begin{equation}\label{goal1}
  \hat\rho(s)\geq
\frac{s^2}{4\hat d}-\hat r\quad \text{ for }\,\,s\geq c_{g}.
\end{equation}

To proceed further, for any $\epsilon>0$ and $s_0>c_{g}$, we define $\overline{\rho}_{\epsilon,s_0}\in C([0,s_0])$ by
\begin{equation}\label{rhos0}
  \overline{\rho}_{\epsilon,s_0}:=
\begin{cases}
\medskip
\frac{s^2}{4\hat d}-\hat r+\frac{\epsilon}{s_0-s}& \text{ for } c_{g}\leq s<s_0,\\
\medskip
\hat{\lambda}_1 s-(\hat d\hat{\lambda}_1^2+\hat r-g_{\max}-\frac{\epsilon}{s_0-c_{g}})& \text{ for } 0 \leq s< c_{g},
\end{cases}
\end{equation}
where $\hat{\lambda}_1 =\frac{c_{g}}{2\hat d}-\sqrt{\frac{g_{\max}}{\hat d}}$ is defined in Step 1. Similar to Step 1, we can verify that
 $\overline{\rho}_{\epsilon,s_0}$  defines a viscosity super-solution of  \eqref{hatrho0} in $(0,s_0)$ for each $s_0>c_{g}$.
By \eqref{calculation0}, one can check $\overline{\rho}_{\epsilon,s_0}(0)\geq 0=\hat{\rho}(0)=0$. Moreover, since
\begin{equation*}
\frac{\hat\rho(s_0)}{s_0}< \infty=\liminf\limits_{s\to s_0}\frac{\overline{\rho}_{\epsilon,s_0}(s)}{s},
\end{equation*}
 we apply Lemma \ref{lemA2rho} with $c_{b}=s_0$ to get
$\hat\rho(s)\leq\overline{\rho}_{\epsilon,s_0}(s)$ for $s\in[0,s_0]$.
Letting $\epsilon \to 0$ and then $s_0\to\infty$, we have
$
\hat\rho(s)\leq\frac{s^2}{4\hat{d}}-\hat{r}$ for $s\in[c_{g},\infty)$, 
which together with \eqref{goal1} implies \eqref{goal}.

Finally, we apply Lemma \ref{lemA2rho} to show that $\hat\rho$ is also uniquely determined in
$[0,c_{g}]$. Thus $\hat\rho$ is unique, and 
the proof of Lemma \ref{lemA1rho} is now complete.
\end{proof}

Next, we prove Proposition \ref{lemA3rho} for the remaining case $\hat\lambda = \infty$.
\begin{proof}[Proof of Proposition {\rm\ref{lemA3rho}} for case $\hat\lambda = \infty$]
In this case, $\tilde{w}(t,x)$ is a viscosity super-solution (resp. sub-solution) of the equation
\begin{align}\label{hatwinfty}
\left\{\begin{array}{ll}
\smallskip
\min\{\partial_tw+\hat d|\partial_xw|^2+\hat r-g\left(\frac{x}{t}\right),w\}=0 \,\,\text{ in }(0,\infty)\times(0,\infty),\\
\smallskip
 w(0,x)=
 \left\{\begin{array}{ll}
 0 &\text{ for } x=0,\\
\infty &\text{ for } x\in  (0,\infty),
\end{array}
\right.\quad
w(t,0)=0 \quad \mathrm{ on }\,\,\, [0,\infty).
\end{array}
\right.
\end{align}
The initial condition is  understood in the sense that  $\tilde{w}(t,x) \to \infty$ if $(t,x) \to (0,x_0)$ for $x_0>0$.

\smallskip
 \noindent {\bf Step 1.}  Let  $\tilde{w}(t,x)$ be a viscosity super-solution of \eqref{hatwinfty}. We show
$\tilde{w}(t,x)\geq t\hat\rho\left(\frac{x}{t}\right) $ in $(0,\infty)\times(0,\infty)$, where $\hat\rho$ is the unique viscosity solution of \eqref{hatrho0}. 

 We first prove $\tilde{w}(t,x)\geq t\hat\rho\left(\frac{x}{t}\right)$ for $x\geq c_{g}t$, where $c_{g}$ is given in {$\rm(H_{g})$}.
Recall from   Step 2 in the proof of Lemma \ref{lemA1rho} that $\underline{\rho}_\lambda$ defined by \eqref{rholambda} is a viscosity sub-solution of
  \begin{equation*}
    \min\{\rho-s\rho'+\hat d|\rho'|^2+\hat r-g(s),\rho\}=0 \,\,\text{ in } (0,\infty),
  \end{equation*}
for all $\lambda>\sqrt{\hat{r}/\hat{d}}$, whence, by a standard verification as in Lemma \ref{lemA2rho}, we may conclude that $t\underline{\rho}_\lambda\left(\frac{x}{t}\right)$ is a viscosity sub-solution to \eqref{hatwinfty}. Observe that 
$$t\underline{\rho}_\lambda\left(0\right)=0\leq \tilde{w}(t,0)\,\,\,\,\text{ and } \,\,\lim_{t\to 0}\left[t\underline{\rho}_\lambda\left(\frac{x}{t}\right)\right]=\lambda x\leq  \tilde{w}(0,x).$$
We apply \cite[Theorem A.1]{LLL20192} to deduce that for all $\lambda>\sqrt{\hat{r}/\hat{d}}$,
\begin{equation}\label{comparisonresult}
  \tilde{w}(t,x)\geq t\underline{\rho}_\lambda\left(\frac{x}{t}\right)\,\,\text{ in }(0,\infty)\times(0,\infty).
\end{equation}
By Lemma \ref{lemA1rho},  we deduce that
$\underline{\rho}_\lambda(s)\to \frac{s^2}{4\hat d}-\hat r=\hat{\rho}(s)$  as $\lambda\to\infty$ for $s\in[c_{g},\infty)$, 
so that letting $\lambda\to\infty$ in \eqref{comparisonresult} gives $ \tilde{w}(t,x)\geq t \hat\rho\left(\frac{x}{t}\right)$ for $x\geq c_{g}t$.

To complete Step 1, it remains to show $\tilde{w}(t,x)\geq t\hat\rho\left(\frac{x}{t}\right)$ for $ 0\leq x\leq c_{g}t$.
Note that $\tilde{w}$ is a viscosity super-solution of the problem
 \begin{align}\label{boundedproblem}
\left\{\begin{array}{ll}
\medskip
\min\{\partial_tw+\hat d|\partial_xw|^2+\hat r-g\left(x/t\right),w\}=0 & \text{  for }0< x< c_{g}t,\\
w(t,0)=0, \,\,\,\, w(t,c_{g}t)= t\hat\rho\left(c_{g}\right) & \text{ for } t\geq 0,
\end{array}
\right.
\end{align}
while,  by direct verification,  $t\hat\rho\left(\frac{x}{t}\right)$ defines a viscosity solution to \eqref{boundedproblem}. Once again we apply \cite[Theorem A.1]{LLL20192} to derive that $\tilde{w}(t,x)\geq t\hat\rho\left(\frac{x}{t}\right)$ for $ 0\leq x\leq c_{g}t$, which completes Step 1.

\smallskip
 \noindent {\bf Step 2.} Let $\tilde{w}(t,x)$ be a viscosity sub-solution of \eqref{hatwinfty}. We show that
 $\tilde{w}(t,x)\leq t\hat\rho\left(\frac{x}{t}\right)$ in $(0,\infty)\times(0,\infty)$.
For any $\epsilon>0$ and $s_0>c_{g}$,  we see that $\overline{\rho}_{\epsilon,s_0}$ given by \eqref{rhos0} 
is a viscosity super-solution of
$$\min\left\{\rho-s\rho'+\hat d|\rho'|^2+\hat r-g(s),\rho\right\}=0 \,\,\text{ in } (0,s_0),$$
from which we can verify  that $t\overline{\rho}_{\epsilon,s_0}\left(\frac{x}{t}\right)$ is a viscosity super-solution to \eqref{hatwinfty} for $0<x<s_0 t$. By \cite[Theorem A.1]{LLL20192} again, we arrive at
 \begin{equation}\label{eq:liu001}
  \tilde{w}(t,x)\leq t\overline{\rho}_{\epsilon,s_0}\left(\frac{x}{t}\right) \quad\text{for } \,\,0<x<s_0 t.
\end{equation}
Letting $\epsilon\to0$ and then $s_0\to\infty$ in \eqref{eq:liu001} (as in the proof of Lemma \ref{lemA1rho}), and noting that $\overline{\rho}_{\epsilon,s_0}(s)\to \frac{s^2}{4\hat d}-\hat r=\hat{\rho}(s)$ for $s\in[c_{g},\infty)$, we deduce 
  $\tilde{w}(t,x)\leq t \hat\rho\left(\frac{x}{t}\right)$ for $x\geq c_{g}t$.
Finally, the fact that $\tilde{w}(t,x)\leq t\hat\rho\left(\frac{x}{t}\right)$ for $ 0\leq x\leq c_{g}t$ can be proved by the same arguments as in Step 1. Step 2 is now complete and Proposition \ref{lemA3rho} is proved.
\end{proof}

\section{Explicit formula for $s_{{\rm nlp}}(c_1,c_2,\lambda)$}
\label{appendix_C}

This section is devoted to determining $s_{{\rm nlp}}(c_1,c_2,\lambda)$ defined in Definition \ref{def:snlp} and proving {Theorem C}. 
We recall that $\rho_{{\rm nlp}}$ is the unique viscosity solution of \eqref{eq:w3nlp} and 
$s_{{\rm nlp}} = \sup\{s \geq 0:\rho_{{\rm nlp}}(s) = 0\}$.
By Lemma \ref{lemA1rho} (with  $c_g=c_1$) one can rewrite $\rho_{\rm nlp}$ explicitly as
\begin{equation}\label{rhosigma_1}
  \rho_{\rm{nlp}}(c_1) =\zeta_1 c_1 - d_3 \zeta^2_1 - r_3(1-a_{31}),
\end{equation}
where $\alpha_3=2\sqrt{d_3r_3}$ and 
$\zeta_1$ is defined in \eqref{lambda3}.
In this way, we can also regard $\rho_{\rm nlp}$ as the unique viscosity solution of
\begin{align}\label{eq:hj}
\begin{cases}
\medskip
\min\{\rho -s\rho'+ d_3 |\rho'|^2 +\mathcal{R}(s), \rho\} =0 \,\,\text{ in } (0,c_1),\\
\rho(0)=0,\quad \rho(c_1)=\zeta_1 c_1 - d_3 \zeta^2_1 - r_3(1-a_{31}),
\end{cases}
\end{align}
where $\mathcal{R}(s)=r_3(1-a_{31}\chi_{\{c_2\leq s\leq c_1\}}-a_{32}\chi_{\{s\leq c_2\}})$.

To prove Theorem C, we first present some sufficient and necessary conditions for $s_{{\rm nlp}} > \alpha_3\sqrt{1-a_{32}}$. Recall the definition of $\zeta_1$ in \eqref{lambda3}.

\begin{proposition}\label{prop:A}
If $\zeta_1 \leq \frac{c_2}{2d_3}$, then
$s_{{\rm nlp}} > \alpha_3\sqrt{1-a_{32}}$ if and only if
\begin{equation}\label{eq:propA1}
\zeta_1 c_2 - d_3 \zeta_1^2 - r_3(1-a_{31}) < c_2 \sqrt{\frac{r_3(1-a_{32})}{d_3}} - 2r_3(1-a_{32}).
\end{equation}
In this case, the unique viscosity solution $\rho_{\rm{nlp}}$ of \eqref{eq:hj} is given by
\begin{equation}\label{eq:propA2}
\rho_{\rm{nlp}} =
\begin{cases}
\smallskip
\zeta_1s - d_3 \zeta_1^2 - r_3(1-a_{31})&\text{ for }c_2 <s\leq c_1,\\ 
\smallskip
\lambda_{{\rm nlp2}}(s- s_{{\rm nlp}}) &\text{ for } s_{{\rm nlp}} < s\leq c_2,\\
0 &\text{ for }0\leq s \leq s_{{\rm nlp}},
\end{cases}
\end{equation}
where 
$
\lambda_{{\rm nlp2}} = \frac{c_2-\sqrt{\left(c_2-2d_3\zeta_1\right)^2+\alpha_3^2(a_{32}-a_{31})}}{2d_3}$ and $s_{{\rm nlp}} = d_3 \lambda_{{\rm nlp2}} + \frac{r_3(1-a_{32})}{\lambda_{{\rm nlp2}}}.
$
\end{proposition}
\begin{remark}\label{rmk:prop1}
Suppose {that} \eqref{eq:propA1} holds. It is straightforward to check that
$\zeta_1 c_2 - d_3 \zeta_1^2- r_3(1-a_{31}) <  \frac{c^2_2}{4d_3}-r_3(1-a_{32}),$
which implies
\begin{equation}\label{welldefineof}
  \left(c_2-2d_3\zeta_1\right)^2+\alpha^2_3(a_{32}-a_{31})>0,
\end{equation}
 so that $\lambda_{{\rm nlp2}}$ is well defined.
By direct calculation, we can verify that \eqref{eq:propA1} is equivalent to \eqref{welldefineof} and $\lambda_{{\rm nlp2}}<\sqrt{\frac{r_3(1-a_{32})}{d_3}}$, whence $s_{{\rm nlp}}>\alpha_3\sqrt{1-a_{32}}$.
\end{remark}

\begin{proof}[Proof of Proposition {\rm\ref{prop:A}}]
We divide the proof into the following two steps.

\smallskip
\textbf{Step 1.} We assume \eqref{eq:propA1}  holds and show $s_{{\rm nlp}} > \alpha_3\sqrt{1-a_{32}}$.
Denote by $\hat{\rho}$  the right hand of \eqref{eq:propA2}. We prove that $\hat{\rho}$ is a viscosity solution of \eqref{eq:hj}.
By construction, $\hat{\rho}$ is continuous in $[0,c_1]$.  Indeed, $\hat{\rho}$ is a classical solution for \eqref{eq:hj} whenever $s \not \in \{c_2,s_{\rm{nlp}}\}$. We claim that $\hat{\rho}$ is a viscosity super-solution of \eqref{eq:hj}. For this purpose, suppose $\hat{\rho} - \phi$ attains a strict local minimum at $s_0 \in \{c_2,s_{\rm{nlp}}\}$.  If $s_0 = s_{\rm nlp}$, then $0 \leq \phi'(s_{\rm nlp})\leq \lambda_{\rm nlp2}$, and therefore  at $s = s_{\rm nlp}$,
$$
\hat\rho( s_{\rm nlp}) - s_{\rm nlp}\phi' + d_3 |\phi'|^2 + \mathcal{R}^*(c_2) \geq  d_3(\phi' - \lambda_{\rm nlp2}) \left( \phi' - \frac{r_3(1-a_{32})}{d_3 \lambda_{\rm nlp2}}\right) \geq 0,
$$
where we used $\hat\rho(s_{\rm nlp} ) = 0$ for the first inequality and the last inequality is a consequence of $\phi' - \frac{r_3(1-a_{32})}{d_3 \lambda_{\rm nlp2}}\leq \phi' - \lambda_{\rm nlp2} \leq  0$ (by Remark \ref{rmk:prop1}).

In case  $s_0 = c_2$, we have $\phi'(c_2) \leq  \zeta_1$, so that  when evaluated at $s = c_2$,
\begin{align*}
\hat\rho(c_2) - c_2 \phi'(c_2) + d_3 |\phi'(c_2)|^2 + \mathcal{R}^*(c_2) &\geq \zeta_1 c_2 - d_3 \zeta_1^2 - c_2 \phi'(c_2) + d_3 |\phi'(c_2)|^2\\
&= d_3 \left(\phi'(c_2) + \zeta_1 - \frac{c_2}{d_3}\right)(\phi'(c_2) - \zeta_1)  \geq 0,
\end{align*}
where we used $\phi'(c_2) \leq \zeta_1$ and $\phi'(c_2) + \eta_1 - \frac{c_2}{d_3} \leq 2 \zeta_1 - \frac{c_2}{d_3} \leq 0$ for last inequality.

It remains to show that $\hat{\rho}$ is also a viscosity sub-solution of \eqref{eq:hj}.  We assume $\hat{\rho} - \phi$ attains a strict local maximum at $s= c_2$ for some test function $\phi\in C^1(0,\infty)$, and $\hat{\rho}(c_2)>0$.
Observe that $(\hat{\rho}-\phi)(s) \leq (\hat{\rho}-\phi)(c_2)$ for $s\approx c_2$, so that
$\zeta_1\leq  \phi'(c_2) \leq \lambda_{\rm nlp2}$. Therefore,  evaluating at $s=c_2$, by \eqref{eq:propA2} we calculate that 
\begin{align*}
\hat{\rho} -c_2\phi'+ d_3 |\phi'|^2 +\mathcal{R}_*(c_2)&=\hat{\rho}-c_2\phi' +d_3|\phi'|^2 + r_3(1 -\max\{a_{31},a_{32}\})\\
&\leq\zeta_1c_2-d_3\zeta_1^2-c_2\phi' +d_3 | \phi'|^2 \\
&= d_3\left(\phi'+\zeta_1-\frac{c_2}{d_3}\right)(\phi'-\zeta_1) \leq 0,
\end{align*}
where the last inequality follows from $\zeta_1\leq \phi'(c_2)\leq\lambda_{\rm{nlp2}}\leq \frac{c_2}{2d_3}$. Hence, $\hat{\rho}$, defined by the right hand side of \eqref{eq:propA2}, is a viscosity super- and sub-solution, and thus a viscosity solution of \eqref{eq:hj}.
Since $\rho_{\rm{nlp}}(c_1)=\hat\rho(c_1)=0$ and $\rho_{\rm{nlp}}(0)=\hat\rho(0)$,  by uniqueness of viscosity solution (from Lemma \ref{lemA1rho}), we deduce
$\rho_{\rm{nlp}}(s)=\hat\rho(s)$ for $[0,c_1)$.
By Remark  \ref{rmk:prop1},  we deduce that $s_{\rm{nlp}}>\alpha_3\sqrt{1-a_{32}}$. Step 1 is completed.

\textbf{ Step 2.} We assume \eqref{eq:propA1} fails and  show $s_{{\rm nlp}} \leq \alpha_3\sqrt{1-a_{32}}$.
Suppose that \eqref{eq:propA1} fails, i.e.
\begin{equation}\label{eq:propA1'}
  \zeta_1 c_2 - d_3 \zeta_1^2 - r_3(1-a_{31}) \geq c_2 \sqrt{\frac{r_3(1-a_{32})}{d_3}} - 2r_3(1-a_{32}).
\end{equation}
Since $\zeta_1 \leq  \frac{c_2}{2d_3}$, the left hand side of \eqref{eq:propA1'} is increasing in $\zeta_1$ in $\left[0,\frac{c_2}{2d_3}\right]$, and the right hand side of \eqref{eq:propA1'} is greater than equal to $-r_3(1-a_{32})$ (since $c_2 \geq \alpha_3 \sqrt{1-a_{32}}$),
we  choose $0 < \tilde\zeta_1\leq \zeta_1\leq  \frac{c_2}{2d_3}$ to satisfy
$
\tilde\zeta_1 c_2 - d_3 \tilde\zeta_1^2 - r_3(1-a_{31})= c_2 \sqrt{\frac{r_3(1-a_{32})}{d_3}} - 2r_3(1-a_{32}).
$
Define 
\begin{equation}\label{eq:propA3}
{\underline{\rho}_3}(s): =
\begin{cases}
\smallskip
\tilde\zeta_1s - d_3 \tilde\zeta_1^2 - r_3(1-a_{31}) &\text{ for }c_2 < {s} \leq c_1,\\
\smallskip
s\sqrt{\frac{r_3(1-a_{32})}{d_3}} - 2r_3(1-a_{32}) &\text{ for } \alpha_3\sqrt{1-a_{32}} < {s} \leq c_2,\\
0 &\text{ for } 0\leq {s} \leq \alpha_3\sqrt{1-a_{32}}.
\end{cases}
\end{equation}
By the choice of $\tilde\zeta_1$, we have $\rho_{\rm nlp}(c_1) \geq {\underline\rho_3(c_1)}$.
Using the same arguments given in Step 1, we may check that ${\underline{\rho}_3}$ is a viscosity sub-solution of \eqref{eq:hj}.
Together with $\rho_{\rm{nlp}}(0) ={\underline{\rho}_3}(0)$, by applying Lemma \ref{lemA2rho} once again, we get
$\rho_{\rm{nlp}}(s)\geq {\underline{\rho}_3}(s)$ for $ s\in [0,c_1]$.
Since $\rho_{\rm{nlp}}(s)\geq{\underline{\rho}_3}(s) >0$ for $s > \alpha_3\sqrt{1-a_{32}}$, we deduce $s_{\rm{nlp}}\leq \alpha_3\sqrt{1-a_{32}}$.
%
\end{proof}

\begin{proposition}\label{prop:B}
If $\zeta_1 > \frac{c_2}{2d_3}$, then
$s_{{\rm nlp}} > \alpha_3\sqrt{(1-a_{32})}$ if and only if
\begin{equation}\label{eq:propB1}
\frac{(c_2)^2}{4d_3} - r_3(1-a_{31})< c_2 \sqrt{\frac{r_3(1-a_{32})}{d_3}} - 2r_3(1-a_{32}).
\end{equation}
In this case, the unique viscosity solution $\rho_{\rm{nlp}}$ of \eqref{eq:hj} is given by
\begin{equation}\label{eq:propB2}
\rho_{\rm{nlp}} =
\begin{cases}
\smallskip
\zeta_1s- d_3 \zeta_1^2 - r(1-a_{31}) &\text{ for }2d_3\zeta_1 < {s} \leq c_1,\\
\smallskip
\frac{s^2}{4d_3} - r_3(1-a_{31})&\text{ for }c_2 < {s} \leq 2d_3\zeta_1,\\
\smallskip
\lambda_{{\rm nlp1}}(s - s_{{\rm nlp}}) &\text{ for } s_{{\rm nlp}} < {s} \leq  c_2,\\
0 &\text{ for } 0\leq {s} \leq s_{{\rm nlp}},
\end{cases}
\end{equation}
where
$
\lambda_{{\rm nlp1}} = \frac{c_2}{2d_3} - \sqrt{\frac{r_3(a_{32} - a_{31})}{d_3}}$ and $ s_{{\rm nlp}} = d_3 \lambda_{{\rm nlp1}} + \frac{r_3(1-a_{32})}{\lambda_{{\rm nlp1}}}.
$
\end{proposition}
\begin{remark}\label{rmk:prop2}
The condition \eqref{eq:propB1} is equivalent to
\begin{equation}\label{equivalent00}
  a_{32} > a_{31}\,\,\text{ and }\,\,\lambda_{{\rm nlp1}}<\sqrt{\tfrac{r_3(1-a_{32})}{d_3}},
\end{equation}
which implies that $\lambda_{{\rm nlp1}}$ is well defined and $s_{{\rm nlp}} = d_3 \lambda_{{\rm nlp1}} + \frac{r_3(1-a_{32})}{\lambda_{{\rm nlp1}}}>\alpha_3\sqrt{1-a_{32}}$.
\end{remark}

\begin{proof}[Proof of Proposition {\rm\ref{prop:B}}]
Under \eqref{eq:propB1}, by the same arguments as in Step 1 of Proposition \ref{prop:A}, we can verify  $\rho_{\rm{nlp}}$ given by \eqref{eq:propB2} defines the unique viscosity solution of \eqref{eq:hj}. Then $s_{{\rm nlp}} >\alpha_3\sqrt{1-a_{32}}$ follows from Remark \ref{rmk:prop2}.
It remains to assume \eqref{eq:propB1} fails and to show $s_{{\rm nlp}} \leq \alpha_3\sqrt{1-a_{32}}$. In this case,
 $\frac{c_2^2}{4d_3} - r_3(1-a_{31}) \geq c_2 \sqrt{\frac{r_3(1-a_{32})}{d_3}} - 2r_3(1-a_{32}). 
 $
 Since $\zeta\mapsto c_2 \zeta - d_3 \zeta^2$ attains maximum value $\frac{c^2_2}{4d_3}$ at $\zeta =\frac{c_2}{2d_3}$,
 we may choose
 $\tilde\zeta_1\geq\frac{c_2}{2d_3}$ to satisfy
$
c_2\tilde\zeta_1-d_3\tilde\zeta^2_1 -r_3(1-a_{31})= c_2 \sqrt{\frac{r_3(1-a_{32})}{d_3}} - 2r_3(1-a_{32}).
$

Now, we define ${\underline{\rho}_4}\in C([0,c_1])$ as follows.
\begin{itemize}
 \item [($\mathrm{i}$)] If $\tilde\zeta_1\leq \zeta_1$, then
\begin{equation*}
{\underline{\rho}_4}{(s)}:= 
\begin{cases}
\smallskip
\zeta_1s- d_3 \zeta_1^2 - r_3(1-a_{31}) &\text{ for }2d_3\zeta_1 < {s} \leq c_1,\\
\smallskip
\frac{s^2}{4d_3} - r_3(1-a_{31})&\text{ for }2d_3\tilde\zeta_1 <{s}\leq 2d_3\zeta_1,\\
\smallskip
\tilde\zeta_1s - d_3 \tilde\zeta_1^2 - r_3(1-a_{31}) &\text{ for }c_2 <{s} \leq 2d_3\tilde\zeta_1,\\
s\sqrt{\frac{r_3(1-a_{32})}{d_3}} - 2r_3(1-a_{32}) &\text{ for } \alpha_3\sqrt{1-a_{32}} < {s} \leq  c_2,\\
0 &\text{ for } 0\leq {s} \leq  \alpha_3\sqrt{1-a_{32}};
\end{cases}
\end{equation*}
  \item [($\mathrm{ii}$)] If $\tilde\zeta_1> \zeta_1$, then
  \begin{equation*}
{\underline{\rho}_4}{(s)} := 
\begin{cases}
\smallskip
\zeta_1s- d_3 \zeta_1^2 -r_3(1-a_{31})  &\text{ for } \min\{d_3(\zeta_1+ \tilde\zeta_1),c_1\} <{s} \leq  c_1,\\
\smallskip
\tilde\zeta_1s- d_3 \tilde\zeta_1^2 -r_3(1-a_{31})&\text{ for }c_2 < {s} \leq \min\{d_3(\zeta_1+ \tilde\zeta_1),c_1\},\\
s\sqrt{\frac{r_3(1-a_{32})}{d_3}} - 2r_3(1-a_{32})&\text{ for } \alpha_3\sqrt{1-a_{32}} < {s} \leq  c_2,\\
0 &\text{ for } 0\leq {s}  \leq \alpha_3\sqrt{1-a_{32}}.
\end{cases}
\end{equation*}
\end{itemize}

Let us show that ${\underline{\rho}_4}$ defined above is a viscosity sub-solution of \eqref{eq:hj}.
 Indeed, ${\underline{\rho}_4}$ is a classical solution for \eqref{eq:hj} whenever ${s} \not \in \{c_2,\alpha_3\sqrt{1-a_{32}} \}$ in  case \rm{(i)} or ${s} \not \in \{d_3(\zeta_1+\tilde\zeta_1),c_2,\alpha_3\sqrt{1-a_{32}} \}$ in case {\rm{(ii)}}.
In both cases, in some small neighborhood of $s= c_2$, ${\underline{\rho}_4}{(s)}$ can be rewritten by
$${\underline{\rho}_4}{(s)}=\max\left\{\tilde\zeta_1s - d_3 \tilde\zeta_1^2 - r_3(1-a_{31}) ,s\sqrt{\tfrac{r_3(1-a_{32})}{d_3}} - 2r_3(1-a_{32}) \right\}. $$
Observe that $\tilde\zeta_1s - d_3 \tilde\zeta_1^2 - r_3(1-a_{31})$ and $s\sqrt{\frac{r_3(1-a_{32})}{d_3}} - 2r_3(1-a_{32}) $ are both viscosity sub-solutions to \eqref{eq:hj}.  Thus ${\underline{\rho}_4}$ is a viscosity sub-solution of \eqref{eq:hj} in this region.

It remains to consider case {\rm{(ii)}} and assume    ${\underline{\rho}_4} - \phi$ attains its strict local maximum at $\hat{s} =d_3(\zeta_1+\tilde\zeta_1)$ for any test function $ \phi\in C^1(0,\infty)$,  and $\rho(\hat{s})>0$.
 In this  case, 
  we can check $\zeta_1\leq  \phi'(\hat{s}) \leq \tilde \zeta_1$, whence, at $\hat{s} =d_3(\zeta_1+\tilde\zeta_1)$,  we deduce
\begin{align*}
{\underline{\rho}_4}(\hat s) -c_2\phi'+ d_3 |\phi'|^2 +\mathcal{R}_*(\hat s)&={\underline{\rho}_4}(\hat s)-\hat s\phi' +d_3|\phi'|^2 + r_3(1 -a_{31})\\
&=\zeta_1c_2-d_3\zeta_1^2-d_3(\zeta_1+\tilde\zeta_1)\phi' +d_3 | \phi'|^2 \\
&\leq d_3(\phi'-\zeta_1)(\phi'-\tilde\zeta_1) \leq 0,
\end{align*}
where we used $\zeta_1, \tilde{\zeta}_1 \geq \frac{c_2}{2d_3}$ for the first inequality. Thus, ${\underline{\rho}_4}$ is a viscosity sub-solution of \eqref{eq:hj}.

In view of $\rho_{\rm{nlp}}(c_1)={\underline{\rho}_4} (c_1)$ and $\rho_{\rm{nlp}}(0)={\underline{\rho}_4}(0)$,  Lemma \ref{lemA2rho} says that
$\rho_{\rm{nlp}}{(s)}\geq {\underline{\rho}_4}{(s)} $ for $ s\in[0,c_1].$
Since $ \rho_{\rm nlp}(s) \geq {\underline{\rho}_4}(s)>0$ for $s >\alpha_3\sqrt{1-a_{32}}$, we have $s_{\rm{nlp}}\leq \alpha_3\sqrt{1-a_{32}}$. 
\end{proof}

We are now in  a position to prove Theorem C.

\begin{proof}[Proof of Theorem {\rm C}]
By   Propositions \ref{charactsnlpc}, \ref{prop:A}, and \ref{prop:B}, we conclude that $s_{{\rm nlp}}(c_1,c_2,\lambda)$ defined by \eqref{charact_snlp} can be expressed by
\begin{equation}\label{snlp00000}
s_{\rm{nlp}}(c_1,c_2,\lambda)= \begin{cases}
\medskip
 d_3 \lambda_{{\rm nlp1}} + \frac{r_3(1-a_{32})}{\lambda_{{\rm nlp1}}}  &\text{ for } \zeta_1> \frac{{c_2}}{2d_3} \text { and } \eqref{eq:propB1} \text { holds}, \\
 \medskip
  d_3 \lambda_{{\rm nlp2}} + \frac{r_3(1-a_{32})}{\lambda_{{\rm nlp2}}}  &\text{ for } \zeta_1\leq\frac{{c_2}}{2d_3} \text { and } \eqref{eq:propA1}\text { holds},\\
  \alpha_3\sqrt{1-a_{32}} & \text {otherwise}.
\end{cases}
\end{equation}
 Observe from Remark \ref{rmk:prop2} that \eqref{eq:propB1} is equivalent to \eqref{equivalent00},
and from Remark \ref{rmk:prop1} that \eqref{eq:propA1} is equivalent to \eqref{welldefineof} and $\lambda_{{\rm nlp2}}<\sqrt{\frac{r_3(1-a_{32})}{d_3}}$.
On the other hand, \eqref{welldefineof} turns out to be equivalent to 
 $a_{31}<a_{32}$, or $a_{31}\geq a_{32}$ and $\zeta_1+\zeta_2<\frac{c_2}{d_3}$,
where we used the definition of $\zeta_2$ in \eqref{lambda3}. Therefore, \eqref{snlp00000} is consistent with \eqref{eq:c_3nlp}, so that \eqref{eq:c_3nlp} in Theorem C follows. 
\end{proof}

\end{appendices}

\baselineskip 18pt
\renewcommand{\baselinestretch}{1.2}




\end{document}